\documentclass[preprint,12pt]{elsarticle}

\usepackage{hyperref}
\hypersetup{
    colorlinks,
    citecolor=black,
    filecolor=black,
    linkcolor=black,
    urlcolor=black
}

\usepackage{amssymb}
\usepackage{amsmath}
\usepackage{mathtools}
\usepackage{amsfonts}
\usepackage{amsthm}
\usepackage{float}
\usepackage{enumitem}
\usepackage{times}
\usepackage{tikz-cd}
\usetikzlibrary{shapes.geometric}
\usepackage{mathabx}
\usepackage{stmaryrd}
\usepackage{multicol}
\usetikzlibrary{arrows}

\floatstyle{boxed}
\newtheorem{theorem}{Theorem}[section]
\newtheorem{proposition}[theorem]{Proposition}
\newtheorem{lemma}[theorem]{Lemma}
\newtheorem{corollary}[theorem]{Corollary}
\theoremstyle{definition}
\newtheorem{example}[theorem]{Example}
\newtheorem{remark}[theorem]{Remark}

\newcommand{\va}{\mathsf} %fontshape for specific varieties 
\newcommand{\m}{\mathbf} %fontshape for math structures and algebras

\newcommand{\jn}{\mathbin{\vee}}
\newcommand{\mt}{\mathbin{\wedge}}
\newcommand{\ZZ}{\mathbb{Z}}
\newcommand{\NN}{\mathbb{N}}
\newcommand{\vr}{\mathbb{V}}

\newcommand{\NC}{\mathcal{N}}
\newcommand{\Nc}{\mathrm{N}}
\newcommand{\Grp}[1]{{#1_{g}}}
\newcommand{\grp}[1]{{#1_{g}}}
\newcommand{\cn}[1]{\boldsymbol{#1}}
\newcommand{\g}{\sigma}
\newcommand{\h}{\gamma}
\newcommand{\slt}{\delta}
\DeclareMathOperator{\lcm}{lcm}
\DeclareMathOperator{\supp}{supp}
\newcommand{\loc}[1]{{#1}_{\mathrm{loc}}}

\newcommand{\norm}[1]{\lVert #1 \rVert}
\DeclareMathOperator{\per}{per}

\DeclareMathOperator{\Cv}{C}
\DeclareMathOperator{\core}{\lambda}
\DeclareMathOperator{\hg}{h}

\newcommand{\f}{\varphi}
\newcommand{\ps}{\psi}

\journal{-}

\begin{document}

\begin{frontmatter}

\title{Axiomatizing small varieties of periodic \texorpdfstring{$\ell$}{ℓ}-pregroups}
\author[inst1]{Nikolaos Galatos}
\ead{ngalatos@du.edu}

\affiliation[inst1]{organization= {Department of Mathematics, University of Denver},
            addressline={2390 S. York St.}, 
            city={Denver},
            postcode={80208}, 
            state={CO},
            country={USA}}

\author[inst2]{Simon Santschi}
\ead{simon.santschi@unibe.ch}
\affiliation[inst2]{organization= {Mathematical Institute, University of Bern},
            addressline={Alpeneggstrasse 22}, 
            city={Bern},
            postcode={3012}, 
            country={Switzerland}}

\begin{abstract}
We provide an axiomatization for the variety generated by the $n$-periodic $\ell$-pregroup $\m F_n(\ZZ)$, for each $n \in \ZZ^+$, as well as for all possible 
joins of such varieties; the finite joins  form an ideal in the subvariety lattice of $\ell$-pregroups and we describe fully its lattice structure.
%joins of such varieties; the finite joins  form an ideal in the subvariety lattice of $\ell$-pregroups and we describe fully its lattice structure.

On the way, we characterize all finitely subdirectly irreducible (FSI) algebras in the variety generated by $\m F_n(\ZZ)$ as the $n$-periodic $\ell$-pregroups that have a totally ordered group skeleton (and are not trivial). The finitely generated FSIs that are not $\ell$-groups are further characterized as lexicographic products of a (finitely generated) totally ordered abelian $\ell$-group and $\m F_k(\ZZ)$, where $k \mid n$.
\end{abstract}
\end{frontmatter}

\section{Introduction}\label{s: intro}

A \emph{lattice-ordered pregroup} (\emph{$\ell$-pregroup}) is an algebra
$(A,\mt,\jn,\cdot,1,^{\ell} ,^{r})$,
such that $(A,\mt,\jn)$ is a lattice, $(A,\cdot,1)$ is a monoid, multiplication preserves the lattice order, and for all $a\in A$,
\[
a^{\ell} a\leq1\leq aa^{\ell} \text{ and }aa^{r}\leq1\leq a^{r}a.
\]
The $\ell$-pregroups that satisfy $x^\ell\approx x^r$ are precisely the \emph{lattice-ordered groups} (\emph{$\ell$-groups}); see \cite{AF, KM, Da} for an introduction to the theory of $\ell$-groups. \emph{Pregroups} are ordered algebraic structures defined similar to $\ell$-pregroups, but without the demand that the partial order forms a lattice; they were first introduced in the study of mathematical linguistics   and later studied extensively in the context of both theoretical and applied linguistics (see \cite{La, Bu, Ba} for example). Moreover,  $\ell$-pregroups are precisely the \emph{involutive residuated lattice}s, where multiplication coincides with its De Morgan dual, hence their study connects to the algebraic investigation of \emph{substructural logics}; see \cite{GJKO} for more on residuated lattices and substructural logics, and \cite{GG1, GG2} for more on the motivation and importance of $\ell$-pregroups.

A prominent and simple example of an $\ell$-pregroup is $\m{F}(\ZZ)$ which is based on the set of all order-preserving finite-to-one functions on the chain $\ZZ$  under functional composition, pointwise order, and suitably (and uniquely) defined $^\ell$ and $^r$ operations; this $\ell$-pregroup is actually \emph{distributive} (it has a distributive lattice reduct). In \cite{GG1} it is shown that $\m{F}(\ZZ)$ generates the whole variety $\mathsf{DLP}$ of distributive $\ell$-pregroups. 
 
An $\ell$-pregroup is called \emph{$n$-periodic}, for a given $n \in \ZZ^+$, if it satisfies the equation $x^{\ell^n} \approx x^{r^n}$ and the corresponding variety is denoted by $\mathsf{LP}_n$; as mentioned above, $\mathsf{LP}_1$ ends up being the variety of $\ell$-groups. An $\ell$-pregroup is called \emph{periodic}, if it is $n$-periodic for some $n \in \ZZ^+$. In \cite{GJ} it is shown that every periodic $\ell$-pregroup is distributive and in \cite{GG2} it is proved that $\mathsf{DLP}$ is equal to the join of the varieties  $\mathsf{LP}_n$. 

Also, for every $\ell$-pregroup, its elements $a$ that satisfy $a^{\ell^n}=a^{r^n}$ (the $n$-periodic elements) form a subalgebra; the subalgebra of $1$-periodic elements is an $\ell$-group and we call it the \emph{group skeleton} of the algebra. The subalgebra of $n$-periodic elements of $\m{F}(\ZZ)$ is denoted by $\m{F}_n(\ZZ)$ and it ends up consisting of the functions on $\ZZ$ that are $n$-periodic as functions (thus justifying the terminology). 
It is observed in \cite{GG2} that $\m{F}_n(\ZZ)$ is a simple algebra, for all $n$. 

In \cite{GG2} it is shown that $\mathsf{DLP}$ is also equal to the join of the varieties $\vr(\m{F}_n(\ZZ))$, for $n \in \ZZ^+$, showcasing the important role of these varieties in understanding distributive $\ell$-pregroups. In contrast to the fact that $\m{F}(\ZZ)$ generates $\mathsf{DLP}$, $\m{F}_n(\ZZ)$ generates a proper subvariety $\vr(\m{F}_n(\ZZ))$ of $\mathsf{LP}_n$, for every $n$. (In that respect the two approximations $\mathsf{DLP}= \bigvee \mathsf{LP}_n$ and $\mathsf{DLP}= \bigvee \vr(\m{F}_n(\ZZ))$ of $\mathsf{DLP}$  are quite different.)
  $\mathsf{LP}_n$ is axiomatized by the $n$-periodicity equation, but no axiomatization of $\vr(\m{F}_n(\ZZ))$ is known. As shown in \cite{GG2}, $\vr(\m{F}_n(\ZZ))$ is decidable for every $n$, so the variety is fairly well behaved, but this does not shed any light on its axiomatization.  Our main goal is to provide such an axiomatization and we will show that this can be achieved by one additional equation. 

 Moreover,  the algebras $\m{F}_n(\ZZ)$, and the varieties they generate, are not only important in the study of the variety of distributive $\ell$-pregroups, as discussed above, but also in the study of each and every $n$-periodic variety  $\mathsf{LP}_n$. This is because, as shown in \cite{GG2}, every $n$-periodic $\ell$-pregroup embeds into a wreath product of an $\ell$-group and $\m{F}_n(\ZZ)$.

\medskip

We prove that the non-trivial subvarieties of $\vr(\m{F}_n(\ZZ))$ are all joins of varieties of the form $\vr(\m{F}_k(\ZZ))$, where $k\mid n$, and that they form a lattice isomorphic to the downset lattice of the divisibility lattice of $n$;  
moreover, the finite joins of varieties of the form $\vr(\m{F}_k(\ZZ))$, where $k \in \ZZ^+$, constitute an ideal in the lattice of $\ell$-pregroup varieties (infinite joins always produce $\mathsf{DLP}$). In that sense they are `small' varieties of periodic $\ell$-pregroups and they occupy a prominent position in this lattice. Also, this result shows that, for each $n$, the subvarieties that are properly $n$-periodic (not periodic for a smaller period) form an interval whose top is $\mathsf{LP}_n$ and whose bottom is the join 
$\bigvee \vr(\m{F}_{p^k}(\ZZ))$, where $p^k$ ranges over all of the prime-power divisors of $n$ with $k$ maximal for each prime $p$;
if $n$ is a prime power this interval is $[\vr(\m{F}_n(\ZZ)), \mathsf{LP}_n]$. 
 Since, for all $n$, the variety $\vr(\m{F}_n(\ZZ))$ is very low relative to $\mathsf{LP}_n$, it is expected that its axiomatizing identity must be quite strong; it is, thus, surprising that this can be achieved by a single simple equation. The equation itself states that: \emph{$n$th powers} of group elements are central.

In more detail, the main result of the paper is that  $\vr(\m{F}_n(\ZZ))$ is axiomatized, relative to $\mathsf{LP}_n$, by the equation $x\g(y)^n \approx \g(y)^n x$, where $\g(y)$ denotes the term $y\mt y^{\ell \ell} \mt \dots \mt y^{\ell^{2(n-1)}}$. In the context of $\mathsf{LP}_n$, the elements of the form $\g(a)$ range precisely over the group/invertible elements of the algebra; we note that in the context of $\mathsf{DLP}$,  there is no term that captures the group elements of the algebra.

It is easy to see that $\m{F}_n(\ZZ)$ is $n$-periodic and satisfies the equation.
Before discussing our approach for obtaining the converse direction of the axiomatization result we should mention that the operators $\m{F}$ and $\m{F}_n$, for all $n$, can be applied to an arbitrary chain $\m \Omega$ and they define the (distributive) $\ell$-pregroup of all maps on  $\m \Omega$ that have \emph{residuals and dual residuals of all orders}; in particular, $\m{F}_1$ is the same as the operator $\m {Aut}$ of all automorphisms (order-preserving bijections) on a chain.  Cayley's representation theorem for groups extends to Holland's embedding theorem for $\ell$-groups  (every $\ell$-group embeds into  $\m {Aut}(\m \Omega)$ for some chain $\m \Omega$ \cite{Ho-em}), to $\mathsf{DLP}$ (every distributive $\ell$-pregroup embeds into  $\m {F}(\m \Omega)$ for some chain $\m \Omega$  \cite{GH}, and $\m \Omega$ can be taken to be a lexicographic product of $\ZZ$ \cite{GG1}) and to $\mathsf{LP}_n$ for every $n$ (every $n$-periodic $\ell$-pregroup embeds into $\m {F}_n(\m \Omega)$ for some chain $\m \Omega$, which can be taken to be a lexicographic product of $\ZZ$ \cite{GG2}). Also, as mentioned above,  every $n$-periodic $\ell$-pregroup embeds into a wreath product of an $\ell$-group and $\m{F}_n(\ZZ)$. The last two results demonstrate that arbitrary $n$-periodic $\ell$-pregroups have a quite intricate and complex structure. Our proof of the main theorem requires a deep understanding of the structure of $n$-periodic $\ell$-pregroups represented as subalgebras of wreath products and this leads us to discovering key structural results of independent interest (about lexicographic products and about $\ell$-pregroups with totally ordered group skeletons); these include the following:
\begin{enumerate}
    \item Among $n$-periodic $\ell$-pregroups that satisfy the equation $x\g(y)^n \approx \g(y)^n x$
the finitely subdirectly irreducibles are exactly the ones whose group skeleton is totally ordered; in this case the group is actually abelian. (Corollary~\ref{c:fsi-commutative}) 
\item  All finitely generated $n$-periodic $\ell$-pregroups with a totally ordered abelian skeleton are lexicographic products of a totally ordered abelian $\ell$-group and an algebra in $\vr(\m{F}_n(\ZZ))$; this algebra can actually be taken to be $\m{F}_k(\ZZ)$ for some  $k \mid n$. (Theorem~\ref{t:lexprod-fsi-Fn})
\item $\vr(\m{F}_n(\ZZ))$ is closed under lexicographic products by totally ordered abelian $\ell$-groups. (Proposition~\ref{p:Fn-closed-lexprod})
\end{enumerate}
To establish the converse (of the main axiomatization theorem), i.e., that every $n$-periodic $\ell$-pregroup satisfying $x\g(y)^n \approx \g(y)^n x$ is in $\vr(\m{F}_n(\ZZ))$, it suffices to focus on  finitely generated FSI algebras. 
If  $\m L$ is a finitely generated FSI $n$-periodic $\ell$-pregroup satisfying the equation $x\g(y)^n \approx \g(y)^n x$, then by Corollary~\ref{c:fsi-commutative} we get that $\m L$ has a totally ordered abelian group skeleton. 
By Theorem~\ref{t:lexprod-fsi-Fn},  $\m L$ is a lexicographic product of a totally ordered abelian $\ell$-group and an algebra in $\vr(\m{F}_n(\ZZ))$.
Finally, by Proposition~\ref{p:Fn-closed-lexprod},  $\m L$ is in $\vr(\m{F}_n(\ZZ))$, as desired.

To prove these three key facts, we also need to study deeply the structure of the algebra $\m{F}_n(\ZZ)$ itself, define and understand lexicographic products, and work with wreath products. 

In particular, in Section~\ref{s: Sk CNS} we show that the group skeleton of an $n$-periodic $\ell$-pregroup can be realized as the image of a closure operator and also as the image of an interior operator ({the term function} $\sigma$), both of which are term-definable. Then we define a norm on an $n$-periodic $\ell$-pregroup that generalizes the absolute value of $\ell$-groups and use it to characterize the convex subalgebra closure process and to show that the lattice of convex subalgebras of an  $n$-periodic $\ell$-pregroup is isomorphic to the one of its group skeleton. Next, we use the characterization of congruences of $n$-periodic $\ell$-pregroups as convex normal subalgebras and show how the equation $x\g(y)^n \approx \g(y)^n x$ simplifies conjugation (needed for normality) and allows us to show that such an algebra and its group skeleton also have isomorphic lattices of convex normal subalgebras (hence congruences). As a result, the FSIs that satisfy $x\g(y)^n \approx \g(y)^n x$ are precisely the non-trivial algebras whose group skeleton is a totally ordered (abelian) $\ell$-group; this is the first item in our list above.

In Section~\ref{s: FnZ}, we show that $\m{F}_n(\ZZ)$ is generated by any of its elements that has periodicity $n$. This is done in successive steps, where  we consider less and less special elements of periodicity $n$, starting with the case of a minimal strictly positive idempotent, then a strictly positive idempotent and finally any such element. The proofs require a good understanding of the idempotents and of the flat elements (between two idempotents), as well as the generation process inside $\m{F}_n(\ZZ)$. We conclude the section by showing that elements in $\m{F}(\ZZ)$ generate an element of periodicity the least common multiple of their periodicities.

In Section~\ref{s: lex} we extend the notion of a lexicographic product of a totally ordered $\ell$-group with an $\ell$-group to the case where the second factor is an arbitrary $\ell$-pregroup, and show that $n$-periodicity is preserved by this construction. More importantly, we establish Proposition~\ref{p:Fn-closed-lexprod}, which is the third item in the list above. 

Section~\ref{s: n-per l-pr} starts by reviewing the representation of any $n$-periodic $\ell$-pregroup as a wreath product of an $\ell$-group with $\m{F}_n(\ZZ)$. We investigate  how to characterize internally, within such a representing wreath product, attributes of an element, such as periodicity, idempotency, and flatness. We also define, relative to a representation of an $n$-periodic $\ell$-pregroup $\m L$, its local subalgebra $\loc{\m L}$, consisting of the elements whose global component in the wreath product is the identity,  and prove that it is a convex normal subalgebra of $\m L$ and that $\m{L}/\loc{\m L} \cong \Grp{\m L}/\loc{(\Grp{\m L})}$, where $\Grp{\m L}$ is the group skeleton of $\m L$.

Section~\ref{s: axiomatization} starts with a deep study of $\ell$-pregroups of periodicity $n$ that have a totally ordered group skeleton. We prove that this skeleton is discretely ordered  and  its local subalgebra is isomorphic to $\ZZ$, when the $\ell$-pregroup is not an $\ell$-group; in particular, if the skeleton is further abelian then it decomposes as a lexicographic product, where the inner factor is the $\ell$-group $\ZZ$. 
This leads to the key result that  if a proper $n$-periodic $\ell$-pregroup has a totally ordered group skeleton, then its local subalgebra is isomorphic to $\m{F}_k(\ZZ)$, for some $k \mid n$ (Theorem~\ref{t:Fn-FSI-local}), and allows us to further establish the second item of our list above, yielding a decomposition of each finitely generated such algebra that is not an $\ell$-group and has an abelian group skeleton as a lexicographic product of an $\ell$-group and  $\m{F}_k(\ZZ)$, for some $k \mid n$.

The previous work allows us to prove, in the second part of Section~\ref{s: axiomatization}, that $\vr(\m{F}_n(\ZZ))$ is axiomatized by $x\g(y)^n \approx \g(y)^n x$ among $n$-periodic $\ell$-pregroups and we also obtain finite (actually  one-equation) axiomatizations for every join of such varieties. We prove that such joins (excluding the variety $\mathsf{DLP}$ itself) form an ideal of the subvariety lattice of $\mathsf{DLP}$ and that it is isomorphic to the finite-downset lattice of the divisibility lattice of the natural numbers.
Finally, as corollaries of our results, we get that the  FSIs in $\vr(\m{F}_n(\ZZ))$  are the $n$-periodic $\ell$-pregroups that have a totally ordered abelian-group skeleton (and are non-trivial). The finitely generated FSIs that are not $\ell$-groups are exactly the lexicographic products of a finitely generated totally ordered abelian $\ell$-group and $\m F_k(\ZZ)$, where $k \mid n$.

\section{The group skeleton and convex (normal) subalgebras}\label{s: Sk CNS}

In this section we define the group skeleton of an $\ell$-pregroup, a subalgebra that is an $\ell$-group and consists of the invertible elements, and the notion of a norm of an element that allows for a convenient description of  convex subalgebras of an $n$-periodic $\ell$-pregroup. Also, we show that every $n$-periodic $\ell$-pregroup and its group skeleton have isomorphic lattices of convex subalgebras. If the algebras satisfy the equation $x\g(y)^n \approx \g(y)^n x$, then the same is true for the lattices of convex normal subalgebras, hence also of their congruence lattices. This will yield the crucial characterization of the finitely subdirectly irreducibles, which will be useful in later sections.

Recall from Section~\ref{s: intro} that a \emph{lattice-ordered pregroup} (\emph{$\ell$-pregroup}) is an algebra
$(A,\mt,\jn,\cdot,1,^{\ell} ,^{r})$,
such that $(A,\mt,\jn)$ is a lattice, $(A,\cdot,1)$ is a monoid, multiplication preserves the lattice order, and for all $a\in A$,
\[
a^{\ell} a\leq1\leq aa^{\ell} \text{ and }aa^{r}\leq1\leq a^{r}a.
\]
An \emph{involutive residuated lattice} is an algebra
$(A,\mt,\jn,\cdot,1,^{\ell} ,^{r})$,
where $(A,\mt,\jn)$ is a lattice, $(A,\cdot,1)$ is a monoid, and for all $a, b, c\in A$,
\[
a^{\ell r}=a=a^{r \ell} \text{ and } ab \leq c \Leftrightarrow b \leq a^r + c \Leftrightarrow a \leq c+ b^\ell,
\]
where $a+b:=(b^r a^r)^\ell$; in this case we also have $a+b=(b^\ell a^\ell)^r$. The next lemma shows that $\ell$-pregroups are special involutive residuated lattices and provides some useful properties that we will use later.

\begin{lemma}[\cite{GJ, GJKO}] \label{l: InRL}
$\ell$-pregroups are exactly the involutive residuated lattices that satisfy $x+y \approx xy$. Hence, they also satisfy $(xy)^r \approx y^rx^r$, $(xy)^\ell \approx y^\ell x^\ell$, and for all $a,b,c$
  \[
  ab \leq c \Leftrightarrow b \leq a^r  c \Leftrightarrow a \leq c b^\ell;
  \]
the latter is called \emph{residuation}.
 Also, multiplication distributes over both joins and meets, and the De Morgan laws hold: 
 $(x\jn y)^r \approx x^r \mt y^r$, $(x\mt y)^r \approx x^r \jn y^r$, 
 $(x\jn y)^\ell \approx x^\ell \mt y^\ell$, $(x\mt y)^\ell \approx x^\ell \jn y^\ell$.
\end{lemma}

\begin{remark}\label{r: xxl}
 In later sections we will see that positive idempotent elements play an important role in our results. 
    Using Lemma~\ref{l: InRL} it is straightforward to check that for an $\ell$-pregroup $\m{L}$ and $a\in L$, $aa^\ell$ is  idempotent and $1 \leq aa^\ell$. Moreover, if $a$ is already an idempotent with $1\leq a$, then $aa^\ell =a$. Indeed, if $1\leq a$ is an idempotent, then we have $aa \leq a$, so, by residuation, $a \leq aa^\ell$; conversely $1\leq a$ implies $a^\ell \leq 1$; hence $aa^\ell \leq a$. 
\end{remark}

\subsection{The group skeleton}

As mentioned in the introduction, for an $\ell$-pregroup $\m{L}$ we define the \emph{group skeleton of} $\m{L}$ to be the set $\grp{L} := \{a \in L : a^\ell=a^r \}$. The following lemma justifies the terminology and shows that $\grp{L}$ supports a subalgebra $\Grp{\m{L}}$.

\begin{lemma} [Lemma 2.6(1) of \cite{GG2}]
   For every $\ell$-pregroup $\m{L}$, $\grp{L}$ gives rise to a subalgebra $\Grp{\m{L}} = (\grp{L}, \mt, \jn, \cdot, 1  , {}^\ell, {}^r)$ of $\m{L}$; moreover,  $\Grp{\m{L}}$ is an $\ell$-group and consists precisely of the invertible elements of $\m L$.
\end{lemma}

For $k\in \NN$ we will write $x^{[k]}$ for $x^{\ell^{2k}}$ and $x^{[-k]}$ for $x^{r^{2k}}$. Note that the $n$-periodicity equation $x^{\ell^n} \approx x^{r^n}$ is equivalent to $x^{\ell^{2n}} \approx x$, i.e., to  $x^{[n]} \approx x$. 

\begin{lemma}\label{l:double-l-auto}
For every  $\ell$-pregroup $\m L$ and $k \in \ZZ$, the map $a \mapsto a^{[k]}$ is an automorphism of  $\m L$ that fixes the group skeleton of $\m{L}$.
\end{lemma}

\begin{proof}
The argument that the map $a \mapsto a^{\ell \ell}$ is an automorphism is given in the proof of Lemma 2.6 of \cite{GG2}, so its inverse $a\mapsto a^{rr}$, and more generally its $k$th power, i.e., the map $a \mapsto a^{[k]}$, is also an automorphism. That it fixes the group skeleton of $\m L$ follows from the fact that an element $a \in L$ is  invertible iff $a^\ell=a^r(=a^{-1})$, so for invertible elements we have $a^{\ell\ell}=a^{r \ell}=a$. 
\end{proof}

For each $n\in \ZZ^+$ we define the terms 
\begin{center}
 $\g_n(x) := x\mt x^{[1]} \mt \dots \mt x^{[n-1]}$ and $\h_n(x) := x\jn x^{[1]} \jn \dots \jn x^{[n-1]}$.    
\end{center}
When we consider $n$-periodic $\ell$-pregroups for a fixed $n$,  we will drop the subscript and write $\g(x)$ and $\h(x)$ for $\g_n(x)$ and $\h_n(x)$, respectively.  As noted in \cite{GJ}, the images of the maps $\g$ and $\h$ on an $n$-periodic $\ell$-pregroup $\m L$ are both equal to the group skeleton $\grp{L}$ of $\m L$.

\begin{remark}
       All of the results in this section still hold if the assumption of $n$-periodicity is replaced by the  equation  $\g_n(x)^\ell \approx \g_n(x)^r$, which is a properly weaker equation (for $n>2$), as we explain below. 
   First note that in a $k$-periodic $\ell$-pregroup, with $k\leq n$, for every element $a$ and $m\in \NN$, we obtain $a^{[m]} = a^{[l]}$, where $l$ is the remainder of dividing $m$ by $k$ (so $0\leq l \leq k-1$), by iteratively applying the map $x \mapsto x^{\ell\ell}$ to $k$-periodicity. Since $k \leq n$, we have
    \[
    \g_n(a) = a \mt a^{[1]} \mt \dots \mt a^{[n-1]} = a \mt a^{[1]} \mt \dots a^{[k-1]} = \g_k(a),
    \]
    and since $\g_k(a)$ is invertible under the assumption of $k$-periodicity, we obtain $\g_n(a)^\ell = \g_k(a)^\ell = \g_k(a)^r = \g_n(a)^r$. Therefore, the variety $\bigvee_{k=1}^n \mathsf{LP}_k$ is contained in the variety axiomatized by $\g_n(x)^\ell \approx \g_n(x)^r$; an interesting question is whether we have equality.
     Now, for $n>2$, there exists $k\leq n$ with $k \nmid n$, thus $k$-periodicity
     implies $\g_n(x)^\ell \approx \g_n(x)^r$ but not $n$-periodicity (as the algebra $\m F_k(\ZZ)$ illustrates); as a result, for $n>2$, $\g_n(x)^\ell \approx \g_n(x)^r$ is strictly weaker than $n$-periodicity. 
     Even though all of our results in this section hold if we replace $n$-periodicity by  $\g_n(x)^\ell \approx \g_n(x)^r$, we opt to state them under the assumption of $n$-periodicity for simplicity, given that we do not have an example of a non-periodic $\ell$-pregroup satisfying $\g_n(x)^\ell \approx \g_n(x)^r$.
    \end{remark}

Recall that an interior operator $f$ on an $\ell$-pregroup $\m L$ (more generally on a residuated lattice) that satisfies $f(a)f(b) \leq f(ab)$ is called a \emph{conucleus}.
Moreover, a closure operator is called \emph{topological} if it is a join homomorphism and an interior operator is called \emph{topological} if it is a meet homomorphism.

\begin{lemma}\label{l: sigmagamma}
    Let $\m{L}$ be an $n$-periodic $\ell$-pregroup.
    \begin{enumerate}[label = \textup{(\arabic*)}]
        \item The map $\g\colon L \to L$, $a \mapsto \g(a)$ is a conucleus on $\m L$ that preserves meets (hence a topological interior operator). In particular, $\Grp{\m{L}}$ is the conuclear image of $\m L$ under $\g$ and for each $a\in L$, $\g(a) = \max \{b \in \grp{L} : b \leq a \}$.
        \item The map $\h\colon L \to L$, $a \mapsto \h(a)$ is a closure operator that preserves joins (hence a topological closure operator) and satisfies $\h(ab) \leq \h(a)\h(b)$.
        In particular, $\Grp{\m{L}}$ is the closure image of $\m L$ under $\h$ and for each $a\in L$, $\h(a) = \min \{b \in \grp{L} : a \leq b \}$.
    \end{enumerate}
\end{lemma}

\begin{proof}
    (1) First note that it follows from Lemma~\ref{l:double-l-auto}, that $\g$ is order-preserving. Moreover,  if $a$ is invertible, then $a^{[k]}=a$, for all $k \in \ZZ$, so $\g(a)=a$; since the image of $\sigma$ is precisely $\grp{L}$, $\g$ is idempotent. Also, by definition, $\g(a) \leq a$; hence $\g$ is a  interior operator. Moreover, for $a,b\in L$, we have 
    \begin{align*}
        \g(a)\g(b) &= (a\mt a^{[1]} \mt \dots \mt a^{[n-1]})(b\mt b^{[1]} \mt \dots \mt b^{[n-1]}) \\
        &= \bigwedge_{i,j \in [0,n-1]} a^{[i]}b^{[j]} \\
        &\leq  \bigwedge_{i \in [0,n-1]} a^{[i]}b^{[i]}  \\
        &= \bigwedge_{i \in [0,n-1]} (ab)^{[i]} = \g(ab). 
    \end{align*}
    where we used Lemma~\ref{l:double-l-auto} for the  third equality. Thus $\g$ is a conucleus with image $\Grp{\m{L}}$. In particular, we get that for each $a\in L$, $\g(a) = \max \{b \in \grp{L} : b \leq a \}$. Finally we have for all $a,b\in L$, 
    \[
    \g(a\mt b) = \bigwedge_{i=0}^{n-1} (a \mt b)^{[i]} = \bigwedge_{i=0}^{n-1} a^{[i]} \mt b^{[i]} = \Big(\bigwedge_{i=0}^{n-1} a^{[i]} \Big) \mt \Big( \bigwedge_{i=0}^{n-1} \mt b^{[i]} \Big) =  \g(a)\mt \g(b),
    \]
    where we use Lemma~\ref{l:double-l-auto} for the second equality, so $\g$ is topological.
    
    The proof of (2) is similar to the one for (1).
\end{proof}

Note that the map $\h$ is a conucleus on the dual $\ell$-pregroup (the multiplication and the identity stay the same, but the order and the inverses are dualized).

\begin{remark}\label{r:grp-conucleus}
    Let $\m{L}$ be an $n$-periodic $\ell$-pregroup, $a\in L$, and $g\in \grp{L}$. Then 
    \begin{enumerate}[label = \textup{(\arabic*)}]
        \item $\g(ga) = g\g(a)$ and $\g(ag) = \g(a)g$.
        \item $\h(ga) = g\h(a)$ and $\h(ag) = \h(a)g$.
    \end{enumerate}
\end{remark}
\begin{proof}
    Immediate by the definition of $\g$ and $\h$ and Lemma~\ref{l:double-l-auto}, since $g^{[k]} = g$ for each $g\in \grp{L}$ and $k\in \ZZ$.
\end{proof}

\begin{lemma}\label{l:grp-ell-r}
    If $\m{L}$ is an $n$-periodic $\ell$-pregroup and $a\in L$, then 
    $\g(a)^{-1} = \h(a^{\ell}) = \h(a^r)$ and $\h(a) = \g(a^\ell)^{-1} = \g(a^r)^{-1}$.
\end{lemma}

\begin{proof}
Note that, by Lemma~\ref{l:double-l-auto} and Lemma~\ref{l: InRL},
    \begin{align*}
    \g(a)^{-1} = \g(a)^\ell &= (a \mt a^{[1]} \mt \dots \mt a^{[n]})^\ell \\
    &= a^\ell \jn (a^{[1]})^\ell \jn \dots \jn (a^{[n]})^\ell \\
    &= a^\ell \jn (a^\ell)^{[1]} \jn \dots \jn (a^\ell)^{[n]} \\
    &= \h(a^\ell).
    \end{align*}
    The proof for the other equalities is very similar.
\end{proof}

\subsection{Convex subalgebras}

If $\m{L}$ is an $n$-periodic $\ell$-pregroup and $a\in L$, we define $\norm{a} = \g(a)^{-1} \jn \h(a)$. Note that if $a$ is invertible, then $\sigma(a)=\gamma(a) =a$, so $\norm{a}=a \jn a^{-1}$; i.e., $\norm{a}$ coincides with the usual absolute value as defined in $\ell$-groups. We establish some basic properties of the norm of an element that will be useful in describing convex subalgebras.

\begin{lemma}\label{l:norm}
     Let $\m L$ be an $n$-periodic $\ell$-pregroup. 
     \begin{enumerate}[label = \textup{(\arabic*)}]
         \item For $a \in L$, $\norm{a}^{-1} \leq a \leq \norm{a}$ and $1 \leq \norm{a}$.
         \item For $a \in L$, $\norm{a} = \norm{a^\ell} = \norm{a^r}$.
         \item For $a,b\in L$, $\norm{a \mt b} \leq \norm{a}\jn \norm{b}$.
         \item For $a,b\in L$, $\norm{a \jn b} \leq \norm{a}\jn \norm{b}$.
          \item For $a,b \in L$, $\norm{ab} \leq \norm{a}\norm{b}\norm{a}$.
     \end{enumerate}
\end{lemma}

\begin{proof}
    (1) We have 
    \[
    \norm{a}^{-1} = \g(a) \mt \h(a)^{-1} \leq \g(a) \leq a \leq \h(a) \leq \g(a)^{-1} \jn \h(a) \leq \norm{a}.
    \]
    Also,  $\g(a) \leq \h(a)$, so $1 \leq \h(a)^{-1} \jn \h(a) \leq \g(a)^{-1} \jn \h(a) = \norm{a}$.

    (2) By Lemma~\ref{l:grp-ell-r}, we have 
    \[
    \norm{a} = \g(a)^{-1} \jn \h(a) = \h(a^\ell) \jn \g(a^\ell)^{-1} = \norm{a^\ell},
    \]
    and similarly $\norm{a} = \norm{a^r}$.

    (3) Using Lemma~\ref{l: sigmagamma}(1), we have 
    \begin{align*}
        \norm{a\mt b} &= \g(a\mt b)^{-1} \jn \h(a \mt b) \\
                      &= (\g(a)\mt \g(b))^{-1} \jn \h(a \mt b)\\
                      &= \g(a)^{-1} \jn \g(b)^{-1} \jn \h(a\mt b) \\
                      &\leq \g(a)^{-1} \jn \g(b)^{-1} \jn \h(a) \jn \h(b) \\
                      &= \norm{a}\jn \norm{b}.
    \end{align*}    
    (4) Since $\g(a),\g(b) \leq \g(a\jn b)$,  we get $\g(a\jn b)^{-1} \leq \g(a)^{-1} \jn \g(b)^{-1}$. Thus, using Lemma~\ref{l: sigmagamma}(2), we have
    \begin{align*}
    \norm{a\jn b} &= \g(a\jn b)^{-1} \jn \h(a \jn b) \\
                      &= \g(a\jn b)^{-1} \jn \h(a) \jn \h(b)\\
                      &\leq \g(a)^{-1} \jn \g(b)^{-1} \jn \h(a) \jn \h(b) \\
                      &= \norm{a}\jn \norm{b}.
    \end{align*}    
     (5) First note that, since $\g(a)\g(b) \leq \g(ab)$, we have $\g(ab)^{-1} \leq (\g(a)\g(b))^{-1} = \g(b)^{-1}\g(a)^{-1}$. Thus, by (1), we get 
    \begin{align*}
    \norm{ab} &= \g(ab)^{-1} \jn \h(ab) \\
              &\leq \g(b)^{-1}\g(a)^{-1} \jn \h(a)\h(b) \\
              &\leq \norm{a}\g(b)^{-1}\norm{a} \jn \norm{a}\h(b)\norm{a} \\
              &= \norm{a}\norm{b}\norm{a}. \qedhere
    \end{align*}
\end{proof}

For $S\subseteq L$ we denote by $\Cv_{\m L}(S)$ the smallest convex subuniverse of $\m L$ that contains $S$ and we write $\Cv_{\m L}(a)$ for $\Cv_{\m L}(\{a\})$; if the algebra is clear from the context, we will omit the subscript.

\begin{lemma}\label{l:conv-subalg-gen}
    Let $\m L$ be an $n$-periodic $\ell$-pregroup and $S\subseteq L$. 
    Then 
    \[
    \Cv(S) = \{a \in L : \norm{a} \leq \norm{s_1}\cdots \norm{s_k} \text{ for some } s_1,\dots, s_k \in S\}.
    \]
\end{lemma}

\begin{proof}
    Let $H = \{a \in L : \norm{a} \leq \norm{s_1}\cdots \norm{s_k} \text{ for some } s_1,\dots, s_k \in S\}$. Note that
    if $a \in H$, then 
    there exist $s_1,\dots, s_k\in S$ with $1 \leq \norm{a} \leq \norm{s_1}\cdots \norm{s_k}$, so since $\norm{s_1}\cdots\norm{s_k} \in  \Cv(S)$ and since $\Cv(S)$ is convex  we get $\norm{a} \in \Cv(S)$. Moreover, $\norm{a}^{-1} \leq a \leq \norm{a}$, by Lemma~\ref{l:norm}(1), so $a \in \Cv(S)$; thus  $H\subseteq \Cv(S)$.
    
    For the reverse inclusion it is enough to show that $H$ is a convex subuniverse, since $S\subseteq H$. 
    If $c,d \in H$, then there exist $s_1,\dots, s_k,t_1,\dots, t_m \in S$ such that  $\norm{c} \leq s$ and $\norm{d} \leq t$, where $s = \norm{s_1}\cdots \norm{s_k}$ and $t = \norm{t_1}\cdots \norm{t_m}$. So, by Lemma~\ref{l:norm}  we get 
    \[
    \norm{c^\ell} = \norm{c^r} =\norm{c} \leq s, 
    \]
    as well as
    \[
    \norm{cd} \leq \norm{c}\norm{d}\norm{c} \leq sts
    \]
    and 
    \[
    \norm{c \mt d},\norm{c \jn d} \leq \norm{c} \jn \norm{d} \leq st.
    \]
    Thus $c^\ell,c^r,cd,c\mt d, c\jn d\in H$. Finally suppose that $b \in L$ with $c\leq b\leq d$. Then $\h(c) \leq \h(b) \leq \h(d)$ and $\g(c) \leq \g(b) \leq \g(d)$, hence also $\g(d)^{-1} \leq \g(b)^{-1} \leq \g(c)^{-1}$. Therefore,
    \[
    \norm{b} = \g(b)^{-1} \jn \h(b) \leq \g(c)^{-1} \jn \h(d) \leq \norm{c} \jn \norm{d} \leq st.
    \]
    Thus $H$ is a convex subuniverse of $\m L$.
\end{proof}

\begin{corollary}
    If $\m L$ is an $n$-periodic $\ell$-pregroup and $a\in L$, then 
    \[
    \Cv(a) = \{b \in L : \norm{b} \leq \norm{a}^{k} \text{ for some } k\in \NN \}.
    \]
\end{corollary}

\begin{corollary}\label{c:norm-convsub}
    Let $\m L$ be an $n$-periodic $\ell$-pregroup, $H$ 
    a convex subuniverse of $\m L$, and $a\in L$. Then $a\in H$ if and only if $\norm{a} \in H$ if and only if $\g(a),\h(a) \in H$.
\end{corollary}

 We can now show that the convex subalgebras of a  periodic $\ell$-pregroup correspond to the ones of its skeleton. 
We denote by $\mathcal{C}(\m{L})$ the lattice  of convex subalgebras of a periodic $\ell$-pregroup $\m{L}$. 

\begin{theorem}\label{t:conv-sub-iso}
    For every periodic $\ell$-pregroup $\m{L}$  the lattices $\mathcal{C}(\m{L})$ and $\mathcal{C}(\Grp{\m L})$ are isomorphic under the inverse maps $H \mapsto H\cap \grp{L}$ and $K \mapsto \overline{K}$, where $\overline{K}$ is the convex closure of $K$ in $\m{L}$.
\end{theorem}

\begin{proof}
 We set $\f\colon \mathcal{C}(\m{L}) \to  \mathcal{C}(\Grp{\m L})$, $\f(H) = H\cap \grp{L}$ and  $\psi \colon \mathcal{C}(\Grp{\m L}) \to \mathcal{C}(\m{L})$, $\psi(K) = \overline{K}$. 
    First note that $\f$ is well-defined and order-preserving. To see that $\psi$ is well-defined, let $K\in  \mathcal{C}(\Grp{\m L})$; we show that $\overline{K} = \Cv_{\m L}(K)$. Clearly $\overline{K} \subseteq \Cv_{\m L}(K)$. For the converse inclusion let  $a \in \Cv_{\m L}(K)$. By Lemma~\ref{l:conv-subalg-gen}, there exist $s\in K$ such that $1 \leq \norm{a} \leq s$, so $\norm{a} \in K$; since $\norm{a}^{-1} \leq a \leq \norm{a}$, we get $a \in \overline{K}$.
    Therefore $\psi$ is well-defined and clearly order-preserving.
    
    It remains to show that the maps are inverses of each other, so let $H\in \mathcal{C}(\m{L})$. Then for each $a \in L$,  $a \in \psi(\f(H))$ iff $\norm{a} \in \f(H) = H \cap \grp{L}$ iff $\norm{a}\in H$ iff $a\in H$, by Corollary~\ref{c:norm-convsub}; hence $\psi(\f(H)) = H$. For the other direction let $K \in \mathcal{C}(\Grp{\m L})$; then for each $f \in \grp{L}$, $f \in \f(\psi(K))$ iff $f \in \psi(K) = \overline{K}$ which is equivalent to $f\in K$, since $K$ is a convex subgroup of $\Grp{\m L}$. So $\f(\psi(K)) = K$.
\end{proof}

It is well-known that the lattice of convex $\ell$-subgroups of an $\ell$-group is a frame \cite{Lo}, i.e., a complete distributive lattice such that binary meets distribute over arbitrary joins. Hence, as a direct consequence of Theorem~\ref{t:conv-sub-iso}, we obtain the following result. 

\begin{corollary}
    If $\m L$ is a periodic $\ell$-pregroup, then $\mathcal{C}(\m L)$ is a frame. In particular, it is a distributive lattice.
\end{corollary}

\subsection{Convex normal subalgebras}

Congruences play an important role in the study of algebraic structures and in some cases (groups, rings, Boolean algebras) they correspond to special subsets (normal subgroups, ideals, filters, respectively). The same is true in the context of (involutive) residuated lattices, hence also of $\ell$-pregroups as we detail below.

Recall that an algebra $\m{A}$ is \emph{subdirectly irreducible (SI)} if the trivial congruence $\Delta_A$ on $\m{A}$ is completely meet-irreducible in the congruence lattice of $\m{A}$, and it is \emph{finitely subdirectly irreducible (FSI)} if $\Delta_A$ is meet-irreducible; this is the internal description of SIs and FSIs, while the (equivalent) external description pertains to subdirect product representations. Note that every subdirectly irreducible algebra is finitely subdirectly irreducible. A variety of algebras $\mathcal{V}$ is generated by its (finitely) subdirectly irreducibles, so to show that two varieties are equal it is enough to show that they contain the same (finitely) subdirectly irreducibles. The correspondence between congruences and special subsets will simplify the study of SIs and FSIs in the context of $\ell$-pregroups.

By Lemma~\ref{l: InRL}, $\ell$-pregroups are special cases of involutive residuated lattices and thus the general theory of the latter applies, including the characterization of congruences. Specializing the general notions to our context, we define the \emph{left conjugate} of $a$ by $b$ to be $b^rab \mt 1$ and the \emph{right conjugate} of $a$ by $b$ to be $bab^\ell \mt 1$. 

\begin{remark}\label{r: conjLP}
Note that the right conjugate of $a$ by $b$ is equal to the left conjugate of $a$ by $b^\ell$, since  $bab^\ell \mt 1=  (b^\ell)^ra(b^\ell) \mt 1$, by Lemma~\ref{l: InRL}; likewise, left conjugates can be viewed as right conjugates, in the context of $\ell$-pregroups.
\emph{Iterated conjugates} are defined by successive application of conjugation. This general notion also simplifies in our context, since iterated conjugates can be viewed as usual conjugates; for example,  
$c^r(b^rab \mt 1)c \mt 1=c^rb^rabc \mt c^rc \mt 1=(bc)^ra(bc) \mt 1$, by Lemma~\ref{l: InRL}.
\end{remark}

 A subalgebra of an $\ell$-pregroup $\m L$ is called \emph{normal} if it is closed under conjugation (by arbitrary elements from $L$); by the above observation it is enough to check closure under left or under right conjugation. Specializing a result from residuated lattices (Theorem~3.47 of \cite{GJKO}), and using Remark~\ref{r: conjLP}, we get the following results. 

 \begin{lemma}\label{l: CNSforRL}
 The convex normal subalgebra $\Nc_\m{L}(S)$ of an $\ell$-pregroup  $\m L$ generated by a subset $S \subseteq L$ is the same as the one generated by 
 $
 T:=\{a\mt a^\ell \mt 1: a \in S\}
 $
 and it is equal to 
 \[
 \Nc_\m{L}(S)=\{b \in L : (\exists p\in P)(p \leq b \leq p^\ell)\},
 \]
 where $P$ is the set of all products of conjugates (by elements of $L$) of elements of $T$; in the above sets one can use $^r$ instead of ${\;}^\ell$ and the same characterization holds.
 \end{lemma}

\begin{lemma}\label{l: CNSandCONforRL}
For any $\ell$-pregroup $\m L$ its congruence lattice is isomorphic to the lattice $\NC(\m{L})$ of convex normal subalgebras of $\m L$ {via the mutually inverse maps}
$\Theta \mapsto [1]_\Theta$ and $H \mapsto \{(a,b) \in L^2 :  ab^\ell \mt ba^\ell \mt 1 \in H\}$.
\end{lemma}
 
We first show that the isomorphism from $\mathcal{C}(\m{L})$ to $\mathcal{C}(\Grp{\m L})$ that we introduced above restricts to an injective map from $\NC(\m{L})$ into $\NC(\Grp{\m{L}})$.

\begin{lemma}\label{l:groupNC-pregroupNC}
If $\m{L}$ is a periodic $\ell$-pregroup and $H \in \NC(\m{L})$, then $H \cap \grp{L}\in \NC(\Grp{\m{L}})$ and $\Nc_\m{L}(H\cap \grp{L}) = \Cv_\m{L}(H\cap \grp{L}) =H$. 
\end{lemma}
\begin{proof}

Since each of $H$ and $ \grp{L}$ is convex and closed under conjugation, we get that $H \cap \grp{L} \in \NC(\Grp{\m{L}})$.
Moreover, by Theorem~\ref{t:conv-sub-iso}, $\Cv_{\m{L}}(H \cap \grp{L}) = H$. Since $H$ is normal, we also have $\Cv_{\m{L}}(H \cap \grp{L}) \subseteq \Nc_{\m{L}}(H \cap \grp{L}) \subseteq H$.
\end{proof}

\begin{corollary}\label{c: FSI skeleton}
    For any periodic $\ell$-pregroup $\m{L}$, $\NC(\m{L})$ is isomorphic to a bounded infinitary meet-subsemilattice of $\NC(\Grp{\m L})$. In particular, if $\Grp{\m{L}}$ is (finitely) subdirectly irreducible or simple, then $\m{L}$ is (finitely) subdirectly irreducible or simple, respectively.
\end{corollary}

\begin{proof}
    By Lemma~\ref{l:groupNC-pregroupNC},  $H \mapsto H\cap \grp{L}$ is an injective map from $\NC(\m{L})$ into $\NC(\Grp{\m L})$ and it is clear that it preserves arbitrary intersections and the bounds. 
\end{proof}

As we will see below,  $n$-periodic $\ell$-pregroups that satisfy  $\g_n(x)^n y \approx y \g_n(x)^n$ allow for a tighter control of the conjugates. {We use this to show that in this case the (mutually inverse) isomorphisms between $\mathcal{C}(\m{L})$ and $\mathcal{C}(\Grp{\m L})$ restrict to isomorphisms between $\NC(\m{L})$ and $\NC(\Grp{\m{L}})$.}

\begin{lemma}\label{l:n-comm-conjugation}
If an $n$-periodic $\ell$-pregroup  $\m{L}$  satisfies the equation $\g_n(x)^n y \approx y \g_n(x)^n$, then for each $a\in L$ and $g\in \grp{L}$ with $g\leq 1$, we have $g^n \leq a^r g a \mt 1$ (and $g^n \leq a g a^\ell \mt 1$). 
\end{lemma}
\begin{proof}
By Lemma~\ref{l: sigmagamma}, the group elements of $L$ are exactly the images of $\sigma$, so since $\m L$ satisfies the equation we have that $g^n$ is central and negative; hence $a g^n = g^n a \leq ga$ and, by residuation, $g^n \leq a^r ga$. 
\end{proof}

\begin{proposition}\label{p:isoCN} 
If $\m{L}$ is an $n$-periodic $\ell$-pregroup that satisfies $\g_n(x)^n y \approx y \g_n(x)^n$, then the maps 
\begin{align*}
&\f \colon \NC(\m{L}) \to \NC(\Grp{\m{L}}),\, H \mapsto H \cap \grp{L} \\
&\ps \colon \NC(\Grp{\m{L}}) \to \NC(\m{L}), \,  G \mapsto \Nc_\m{L}(G).
\end{align*}
form a pair of mutually inverse lattice isomorphisms.  {In particular, they are the restrictions of the isomorphisms between $\mathcal{C}(\m{L})$ and $\mathcal{C}(\Grp{\m{L}})$ from Theorem~\ref{t:conv-sub-iso}.}
\end{proposition}

\begin{proof}
It is clear that both maps are order-preserving and, by Lemma~\ref{l:groupNC-pregroupNC}, $\f$ is well defined and $\ps\circ \f = id$. To show that $\f \circ \ps = id$, let $G \in \NC(\Grp{\m{L}})$ and define $H= \Nc_\m{L}(G)$. For $h\in H \cap\grp{L}$, we have  $f:=h \mt h^{-1} \mt  1 \in  H$. 
Then, by Lemma~\ref{l:n-comm-conjugation}, and Lemma~\ref{l: CNSforRL}, there exist $g_1,\dots, g_k \in G$ such that  $g_1^n\cdots g_k^n \leq f \leq 1$, yielding $f \in G$. Since $f \leq h \leq f^{-1}$, by convexity, we get $h \in G$. Hence, $H \cap\grp{L} = G$. 

Finally, by Lemma~\ref{l:groupNC-pregroupNC} and the fact that $\f$ is surjective, we get that for each $G \in \NC(\Grp{\m{L}})$, $\Nc_\m{L}(G) = \Cv_\m{L}(G)$, so $\f$ and $\ps$ are the restrictions of the isomorphisms of  Theorem~\ref{t:conv-sub-iso}.
\end{proof}

For an $\ell$-group (or more generally for a $1$-cyclic residuated lattice) $\m{L}$, the lattice $\NC(\m{L})$ is a complete sublattice of $\mathcal{C}(\m{L})$; see \cite[Proposition 5.4]{Bo}. So Proposition~\ref{p:isoCN} yields the following corollary.

\begin{corollary}
    If $\m{L}$ is an $n$-periodic $\ell$-pregroup that satisfies $\g_n(x)^n y \approx y \g_n(x)^n$, then $\NC(\m{L})$ is a complete sublattice of $\mathcal{C}(\m{L})$.
\end{corollary}

\begin{remark}
  As shown in Proposition~\ref{p:isoCN}, for any $n$-periodic $\ell$-pregroup $\m L$ that satisfies $\g_n(x)^ny \approx y\g(x)^n$, the isomorphisms between $\mathcal{C}(\m{L})$ and $\mathcal{C}(\Grp{\m{L}})$ of Theorem~\ref{t:conv-sub-iso} restrict to isomorphisms between $\NC(\m{L})$ and $\NC(\Grp{\m{\m{L}}})$. However, it is not clear if this restrictions yield isomorphisms for arbitrary $n$-periodic $\ell$-pregroups.
\end{remark}

An $\ell$-group is called \emph{representable} if it is a subdirect product of totally ordered $\ell$-groups. The class of representable $\ell$-groups forms a variety that is axiomatized by $1 \leq x \jn yx^{-1}y^{-1}$ (see, e.g., \cite{KM} for more details). The next lemma characterizes the representable $\ell$-groups that are finitely subdirectly irreducible. The result is well-known, but we prove it for completeness. 
\begin{lemma}\label{l:representable-FSI} 
    A non-trivial representable $\ell$-group is finitely subdirectly irreducible if and only if it is totally ordered. In particular, \ a non-trivial abelian $\ell$-group is finitely subdirectly irreducible if and only if it is totally ordered.
\end{lemma}

\begin{proof}
    For the left-to-right implication note that the subdirectly irreducible representable $\ell$-groups are totally ordered, since by definition they are subdirect products of totally ordered $\ell$-groups; also, each finitely subdirectly irreducible embeds into an ultraproduct of subdirectly irreducible algebras, by \cite[Lemma~1.5]{CD}. 

    For the right-to-left implication note that if $\m{G}$ is a totally ordered $\ell$-group, then its lattice of convex normal $\ell$-subgroups is a chain (e.g., using Lemma~\ref{l: CNSforRL}). Hence the trivial convex normal $\ell$-subgroup of $\m{G}$ is  meet-irreducible, i.e., $\m{G}$ is finitely subdirectly irreducible.

    Finally it is well-known that abelian $\ell$-groups are representable (see e.g., \cite[Proposition~9.3.3]{KM}).
\end{proof}

\begin{proposition}[{\cite[Lemma~1.4]{Da2}}] 
\label{p: abelian}
    If an $\ell$-group satisfies $xy^n \approx y^n x$  for some $n>0$, then it is abelian.
\end{proposition}

\begin{corollary}\label{c: Lg is abelian}
    If $\m{L}$ is an $n$-periodic $\ell$-pregroup that satisfies $x\g(y)^n \approx \g(y)^n x$,  then  $\Grp{\m L}$ is abelian.
\end{corollary}

We now obtain a characterization of the finitely subdirectly irreducibles in the subvariety of $\mathsf{LP}_n$ axiomatized by $x\g(y)^n \approx \g(y)^n x$.

\begin{corollary}\label{c:fsi-commutative}
    Let $\m{L}$ be a non-trivial $n$-periodic $\ell$-pregroup that satisfies the equation $x\g(y)^n \approx \g(y)^n x$. 
    Then $\m{L}$ is finitely subdirectly irreducible if and only if $\Grp{\m{L}}$ is totally ordered (and abelian). 
\end{corollary}

\begin{proof}

By Proposition~\ref{p:isoCN}, $\m{L}$ is finitely subdirectly irreducible if and only if $\Grp{\m{L}}$ is. 
Since $\m{L}$  satisfies $x\g(y)^n \approx \g(y)^n x$,  $\Grp{\m{L}}$ is abelian, by Corollary~\ref{c: Lg is abelian}. So, by Lemma~\ref{l:representable-FSI}, $\Grp{\m{L}}$ is finitely subdirectly irreducible if and only if it is totally ordered.
\end{proof}

\section{On the structure of \texorpdfstring{$\m{F}_n(\ZZ)$}{Fₙ(ℤ)}}\label{s: FnZ}

In this section, we  prove some results about the structure of $\m{F}_n(\ZZ)$ that will be used later in the axiomatization proof, but we note that they also  are of independent interest.
For example, we define the notion of periodicity of an element and show that every element of $\m{F}_n(\ZZ)$ of periodicity $k \mid n$ generates all of $\m{F}_k(\ZZ)$ and that elements of incomparable periodicities generate an element with periodicity the least common multiple of the periodicities. Also, we define the height of an element of $\m{F}_n(\ZZ)$ and show that every  strictly positive idempotent $a$ generates a  positive idempotent,  $\slt(a)$, with height $1$. We start by defining the operators $\m F$ and $\m F_n$.

\medskip

Given functions $f\colon \mathbf{P}\rightarrow\mathbf{Q}$ and $g\colon \mathbf{Q}\rightarrow\mathbf{P}$ between posets, we say that $g$ is a \emph{residual} for $f$, or that $f$ is a \emph{dual residual} for $g$, or that $(f,g)$ forms a \emph{residuated pair}, if 
\[
f(p)\leq q\Leftrightarrow p\leq g(q)\text{, for all }p\in P,q\in Q.
\]
The residual and the dual residual of $f$ are unique when they exist and we denote them by $f^r$ and $f^{\ell}$, respectively; if they exist, they are given by
\[
f^{\ell}(q)=\min \{ p\in P : q\leq f (p) \}
\quad \text{ and } \quad 
f^{r}(q)=\max \{ p\in P : f(p)\leq q \}
\]
and $f$ is called \emph{residuated} or \emph{dually residuated}, respectively. \emph{Residuals  and dual residuals of order $n$} are defined, when they exist, by $f^{(-n)}:=f^{r^n}$ and $f^{(n)}:=f^{\ell^n}$; i.e.,  $f^{(0)}=f$, $f^{(n+1)}=(f^{(n)})^\ell$ and $f^{(-n-1)}=(f^{(-n)})^r$, for all $n \in \ZZ^+$.

The set of all maps on a chain $\mathbf{\Omega}$ that have residuals and dual residuals of all orders supports an $\ell$-pregroup $\m F(\m{\Omega})$, under composition, identity, the pointwise order, ${}^r$ and ${}^\ell$ \cite{GJKO}. In fact  $\m F(\m{\Omega})$ is a \emph{distributive} $\ell$-pregroup (i.e., its underlying lattice is distributive); we denote the variety of distributive $\ell$-pregroups by $\mathsf{DLP}$. In \cite{GH} it is shown that every distributive $\ell$-pregroup can be embedded into $\m F(\m{\Omega})$  for some  chain $\m \Omega$ and in \cite{GG1} it is proved that $\m \Omega$ can be taken to be a lexicographic product $\m{J} \overrightarrow{\times}\mathbb{Z}$ for some chain $\m J$. 

 We denote by $\mathsf{LP}_n$ the variety of all $n$-periodic $\ell$-pregroups. It is proved in \cite{GJ} that $\mathsf{LP}_n \subseteq \mathsf{DLP}$, for all $n$; and we will call an element $a$ of an $\ell$-pregroup $n$-periodic if it satisfies $a = a^{[n]}$.
 In \cite{GG2} it is  shown that, for every $n \in \ZZ^+$, the $n$-periodic elements of an $\ell$-pregroup form a subalgebra (which is an $n$-periodic $\ell$-pregroup); 
 for $n=1$, we obtain the  \emph{group skeleton} as discussed above. 
 
 We denote by $\m F_n(\m{\Omega})$ the subalgebra of  $n$-periodic elements of $\m F(\m{\Omega})$; in particular $\m F_1(\m{\Omega})=\m {Aut}(\m{\Omega})$, the $\ell$-group of order-automorphisms on $\m \Omega$. 
  In \cite{GG2} it is shown that every $n$-periodic $\ell$-pregroup can be embedded in $\m F_n(\m{J} \overrightarrow{\times}\mathbb{Z})$  for some  chain  $\m 
 J$. Therefore, the operators $\m F$ and $\m F_n$ play, for distributive and $n$-periodic $\ell$-pregroups respectively, the role that $\m {Aut}$ plays for $\ell$-groups in Holland's representation theorem and the role that the group of permutations operator plays for groups in Cayley's representation theorem. 

 It is easy to see \cite{GJKO, GG1} that $\m F(\ZZ)$ consists precisely of the order-preserving maps on $\ZZ$ that are finite-to-one (i.e., the preimage of every singleton is a finite set/interval). Also, in \cite{GG2} it is observed that, for all $n$, $\m F_n(\ZZ)$ consists precisely of the order-preserving maps on $\ZZ$ that are $n$-periodic as functions. In the following we will be denoting by $\cn{s}$ the function $x\mapsto x+1$ on $\ZZ$  when viewing it as an element of $\m F(\ZZ)$ (or of $\m F_n(\ZZ)$). Note that $\Grp{\m{F}(\ZZ)} = \langle \cn{s} \rangle$ and that this algebra is isomorphic to the abelian $\ell$-group based on $\ZZ$, where for a subset $X$ of an algebra, we denote by $\langle X \rangle$ the subalgebra generated by $X$. 
 So $\m F_n(\ZZ)$ is simple, by Corollary~\ref{c: FSI skeleton}; this was already observed directly in \cite{GG2}.

\begin{lemma}[\cite{GG1,GG2}]\label{l:basic-per}
    Let $f$ be a map on $\ZZ$.
    \begin{enumerate}[label = \textup{(\arabic*)}]
        \item $f\in \mathrm{F}(\ZZ)$ iff $f$ is order-preserving and $f^{-1}[x]$ is finite for all $x\in \ZZ$.
        \item If $f\in \mathrm{F}(\ZZ)$, then $f \in \mathrm{F}_n(\ZZ)$ iff $f(x +kn) = f(x) +kn$ for all  $x,k\in \ZZ$.
    \end{enumerate}
\end{lemma}

\begin{corollary}\label{c:interval-leq}
    Let $f,g\in \mathrm{F}_n(\ZZ)$. Then $f \leq g$ if and only if for some $n$-element interval $I\subseteq \ZZ$, $f(x) \leq g(x)$ for all $x \in I$. 
\end{corollary}

\begin{lemma}\label{l:b-conjugation}
    If $a \in \mathrm{F}_n(\ZZ)$ and $k, l,x\in \ZZ$, then 
    \begin{enumerate}[label = \textup{(\arabic*)}]
    \item  $a^{[k]}(x) = a(x-k) + k = \cn{s}^ka\cn{s}^{-k}(x)$.
        \item     $a^{[k-l]}(x) = a^{[k]}(x+l)-l$.
        \item $a\cn{s}^{n}=\cn{s}^n a$.
    \end{enumerate}
\end{lemma}

\begin{proof}
(1) The first equality is Lemma~2.3 of \cite{GG2} and the second equality holds, by definition.

(2) By (1), we have $a^{[k-l]}(x) =  a(x-(k-l)) + k-l
=  a(x+l-k) + k-l= a^{[k]}(x+l)-l$.

(3) By (1) we have $a^{[n]}\cn{s}^n =\cn{s}^{n}a$ and, by $n$-periodicity, we have $a^{[n]}=a$.
\end{proof}

The \emph{height} of an element  $a \in \mathrm{F}(\ZZ)$ \emph{at a point} $k \in \ZZ$ is defined to be the quantity $a(k)-k$, and the 
 \emph{set of all heights} of  $a$ is
$\{a(k)-k : k \in \ZZ\}$; note that when $a$ is periodic, this set is finite. The \emph{maximal height}, $\hg^+(a)$, of $a$ is defined to be the maximum of this set, when it exists, and the \emph{minimal height}, $\hg^-(a)$, of $a$ is defined to be the minimum of this set, when it exists. 

\begin{lemma}\label{l:grp-hg}
    For $f \in \mathrm{F}_n(\ZZ)$,  $\g(f) = \cn{s}^{\hg^-(f)}$ and $\h(f) = \cn{s}^{\hg^+(f)}$, i.e., $\g(f)(x) = x + \hg^-(f)$ and  $\h(f)(x) =x + \hg^+(f)$, for every $x \in \ZZ$. In particular, for every $f \in \mathrm{F}_n(\ZZ)$ there exists $k \in \NN$ such that $\cn{s}^{-k} \leq f \leq \cn{s}^{k}$.
\end{lemma}
\begin{proof}
    The first claim follows immediately from the fact that any invertible element of $\mathrm{F}_n(\ZZ)$ is of the form $\cn{s}^k$ for some $k\in \ZZ$. The second claim follows by the first claim and Lemma~\ref{l: sigmagamma}.
\end{proof}

If $\m{L}$ is an $\ell$-pregroup and $a\in L$, the \emph{periodicity} $\per(a)$ of 
 $a$ is defined to be the minimum  $k\geq 1$ such that $a$ is $k$-periodic; we write $\per(a)=\infty$ if no such minimum exists. 
We call an element $a\in \mathrm{F}(\ZZ)$ \emph{positive (negative)} if $1 \leq a$ ($a \leq 1$) and \emph{strictly} positive (negative) if $1<a$ ($a < 1$).

The goal of this section is to show that $\m{F}_n(\ZZ)$ is generated by any of its elements of periodicity $n$. To this end we first show that it is generated by any of its minimal strictly positive idempotents. Then we show that more and more general types of elements of $\m{F}_n(\ZZ)$ can be reduced to minimal strictly positive idempotents.

\begin{lemma}\label{l:k-interval}
 If $a\in \mathrm{F}_n(\ZZ)$ has periodicity $n$ and  $k \mid n$ with $k\neq n$, then there exists an $l \in \ZZ$ such that $a(l+k) < a(l) + k$. 
\end{lemma}
\begin{proof}
Suppose for a contradiction that  $a(l) + k\leq a(l+k)$ for each $l\in \ZZ$, and let $q\in \NN$ with $n = kq$. Then for each $l \in \ZZ$ we have
\[
    a(l) + n = a(l) + kq\leq a(l+k) + k(q-1)\leq \dots \leq a(l+kq) = a(l+n),
\]
and, since $a$ is $n$-periodic, $a(l) + n = a(l+n)$, yielding $a(l+k) = a(l) + k$, contradicting the assumption that $a$ has periodicity $n > k$.
\end{proof}

\begin{remark}\label{r:posidem}
Note that the positive idempotents of $\m{F}(\ZZ)$ are exactly the closure operators on $\ZZ$.
In particular a positive idempotent $a \in \mathrm{F}(\ZZ)$ is completely determined by its fixed points, that is, 
    \[
    a(k) = \min \{l \in \ZZ : a(l) = l \text{ and } l \geq k \}.
    \]
\end{remark}
We call an element $a$ of an $\ell$-pregroup $\m{L}$ $\emph{flat}$ if there exist idempotents $b,c \in L$ such that $b\leq a \leq c$. 
For $f \in \mathrm{F}_n(\ZZ)$ we call an $n$-element interval $I \subseteq \ZZ$ a \emph{flat} interval  for $f$ if $f[I] \subseteq I$.

\begin{remark}\label{r: flat via automorohism}
    If $\m{L}$ is an $\ell$-pregroup and $a\in L$ is flat then, for each $k\in \ZZ$, $a^{[k]}$ is also  flat, since $x\mapsto x^{[k]}$ is an automorphism of $\m{L}$.
\end{remark}

In what follows we will often talk about intervals of $\ZZ$. For $x,y\in \ZZ$ we define $[x,y] = \{z\in \ZZ : x\leq z \leq y \}$. An interval of $\ZZ$ that is of particular importance for the study of $\m{F}_n(\ZZ)$ is the interval $[0,n-1]$ which we will also denote by $\ZZ_n$.

\begin{lemma}\label{l:flat-Fn}
For an element $f \in \mathrm{F}_n(\ZZ)$ the following are equivalent:
    \begin{enumerate}[label = \textup{(\alph*)}]
        \item $f$ is flat.
        \item $f$ has a fixed point.
        \item $f^{n} = f^{n-1}$.
        \item There exists a flat interval for $f$.
    \end{enumerate} 
If, furthermore, $f$ is positive, then the previous four items are equivalent to:
    \begin{enumerate}[label = \textup{(\alph*)}]
        \setcounter{enumi}{4}
        \item $\g(f) = 1$.
    \end{enumerate}
\end{lemma}
\begin{proof}
    (a)$\Rightarrow$(b) Suppose that $f$ is flat, i.e., there exist idempotents $b,c \in \mathrm{F}_n(\ZZ)$ such that $b \leq f \leq c$. 
    Let $x \in \ZZ$ and assume without loss of generality that $f(x) \leq x$. Then by monotonicity it follows that $f^2(x) \leq f(x)$ and inductively $f^{k+1}(x) \leq f^k(x)$ for each $k\geq 1$. Also  $b(x) = b^k(x) \leq f^k(x)$ for each $k\geq 1$, by assumption. Hence there exists a $k\geq 1$ with $f^{k+1}(x) = f^k(x)$, i.e., $f^k(x)$ is a fixed point of $f$.

    (b)$\Rightarrow$(c) Suppose that $f$ has a fixed point, i.e., there exists an $x\in \ZZ$ with $f(x) = x$. Since $f$ is $n$-periodic, by Lemma~\ref{l:basic-per}(1), $f(x + kn) = f(x) +kn = x +k n$ for each $k\in \ZZ$. We want to show that $f^{n}(y) = f^{n-1}(y)$ for all $y\in \ZZ$; so let $y \in \ZZ$. Then there exists an $l\in \ZZ$ with $x + ln < y \leq x + (l+1)n$, so,  
    \[
    x + ln = f^m(x+ln) \leq f^m(y) \leq f^m(x+ (l+1)n) = x+ (l+1)n,
    \]
    for each $m\geq 1$, by monotonicity and the calculation above. Since the interval $(x + ln,x + (l+1)n]$ contains only $n$ elements, it follows from monotonicity that $f^n(y) = f^{n-1}(y)$.

    (c)$\Rightarrow$(d) 
    If  $f^n = f^{n-1}$, then for all $y \in \ZZ$ we have  $f(f^{n-1}(y))=f^n (y)= f^{n-1}(y)$, so $f^{n-1}(y)$ is a fixed point of $f$, hence also of $f^{n-1}$; let $x\in \ZZ$ be one of the fixed points of $f^{n-1}$. Let $m = \min (f^{n-1})^{-1}[x]$ and define $I = [m,m+n-1]$. 
    By definition and monotonicity we have $m \leq x$, 
    Moreover, we have that $f(m-1) \leq m-1$, for otherwise  $m \leq f(m-1)$ would imply 
    \[
    x= f^{n-1}(m) \leq  f^{n-1}(f(m-1))=f^{n-1}(m-1) \leq f^{n-1}(m)= x,
    \]
    contradicting the minimality of $m$. So, by Lemma~\ref{l:basic-per}(1), $f(m+n-1) = f(m-1) + n \leq m-1 + n$. Hence $m \leq x=f(m)\leq f(m+n-1)  \leq m-1 + n$, by the monotonicity of $f$; so $f[I] \subseteq I$.

    (d)$\Rightarrow$(a) Suppose that there exists a flat interval $I\subseteq \ZZ$ for $f$, i.e., $I = [k, k+n-1]$ for some $k \in \ZZ$. Let $b \in \mathrm{F}_n(\ZZ)$ be defined by $b(x) = k$ for $x\in [k,k+n-1]$ and $n$-periodically extended, and let $c \in \mathrm{F}_n(\ZZ)$ be  defined as the $n$-periodic extension of $c(x) = k+n-1$ for $x\in [k,k+n-1]$. Note that $b$ and $c$ are idempotent and for all $x \in I$, $b(x) \leq f(x) \leq c(x)$. Thus, by Corollary~\ref{c:interval-leq}, $b \leq f \leq c$, i.e., $f$ is flat.

    Finally, suppose that $f$ is positive; then $\hg^-(f) \geq 0$.
    Now if $f$ has a fixed point $x\in \ZZ$, then $f(x) -x = 0$ and $0 = \hg^-(f)$; so Lemma~\ref{l:grp-hg} yields $\g(f) = \cn{s}^0 = 1$. Conversely if $\g(f) = 1$, then by Lemma~\ref{l:grp-hg}, there exists an $x\in \ZZ$ with $f(x) -x = 0$, i.e., $f$ has a fixed point. This shows that (e) is equivalent to (b).
\end{proof}

Note that $\g(x) \leq x$, by definition, thus $1 \leq \g(x)^{-1}x$, for every element $x$ in an $n$-periodic $\ell$-pregroup.

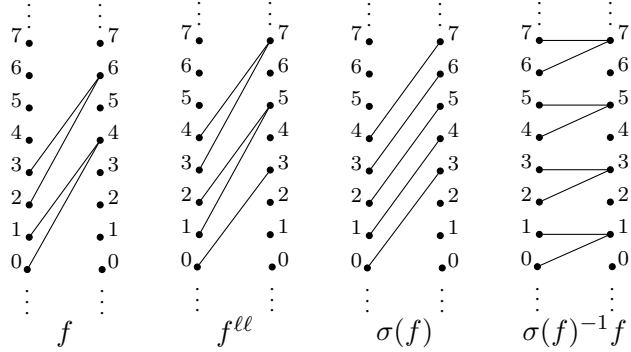
\begin{figure}[ht]\def\eq{=}
\centering
{\scriptsize
\begin{tikzpicture}
[scale=0.43]
%--------
\node[fill,draw,circle,scale=0.3,left](9) at (0,9){};
\node[right](9.1,0) at (-.9,9.25){$7$};
\node[fill,draw,circle,scale=0.3,right](9r) at (2,9){};
\node[right](9.1) at (2.1,9.25){$7$};

\node[fill,draw,circle,scale=0.3,left](8) at (0,8){};
\node[right](8.1,0) at (-.9,8.25){$6$};
\node[fill,draw,circle,scale=0.3,right](8r) at (2,8){};
\node[right](8.1) at (2.1,8.25){$6$};

\node[fill,draw,circle,scale=0.3,left](7) at (0,7){};
\node[right](7.1,0) at (-.9,7.25){$5$};
\node[fill,draw,circle,scale=0.3,right](7r) at (2,7){};
\node[right](7.1) at (2.1,7.25){$5$};

\node[fill,draw,circle,scale=0.3,left](6) at (0,6){};
\node[right](6.1,0) at (-.9,6.25){$4$};
\node[fill,draw,circle,scale=0.3,right](6r) at (2,6){};
\node[right](6.1) at (2.1,6.25){$4$};

\node[fill,draw,circle,scale=0.3,left](5) at (0,5){};
\node[right](9.1,0) at (-.9,5.25){$3$};
\node[fill,draw,circle,scale=0.3,right](5r) at (2,5){};
\node[right](5.1) at (2.1,5.25){$3$};

\node[fill,draw,circle,scale=0.3,left](4) at (0,4){};
\node[right](8.1,0) at (-.9,4.25){$2$};
\node[fill,draw,circle,scale=0.3,right](4r) at (2,4){};
\node[right](4.1) at (2.1,4.25){$2$};

\node[fill,draw,circle,scale=0.3,left](3) at (0,3){};
\node[right](3.1,0) at (-.9,3.25){$1$};
\node[fill,draw,circle,scale=0.3,right](3r) at (2,3){};
\node[right](3.1) at (2.1,3.25){$1$};

\node[fill,draw,circle,scale=0.3,left](2) at (-0.06,2){};
\node[right](2.1,0) at (-.9,2.25){$0$};
\node[fill,draw,circle,scale=0.3,right](2r) at (2.06,2){};
\node[right](2.1) at (2.1,2.25){$0$};
%---------
\node at (9)[above=3pt]{$\vdots$};
\node at (9r)[above=3pt]{$\vdots$};
\node at (2)[below=-1pt]{$\vdots$};
\node at (2r)[below=-1pt]{$\vdots$};
%--------
\draw[-](2)--(6r);
\draw[-](3)--(6r);
\draw[-](4)--(8r);
\draw[-](5)--(8r);

%--------
\node at (1,0){\small $f$};
\end{tikzpicture}
\qquad
\begin{tikzpicture}
[scale=0.43]
%--------
\node[fill,draw,circle,scale=0.3,left](9) at (0,9){};
\node[right](9.1,0) at (-.9,9.25){$7$};
\node[fill,draw,circle,scale=0.3,right](9r) at (2,9){};
\node[right](9.1) at (2.1,9.25){$7$};

\node[fill,draw,circle,scale=0.3,left](8) at (0,8){};
\node[right](8.1,0) at (-.9,8.25){$6$};
\node[fill,draw,circle,scale=0.3,right](8r) at (2,8){};
\node[right](8.1) at (2.1,8.25){$6$};

\node[fill,draw,circle,scale=0.3,left](7) at (0,7){};
\node[right](7.1,0) at (-.9,7.25){$5$};
\node[fill,draw,circle,scale=0.3,right](7r) at (2,7){};
\node[right](7.1) at (2.1,7.25){$5$};

\node[fill,draw,circle,scale=0.3,left](6) at (0,6){};
\node[right](6.1,0) at (-.9,6.25){$4$};
\node[fill,draw,circle,scale=0.3,right](6r) at (2,6){};
\node[right](6.1) at (2.1,6.25){$4$};

\node[fill,draw,circle,scale=0.3,left](5) at (0,5){};
\node[right](9.1,0) at (-.9,5.25){$3$};
\node[fill,draw,circle,scale=0.3,right](5r) at (2,5){};
\node[right](5.1) at (2.1,5.25){$3$};

\node[fill,draw,circle,scale=0.3,left](4) at (0,4){};
\node[right](8.1,0) at (-.9,4.25){$2$};
\node[fill,draw,circle,scale=0.3,right](4r) at (2,4){};
\node[right](4.1) at (2.1,4.25){$2$};

\node[fill,draw,circle,scale=0.3,left](3) at (0,3){};
\node[right](3.1,0) at (-.9,3.25){$1$};
\node[fill,draw,circle,scale=0.3,right](3r) at (2,3){};
\node[right](3.1) at (2.1,3.25){$1$};

\node[fill,draw,circle,scale=0.3,left](2) at (-0.06,2){};
\node[right](2.1,0) at (-.9,2.25){$0$};
\node[fill,draw,circle,scale=0.3,right](2r) at (2.06,2){};
\node[right](2.1) at (2.1,2.25){$0$};
%---------
\node at (9)[above=3pt]{$\vdots$};
\node at (9r)[above=3pt]{$\vdots$};
\node at (2)[below=-1pt]{$\vdots$};
\node at (2r)[below=-1pt]{$\vdots$};
%--------
\draw[-](2)--(5r);
\draw[-](3)--(7r);
\draw[-](4)--(7r);
\draw[-](5)--(9r);
\draw[-](6)--(9r);
%--------
\node at (1,0){\small $f^{\ell\ell}$};
\end{tikzpicture}
\qquad
\begin{tikzpicture}
[scale=0.43]
%--------
\node[fill,draw,circle,scale=0.3,left](9) at (0,9){};
\node[right](9.1,0) at (-.9,9.25){$7$};
\node[fill,draw,circle,scale=0.3,right](9r) at (2,9){};
\node[right](9.1) at (2.1,9.25){$7$};

\node[fill,draw,circle,scale=0.3,left](8) at (0,8){};
\node[right](8.1,0) at (-.9,8.25){$6$};
\node[fill,draw,circle,scale=0.3,right](8r) at (2,8){};
\node[right](8.1) at (2.1,8.25){$6$};

\node[fill,draw,circle,scale=0.3,left](7) at (0,7){};
\node[right](7.1,0) at (-.9,7.25){$5$};
\node[fill,draw,circle,scale=0.3,right](7r) at (2,7){};
\node[right](7.1) at (2.1,7.25){$5$};

\node[fill,draw,circle,scale=0.3,left](6) at (0,6){};
\node[right](6.1,0) at (-.9,6.25){$4$};
\node[fill,draw,circle,scale=0.3,right](6r) at (2,6){};
\node[right](6.1) at (2.1,6.25){$4$};

\node[fill,draw,circle,scale=0.3,left](5) at (0,5){};
\node[right](9.1,0) at (-.9,5.25){$3$};
\node[fill,draw,circle,scale=0.3,right](5r) at (2,5){};
\node[right](5.1) at (2.1,5.25){$3$};

\node[fill,draw,circle,scale=0.3,left](4) at (0,4){};
\node[right](8.1,0) at (-.9,4.25){$2$};
\node[fill,draw,circle,scale=0.3,right](4r) at (2,4){};
\node[right](4.1) at (2.1,4.25){$2$};

\node[fill,draw,circle,scale=0.3,left](3) at (0,3){};
\node[right](3.1,0) at (-.9,3.25){$1$};
\node[fill,draw,circle,scale=0.3,right](3r) at (2,3){};
\node[right](3.1) at (2.1,3.25){$1$};

\node[fill,draw,circle,scale=0.3,left](2) at (-0.06,2){};
\node[right](2.1,0) at (-.9,2.25){$0$};
\node[fill,draw,circle,scale=0.3,right](2r) at (2.06,2){};
\node[right](2.1) at (2.1,2.25){$0$};
%---------
\node at (9)[above=3pt]{$\vdots$};
\node at (9r)[above=3pt]{$\vdots$};
\node at (2)[below=-1pt]{$\vdots$};
\node at (2r)[below=-1pt]{$\vdots$};
%--------
\draw[-](2)--(5r);
\draw[-](3)--(6r);
\draw[-](4)--(7r);
\draw[-](5)--(8r);
\draw[-](6)--(9r);
%--------
\node at (1,0){\small $\g(f)$};
\end{tikzpicture}
\qquad
\begin{tikzpicture}
[scale=0.43]
%--------
\node[fill,draw,circle,scale=0.3,left](9) at (0,9){};
\node[right](9.1,0) at (-.9,9.25){$7$};
\node[fill,draw,circle,scale=0.3,right](9r) at (2,9){};
\node[right](9.1) at (2.1,9.25){$7$};

\node[fill,draw,circle,scale=0.3,left](8) at (0,8){};
\node[right](8.1,0) at (-.9,8.25){$6$};
\node[fill,draw,circle,scale=0.3,right](8r) at (2,8){};
\node[right](8.1) at (2.1,8.25){$6$};

\node[fill,draw,circle,scale=0.3,left](7) at (0,7){};
\node[right](7.1,0) at (-.9,7.25){$5$};
\node[fill,draw,circle,scale=0.3,right](7r) at (2,7){};
\node[right](7.1) at (2.1,7.25){$5$};

\node[fill,draw,circle,scale=0.3,left](6) at (0,6){};
\node[right](6.1,0) at (-.9,6.25){$4$};
\node[fill,draw,circle,scale=0.3,right](6r) at (2,6){};
\node[right](6.1) at (2.1,6.25){$4$};

\node[fill,draw,circle,scale=0.3,left](5) at (0,5){};
\node[right](9.1,0) at (-.9,5.25){$3$};
\node[fill,draw,circle,scale=0.3,right](5r) at (2,5){};
\node[right](5.1) at (2.1,5.25){$3$};

\node[fill,draw,circle,scale=0.3,left](4) at (0,4){};
\node[right](8.1,0) at (-.9,4.25){$2$};
\node[fill,draw,circle,scale=0.3,right](4r) at (2,4){};
\node[right](4.1) at (2.1,4.25){$2$};

\node[fill,draw,circle,scale=0.3,left](3) at (0,3){};
\node[right](3.1,0) at (-.9,3.25){$1$};
\node[fill,draw,circle,scale=0.3,right](3r) at (2,3){};
\node[right](3.1) at (2.1,3.25){$1$};

\node[fill,draw,circle,scale=0.3,left](2) at (-0.06,2){};
\node[right](2.1,0) at (-.9,2.25){$0$};
\node[fill,draw,circle,scale=0.3,right](2r) at (2.06,2){};
\node[right](2.1) at (2.1,2.25){$0$};
%---------
\node at (9)[above=3pt]{$\vdots$};
\node at (9r)[above=3pt]{$\vdots$};
\node at (2)[below=-1pt]{$\vdots$};
\node at (2r)[below=-1pt]{$\vdots$};
%--------
\draw[-](2)--(3r);
\draw[-](3)--(3r);
\draw[-](4)--(5r);
\draw[-](5)--(5r);
\draw[-](6)--(7r);
\draw[-](7)--(7r);
\draw[-](8)--(9r);
\draw[-](9)--(9r);
%--------
\node at (1,0){\small $\g(f)^{-1}f$};
\end{tikzpicture}
}
\caption{
An illustration of a function $f$ in $\m F_2(\ZZ)$, of $\g_2(f)=f \mt f^{\ell \ell}$ (the largest invertible element below $f$) and of $\g(f)^{-1}f$ (the associated strictly positive flat element). 
}
\label{f: spif}
\end{figure}

\begin{corollary}\label{c:flat-g}
    For every $f\in \mathrm{F}_n(\ZZ)$,  $\g(f)^{-1}f$ is flat and of the same periodicity as $f$. If $f$ is not invertible, then $\g(f)^{-1}f$ is strictly positive.
\end{corollary}
\begin{proof}
    By Remark~\ref{r:grp-conucleus}, $\g(\g(f)^{-1}f) = \g(f)^{-1}\g(f) = 1$ and we also have $1 \leq \g(f)^{-1}f$; so Lemma~\ref{l:flat-Fn} yields that $f$ is flat. Moreover,  for all $k \in \ZZ$, we have: $(\g(f)^{-1}f)^{[k]}=\g(f)^{-1}f$ iff  $(\g(f)^{-1})^{[k]}f^{[k]}=\g(f)^{-1}f$ iff  $\g(f)^{-1} f^{[k]}=\g(f)^{-1}f$ iff  $f^{[k]}=f$, since for every invertible element $g$ we have $g^{[k]}=g$. Finally, if $f$ is not invertible, then $\g(f)^{-1}f$ is not the identity, so it is strictly positive; see Figure~\ref{f: spif}.
\end{proof}

An element $a\in \mathrm{F}(\ZZ)$ is called an \emph{$n$-atom} if $a\in \mathrm{F}_n(\ZZ)$ and  $a$ is a cover of $1$ in $\m{F}_n(\ZZ)$, i.e.,  an  atom of the positive cone $\m{F}_n(\ZZ)^+:=\{b \in {F}_n(\ZZ) : 1 \leq b\}$.

\begin{proposition}\label{p:n-cover}
   For $n>1$, an element of $\mathrm{F}_n(\ZZ)$ is an $n$-atom if and only if it is positive idempotent with set of fixed points of the form $\ZZ \setminus (n\ZZ + k)$ for some $k \in \ZZ_n$.
\end{proposition}

\begin{proof}
    For the right-to-left direction we may assume without loss of generality that the set of fixed points of $a\in \mathrm{F}(\ZZ)$ is of the form $\ZZ \setminus n\ZZ$. By Remark~\ref{r:posidem}, we get $a(0) = 1$ and $a(x) = x$ for each $x \in [1,n-1]$. Thus if $b\in \mathrm{F}_n(\ZZ)$ with $1\leq b < a$, then $0 = 1(0) \leq b(0) < a(0) = 1$ and $x = 1(x) \leq b(x) \leq a(x) = x$ for $x\in [1,n-1]$. So $b(x) = x$ for all $x \in [0,n-1]$ and by $n$-periodicity $b = 1$, yielding that $a$ is an $n$-atom. 

    For the left-to-right direction assume that $a$ is an $n$-atom. Then there exists a $k\in \ZZ$ such that $a(k) > k$, i.e., $a(k) \geq k+1$. Let $b \in \mathrm{F}_n(\ZZ)$ be defined by 
    \[
    b(x) = \begin{cases}
        x + 1 &\text{if } x \in n\ZZ + k, \\
        x    &\text{otherwise.}
    \end{cases}
    \]
    Then, by $n$-periodicity, for each $x \in n\ZZ + k$, $b(x) = x+1 \leq a(x)$ and for each $x \in \ZZ\setminus(n\ZZ + k)$, $b(x) = x \leq a(x)$. Hence $1 < b \leq a$, yielding $b = a$, since $a$ is an $n$-atom.
\end{proof}
As an immediate corollary we obtain that all $n$-atoms are conjugates in view of Lemma~\ref{l:b-conjugation}(1).

\begin{corollary}\label{c:n-cover-conj}
If $a,b \in \mathrm{F}(\ZZ)$ are $n$-atoms, then there exists a $k\in \ZZ_n$ such that $b = a^{[k]}$, i.e.,  $b = \cn{s}^ka\cn{s}^{-k}$.
\end{corollary}

\begin{proposition}\label{p: n-cover generates}
    If $a \in \mathrm{F}_n(\ZZ)$ is an $n$-atom, then $\langle a\rangle = \m{F}_n(\ZZ)$.
\end{proposition}
\begin{proof}
Since $a$ is $n$-periodic, every element of $\langle a \rangle$ is also $n$-periodic, so the forward inclusion holds.

For the backward direction,  first note that by Corollary~\ref{c:n-cover-conj}, since $a$ is an $n$-atom, $\langle a \rangle$ contains all of the $n$-atoms. For $i \in [1, n]$, we define $c_i$,  by: $c_i(i-1 + tn)=i+tn$, for all $t\in \ZZ$, and $c_i(x)=x$ otherwise. By Proposition~\ref{p:n-cover}, the maps $c_1,\dots, c_n$ are exactly the $n$-atoms, so $c_1,\dots, c_n \in \langle a \rangle$. Moreover, $\hg^+(a) = 1$, so by Lemma~\ref{l:grp-hg}, $\cn{s} = \h(a) \in \langle a \rangle$. Finally note that it suffices to show that $a$ generates all the positive elements of $\m{F}_n(\ZZ)$, since for each $f\in \mathrm{F}_n(\ZZ)$ we have $f = (f \jn 1)(f \mt 1)=(f \jn 1)(f^\ell \jn 1)^r$.

 Let $f \in \mathrm{F}_n(\ZZ)$ be positive. We will first construct a suitable double shift $g$ of $f$ that maps the interval $[0,n-1]$ into itself. For $k:=\max f^{-1}[f(0)]$, note that $f(k+1) > f(0)$, since $k+1 \notin f^{-1}[f(0)]$, and $0 \leq k$, since $0 \in f^{-1}[f(0)]$. Also, by Lemma~\ref{l:basic-per}(2), $f(n)=f(0)+n\neq f(0)$, so $n  \notin f^{-1}[f(0)]$; hence $k<n$. By the same lemma we get $f(k+n)=f(k)+n=f(0)+n$. So, the interval $I:=[k+1, k+n]$ has $n$ elements and 
 \[
 f(k+n)-f(k+1)=f(0)+n-f(k+1)<f(0)+n-f(0)=n,
 \]
 hence $f[I]$ is also contained in an $n$-element interval. For $g:=\cn{s}^{-f(k+1)}f\cn{s}^{k+1}$ we have $g(0)=\cn{k}^{-f(k+1)}f\cn{s}^{k+1}(0)=\cn{s}^{-f(k+1)}f(k+1)=-f(k+1)+f(k+1)=0$ and 
$g(n-1)=\cn{s}^{-f(k+1)}f\cn{s}^{k+1}(n-1)=\cn{s}^{-f(k+1)}f(k+n)=-f(k+1)+f(k+n)<n$, from the calculation above. As a result, $g$ maps the interval $[0,n-1]$ into itself and  $f=\cn{s}^{f(k+1)}g\cn{s}^{-(k+1)}$.

 We denote by $\{k_1 <k_2< \ldots< k_{r}\}$ the set of all $k\in [0,n-1]$ that are in the image of $g$ and  $g^{-1}[k]\neq \{k\}$. For every such $k$,  we set $l_k=\min g^{-1}[k]+1$ and 
 note that because $g$ is positive and $g^{-1}[k]\neq \{k\}$ we have $\min g^{-1}[k]<k$, so $l_k\leq k$; we define the product
 $g_k:=c_{k}\cdots c_{l_k}$. 
 We will verify that $g=g_{k_1}g_{k_2}\cdots g_{k_r}$, where the empty product is the identity function; it is enough to check this equality on $\ZZ_n$. 
 Indeed, if $g(x)=y$ for some $x \in \ZZ_n$, 
 then in case $g^{-1}[y]= \{y\}$, we get $y=x$ and $g_{k_1}g_{k_2}\cdots g_{k_r}(x)=x=y$; in case $g^{-1}[y] \neq \{y\}$,
 then $y=k_i$, for some $i$ and $x \in g^{-1}[k_i]$. Also $g_{k_i}(x)=k_i$, while $g_{k_j}(x)=x$, for $j \neq i$. So, $g_{k_1}g_{k_2}\cdots g_{k_r}(x)=g_{k_1}g_{k_2}\cdots g_{k_i}(x)=g_{k_1}g_{k_2}\cdots g_{k_{i-1}}(k_i)=k_i=y$, because $ g_{k_j}(k_i)=k_i$ for $j<i$. Therefore, $g \in \langle a \rangle$, hence $f \in \langle a \rangle$.
\end{proof}

For $a \in \mathrm{F}_n(\ZZ)$ we define the set $P_a := \{l \in \ZZ_n : a^{[l]}(0) > 0\}$ and we write  $+_n$ and $-_n$ for addition and subtraction modulo $n$ in $\ZZ_n $. Moreover, if $P_a\neq \emptyset$, we define $\core(a) = \bigwedge_{l\in P_a} a^{[l]}$; see Figure~\ref{f: core}.

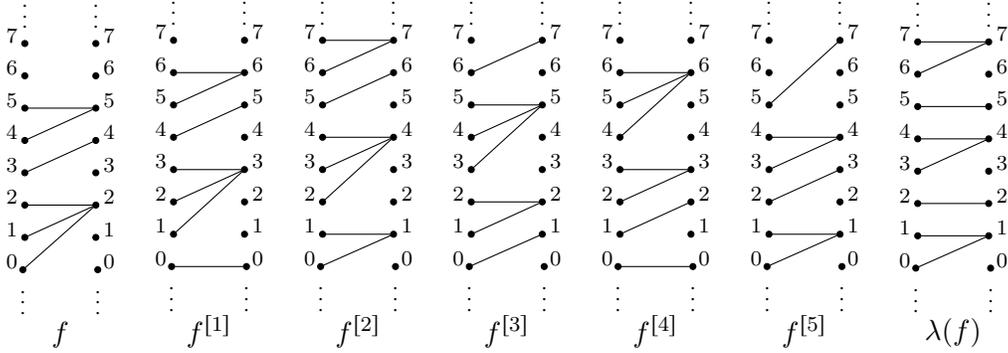
\begin{figure}[ht]\def\eq{=}
 \centering
{\scriptsize
\begin{tikzpicture}
[scale=0.43]
%--------
\node[fill,draw,circle,scale=0.3,left](9) at (0,9){};
\node[right](9.1,0) at (-.9,9.25){$7$};
\node[fill,draw,circle,scale=0.3,right](9r) at (2,9){};
\node[right](9.1) at (2.1,9.25){$7$};

\node[fill,draw,circle,scale=0.3,left](8) at (0,8){};
\node[right](8.1,0) at (-.9,8.25){$6$};
\node[fill,draw,circle,scale=0.3,right](8r) at (2,8){};
\node[right](8.1) at (2.1,8.25){$6$};

\node[fill,draw,circle,scale=0.3,left](7) at (0,7){};
\node[right](7.1,0) at (-.9,7.25){$5$};
\node[fill,draw,circle,scale=0.3,right](7r) at (2,7){};
\node[right](7.1) at (2.1,7.25){$5$};

\node[fill,draw,circle,scale=0.3,left](6) at (0,6){};
\node[right](6.1,0) at (-.9,6.25){$4$};
\node[fill,draw,circle,scale=0.3,right](6r) at (2,6){};
\node[right](6.1) at (2.1,6.25){$4$};

\node[fill,draw,circle,scale=0.3,left](5) at (0,5){};
\node[right](9.1,0) at (-.9,5.25){$3$};
\node[fill,draw,circle,scale=0.3,right](5r) at (2,5){};
\node[right](5.1) at (2.1,5.25){$3$};

\node[fill,draw,circle,scale=0.3,left](4) at (0,4){};
\node[right](8.1,0) at (-.9,4.25){$2$};
\node[fill,draw,circle,scale=0.3,right](4r) at (2,4){};
\node[right](4.1) at (2.1,4.25){$2$};

\node[fill,draw,circle,scale=0.3,left](3) at (0,3){};
\node[right](3.1,0) at (-.9,3.25){$1$};
\node[fill,draw,circle,scale=0.3,right](3r) at (2,3){};
\node[right](3.1) at (2.1,3.25){$1$};

\node[fill,draw,circle,scale=0.3,left](2) at (-0.06,2){};
\node[right](2.1,0) at (-.9,2.25){$0$};
\node[fill,draw,circle,scale=0.3,right](2r) at (2.06,2){};
\node[right](2.1) at (2.1,2.25){$0$};
%---------
\node at (9)[above=3pt]{$\vdots$};
\node at (9r)[above=3pt]{$\vdots$};
\node at (2)[below=-1pt]{$\vdots$};
\node at (2r)[below=-1pt]{$\vdots$};
%--------
\draw[-](2)--(4r);
\draw[-](3)--(4r);
\draw[-](4)--(4r);
\draw[-](5)--(6r);
\draw[-](6)--(7r);
\draw[-](7)--(7r);

%--------
\node at (1,0){\small $f$};
\end{tikzpicture}
\quad
\begin{tikzpicture}
[scale=0.43]
%--------
\node[fill,draw,circle,scale=0.3,left](9) at (0,9){};
\node[right](9.1,0) at (-.9,9.25){$7$};
\node[fill,draw,circle,scale=0.3,right](9r) at (2,9){};
\node[right](9.1) at (2.1,9.25){$7$};

\node[fill,draw,circle,scale=0.3,left](8) at (0,8){};
\node[right](8.1,0) at (-.9,8.25){$6$};
\node[fill,draw,circle,scale=0.3,right](8r) at (2,8){};
\node[right](8.1) at (2.1,8.25){$6$};

\node[fill,draw,circle,scale=0.3,left](7) at (0,7){};
\node[right](7.1,0) at (-.9,7.25){$5$};
\node[fill,draw,circle,scale=0.3,right](7r) at (2,7){};
\node[right](7.1) at (2.1,7.25){$5$};

\node[fill,draw,circle,scale=0.3,left](6) at (0,6){};
\node[right](6.1,0) at (-.9,6.25){$4$};
\node[fill,draw,circle,scale=0.3,right](6r) at (2,6){};
\node[right](6.1) at (2.1,6.25){$4$};

\node[fill,draw,circle,scale=0.3,left](5) at (0,5){};
\node[right](9.1,0) at (-.9,5.25){$3$};
\node[fill,draw,circle,scale=0.3,right](5r) at (2,5){};
\node[right](5.1) at (2.1,5.25){$3$};

\node[fill,draw,circle,scale=0.3,left](4) at (0,4){};
\node[right](8.1,0) at (-.9,4.25){$2$};
\node[fill,draw,circle,scale=0.3,right](4r) at (2,4){};
\node[right](4.1) at (2.1,4.25){$2$};

\node[fill,draw,circle,scale=0.3,left](3) at (0,3){};
\node[right](3.1,0) at (-.9,3.25){$1$};
\node[fill,draw,circle,scale=0.3,right](3r) at (2,3){};
\node[right](3.1) at (2.1,3.25){$1$};

\node[fill,draw,circle,scale=0.3,left](2) at (-0.06,2){};
\node[right](2.1,0) at (-.9,2.25){$0$};
\node[fill,draw,circle,scale=0.3,right](2r) at (2.06,2){};
\node[right](2.1) at (2.1,2.25){$0$};
%---------
\node at (9)[above=3pt]{$\vdots$};
\node at (9r)[above=3pt]{$\vdots$};
\node at (2)[below=-1pt]{$\vdots$};
\node at (2r)[below=-1pt]{$\vdots$};
%--------
\draw[-](2)--(2r);
\draw[-](3)--(5r);
\draw[-](4)--(5r);
\draw[-](5)--(5r);
\draw[-](6)--(7r);
\draw[-](7)--(8r);
\draw[-](8)--(8r);
%--------
\node at (1,0){\small $f^{[1]}$};
\end{tikzpicture}
\quad
\begin{tikzpicture}
[scale=0.43]
%--------
\node[fill,draw,circle,scale=0.3,left](9) at (0,9){};
\node[right](9.1,0) at (-.9,9.25){$7$};
\node[fill,draw,circle,scale=0.3,right](9r) at (2,9){};
\node[right](9.1) at (2.1,9.25){$7$};

\node[fill,draw,circle,scale=0.3,left](8) at (0,8){};
\node[right](8.1,0) at (-.9,8.25){$6$};
\node[fill,draw,circle,scale=0.3,right](8r) at (2,8){};
\node[right](8.1) at (2.1,8.25){$6$};

\node[fill,draw,circle,scale=0.3,left](7) at (0,7){};
\node[right](7.1,0) at (-.9,7.25){$5$};
\node[fill,draw,circle,scale=0.3,right](7r) at (2,7){};
\node[right](7.1) at (2.1,7.25){$5$};

\node[fill,draw,circle,scale=0.3,left](6) at (0,6){};
\node[right](6.1,0) at (-.9,6.25){$4$};
\node[fill,draw,circle,scale=0.3,right](6r) at (2,6){};
\node[right](6.1) at (2.1,6.25){$4$};

\node[fill,draw,circle,scale=0.3,left](5) at (0,5){};
\node[right](9.1,0) at (-.9,5.25){$3$};
\node[fill,draw,circle,scale=0.3,right](5r) at (2,5){};
\node[right](5.1) at (2.1,5.25){$3$};

\node[fill,draw,circle,scale=0.3,left](4) at (0,4){};
\node[right](8.1,0) at (-.9,4.25){$2$};
\node[fill,draw,circle,scale=0.3,right](4r) at (2,4){};
\node[right](4.1) at (2.1,4.25){$2$};

\node[fill,draw,circle,scale=0.3,left](3) at (0,3){};
\node[right](3.1,0) at (-.9,3.25){$1$};
\node[fill,draw,circle,scale=0.3,right](3r) at (2,3){};
\node[right](3.1) at (2.1,3.25){$1$};

\node[fill,draw,circle,scale=0.3,left](2) at (-0.06,2){};
\node[right](2.1,0) at (-.9,2.25){$0$};
\node[fill,draw,circle,scale=0.3,right](2r) at (2.06,2){};
\node[right](2.1) at (2.1,2.25){$0$};
%---------
\node at (9)[above=3pt]{$\vdots$};
\node at (9r)[above=3pt]{$\vdots$};
\node at (2)[below=-1pt]{$\vdots$};
\node at (2r)[below=-1pt]{$\vdots$};
%--------
\draw[-](2)--(3r);
\draw[-](3)--(3r);
\draw[-](4)--(6r);
\draw[-](5)--(6r);
\draw[-](6)--(6r);
\draw[-](7)--(8r);
\draw[-](8)--(9r);
\draw[-](9)--(9r);
%--------
\node at (1,0){\small $f^{[2]}$};
\end{tikzpicture}
\quad
\begin{tikzpicture}
[scale=0.43]
%--------
\node[fill,draw,circle,scale=0.3,left](9) at (0,9){};
\node[right](9.1,0) at (-.9,9.25){$7$};
\node[fill,draw,circle,scale=0.3,right](9r) at (2,9){};
\node[right](9.1) at (2.1,9.25){$7$};

\node[fill,draw,circle,scale=0.3,left](8) at (0,8){};
\node[right](8.1,0) at (-.9,8.25){$6$};
\node[fill,draw,circle,scale=0.3,right](8r) at (2,8){};
\node[right](8.1) at (2.1,8.25){$6$};

\node[fill,draw,circle,scale=0.3,left](7) at (0,7){};
\node[right](7.1,0) at (-.9,7.25){$5$};
\node[fill,draw,circle,scale=0.3,right](7r) at (2,7){};
\node[right](7.1) at (2.1,7.25){$5$};

\node[fill,draw,circle,scale=0.3,left](6) at (0,6){};
\node[right](6.1,0) at (-.9,6.25){$4$};
\node[fill,draw,circle,scale=0.3,right](6r) at (2,6){};
\node[right](6.1) at (2.1,6.25){$4$};

\node[fill,draw,circle,scale=0.3,left](5) at (0,5){};
\node[right](9.1,0) at (-.9,5.25){$3$};
\node[fill,draw,circle,scale=0.3,right](5r) at (2,5){};
\node[right](5.1) at (2.1,5.25){$3$};

\node[fill,draw,circle,scale=0.3,left](4) at (0,4){};
\node[right](8.1,0) at (-.9,4.25){$2$};
\node[fill,draw,circle,scale=0.3,right](4r) at (2,4){};
\node[right](4.1) at (2.1,4.25){$2$};

\node[fill,draw,circle,scale=0.3,left](3) at (0,3){};
\node[right](3.1,0) at (-.9,3.25){$1$};
\node[fill,draw,circle,scale=0.3,right](3r) at (2,3){};
\node[right](3.1) at (2.1,3.25){$1$};

\node[fill,draw,circle,scale=0.3,left](2) at (-0.06,2){};
\node[right](2.1,0) at (-.9,2.25){$0$};
\node[fill,draw,circle,scale=0.3,right](2r) at (2.06,2){};
\node[right](2.1) at (2.1,2.25){$0$};
%---------
\node at (9)[above=3pt]{$\vdots$};
\node at (9r)[above=3pt]{$\vdots$};
\node at (2)[below=-1pt]{$\vdots$};
\node at (2r)[below=-1pt]{$\vdots$};
%--------
\draw[-](2)--(3r);
\draw[-](3)--(4r);
\draw[-](4)--(4r);
\draw[-](5)--(7r);
\draw[-](6)--(7r);
\draw[-](7)--(7r);
\draw[-](8)--(9r);
%\draw[-](8)--(9r);
%--------
\node at (1,0){\small $f^{[3]}$};
\end{tikzpicture}
\quad
\begin{tikzpicture}
[scale=0.43]
%--------
\node[fill,draw,circle,scale=0.3,left](9) at (0,9){};
\node[right](9.1,0) at (-.9,9.25){$7$};
\node[fill,draw,circle,scale=0.3,right](9r) at (2,9){};
\node[right](9.1) at (2.1,9.25){$7$};

\node[fill,draw,circle,scale=0.3,left](8) at (0,8){};
\node[right](8.1,0) at (-.9,8.25){$6$};
\node[fill,draw,circle,scale=0.3,right](8r) at (2,8){};
\node[right](8.1) at (2.1,8.25){$6$};

\node[fill,draw,circle,scale=0.3,left](7) at (0,7){};
\node[right](7.1,0) at (-.9,7.25){$5$};
\node[fill,draw,circle,scale=0.3,right](7r) at (2,7){};
\node[right](7.1) at (2.1,7.25){$5$};

\node[fill,draw,circle,scale=0.3,left](6) at (0,6){};
\node[right](6.1,0) at (-.9,6.25){$4$};
\node[fill,draw,circle,scale=0.3,right](6r) at (2,6){};
\node[right](6.1) at (2.1,6.25){$4$};

\node[fill,draw,circle,scale=0.3,left](5) at (0,5){};
\node[right](9.1,0) at (-.9,5.25){$3$};
\node[fill,draw,circle,scale=0.3,right](5r) at (2,5){};
\node[right](5.1) at (2.1,5.25){$3$};

\node[fill,draw,circle,scale=0.3,left](4) at (0,4){};
\node[right](8.1,0) at (-.9,4.25){$2$};
\node[fill,draw,circle,scale=0.3,right](4r) at (2,4){};
\node[right](4.1) at (2.1,4.25){$2$};

\node[fill,draw,circle,scale=0.3,left](3) at (0,3){};
\node[right](3.1,0) at (-.9,3.25){$1$};
\node[fill,draw,circle,scale=0.3,right](3r) at (2,3){};
\node[right](3.1) at (2.1,3.25){$1$};

\node[fill,draw,circle,scale=0.3,left](2) at (-0.06,2){};
\node[right](2.1,0) at (-.9,2.25){$0$};
\node[fill,draw,circle,scale=0.3,right](2r) at (2.06,2){};
\node[right](2.1) at (2.1,2.25){$0$};
%---------
\node at (9)[above=3pt]{$\vdots$};
\node at (9r)[above=3pt]{$\vdots$};
\node at (2)[below=-1pt]{$\vdots$};
\node at (2r)[below=-1pt]{$\vdots$};
%--------
\draw[-](2)--(2r);
\draw[-](3)--(4r);
\draw[-](4)--(5r);
\draw[-](5)--(5r);
\draw[-](6)--(8r);
\draw[-](7)--(8r);
\draw[-](8)--(8r);
%--------
\node at (1,0){\small $f^{[4]}$};
\end{tikzpicture}
\quad
\begin{tikzpicture}
[scale=0.43]
%--------
\node[fill,draw,circle,scale=0.3,left](9) at (0,9){};
\node[right](9.1,0) at (-.9,9.25){$7$};
\node[fill,draw,circle,scale=0.3,right](9r) at (2,9){};
\node[right](9.1) at (2.1,9.25){$7$};

\node[fill,draw,circle,scale=0.3,left](8) at (0,8){};
\node[right](8.1,0) at (-.9,8.25){$6$};
\node[fill,draw,circle,scale=0.3,right](8r) at (2,8){};
\node[right](8.1) at (2.1,8.25){$6$};

\node[fill,draw,circle,scale=0.3,left](7) at (0,7){};
\node[right](7.1,0) at (-.9,7.25){$5$};
\node[fill,draw,circle,scale=0.3,right](7r) at (2,7){};
\node[right](7.1) at (2.1,7.25){$5$};

\node[fill,draw,circle,scale=0.3,left](6) at (0,6){};
\node[right](6.1,0) at (-.9,6.25){$4$};
\node[fill,draw,circle,scale=0.3,right](6r) at (2,6){};
\node[right](6.1) at (2.1,6.25){$4$};

\node[fill,draw,circle,scale=0.3,left](5) at (0,5){};
\node[right](9.1,0) at (-.9,5.25){$3$};
\node[fill,draw,circle,scale=0.3,right](5r) at (2,5){};
\node[right](5.1) at (2.1,5.25){$3$};

\node[fill,draw,circle,scale=0.3,left](4) at (0,4){};
\node[right](8.1,0) at (-.9,4.25){$2$};
\node[fill,draw,circle,scale=0.3,right](4r) at (2,4){};
\node[right](4.1) at (2.1,4.25){$2$};

\node[fill,draw,circle,scale=0.3,left](3) at (0,3){};
\node[right](3.1,0) at (-.9,3.25){$1$};
\node[fill,draw,circle,scale=0.3,right](3r) at (2,3){};
\node[right](3.1) at (2.1,3.25){$1$};

\node[fill,draw,circle,scale=0.3,left](2) at (-0.06,2){};
\node[right](2.1,0) at (-.9,2.25){$0$};
\node[fill,draw,circle,scale=0.3,right](2r) at (2.06,2){};
\node[right](2.1) at (2.1,2.25){$0$};
%---------
\node at (9)[above=3pt]{$\vdots$};
\node at (9r)[above=3pt]{$\vdots$};
\node at (2)[below=-1pt]{$\vdots$};
\node at (2r)[below=-1pt]{$\vdots$};
%--------
\draw[-](2)--(3r);
\draw[-](3)--(3r);
\draw[-](4)--(5r);
\draw[-](5)--(6r);
\draw[-](6)--(6r);
\draw[-](7)--(9r);
%--------
\node at (1,0){\small $f^{[5]}$};
\end{tikzpicture}
\quad
\begin{tikzpicture}
[scale=0.43]
%--------
\node[fill,draw,circle,scale=0.3,left](9) at (0,9){};
\node[right](9.1,0) at (-.9,9.25){$7$};
\node[fill,draw,circle,scale=0.3,right](9r) at (2,9){};
\node[right](9.1) at (2.1,9.25){$7$};

\node[fill,draw,circle,scale=0.3,left](8) at (0,8){};
\node[right](8.1,0) at (-.9,8.25){$6$};
\node[fill,draw,circle,scale=0.3,right](8r) at (2,8){};
\node[right](8.1) at (2.1,8.25){$6$};

\node[fill,draw,circle,scale=0.3,left](7) at (0,7){};
\node[right](7.1,0) at (-.9,7.25){$5$};
\node[fill,draw,circle,scale=0.3,right](7r) at (2,7){};
\node[right](7.1) at (2.1,7.25){$5$};

\node[fill,draw,circle,scale=0.3,left](6) at (0,6){};
\node[right](6.1,0) at (-.9,6.25){$4$};
\node[fill,draw,circle,scale=0.3,right](6r) at (2,6){};
\node[right](6.1) at (2.1,6.25){$4$};

\node[fill,draw,circle,scale=0.3,left](5) at (0,5){};
\node[right](9.1,0) at (-.9,5.25){$3$};
\node[fill,draw,circle,scale=0.3,right](5r) at (2,5){};
\node[right](5.1) at (2.1,5.25){$3$};

\node[fill,draw,circle,scale=0.3,left](4) at (0,4){};
\node[right](8.1,0) at (-.9,4.25){$2$};
\node[fill,draw,circle,scale=0.3,right](4r) at (2,4){};
\node[right](4.1) at (2.1,4.25){$2$};

\node[fill,draw,circle,scale=0.3,left](3) at (0,3){};
\node[right](3.1,0) at (-.9,3.25){$1$};
\node[fill,draw,circle,scale=0.3,right](3r) at (2,3){};
\node[right](3.1) at (2.1,3.25){$1$};

\node[fill,draw,circle,scale=0.3,left](2) at (-0.06,2){};
\node[right](2.1,0) at (-.9,2.25){$0$};
\node[fill,draw,circle,scale=0.3,right](2r) at (2.06,2){};
\node[right](2.1) at (2.1,2.25){$0$};
%---------
\node at (9)[above=3pt]{$\vdots$};
\node at (9r)[above=3pt]{$\vdots$};
\node at (2)[below=-1pt]{$\vdots$};
\node at (2r)[below=-1pt]{$\vdots$};
%--------
\draw[-](2)--(3r);
\draw[-](3)--(3r);
\draw[-](4)--(4r);
\draw[-](5)--(6r);
\draw[-](6)--(6r);
\draw[-](7)--(7r);
\draw[-](8)--(9r);
\draw[-](9)--(9r);
%--------
\node at (1,0){\small $\core(f)$};
\end{tikzpicture}
}
\caption{For $f \in \m F_6(\ZZ)$, we have $f^{[l]}(0) > 0$ iff $l \in \{0,2,3,5\}=P_f$, so 
$\core(f) = \bigwedge_{l\in P_f} f^{[l]}$ is shown in the figure.}
\label{f: core}
\end{figure}

\begin{remark}\label{r:core}
Let $a \in \mathrm{F}_n(\ZZ)$ with $P_a \neq \emptyset$.
  \begin{enumerate}[label = \textup{(\arabic*)}]
      \item By Lemma~\ref{l:b-conjugation}, we have $P_a = \{l \in \ZZ_n : a(-l) > -l\}$.
      \item   For
      \(
      Q_a = \{l \in \ZZ : a^{[l]}(0) > 0 \} = \{-l \in \ZZ : a(l) > l \}
      \)
      we have $\core(a) = \bigwedge_{l \in Q_a} a^{[l]}$, by the $n$-periodicity of $a$.
      \item For all $l,x \in \ZZ$, if $\core(a)(x)>x$ and $a^{[l]}(0) >0$, then $a^{[l]}(x) >x$. 
  \end{enumerate}
\end{remark}

Part (1) of the following lemma will contribute toward the main result of this section, while part (3) will be useful in later sections in understanding the structure of $n$-periodic $\ell$-groups with linearly ordered group skeleton.

\begin{lemma}\label{l:core-lemma} 
    Let $a\in \mathrm{F}_n(\ZZ)$ be a strictly  positive idempotent.
    \begin{enumerate}[label = \textup{(\arabic*)}]
        \item The map $\core(a)$ is a strictly positive idempotent  with flat interval $[0,n-1]$.
        \item The set $P_a$ uniquely determines $a$ as a strictly  positive idempotent,  i.e., for each strictly positive idempotent $b \in \mathrm{F}_n(\ZZ)$: $a = b^{[k]}$ iff $P_a = P_b -_n k$. 
        \item  Let $a,b\in \mathrm{F}_n(\ZZ)$ be strictly positive idempotents such that for all $k\in \ZZ_n$, $a\neq b^{[k]}$. Then
    $\bigwedge_{l\in P_a} a^{[l]} > \bigwedge_{l\in P_a} b^{[l]} = 1$ or $\bigwedge_{l\in P_b} b^{[l]}  > \bigwedge_{l\in P_b} a^{[l]} = 1$.
    \end{enumerate} 
\end{lemma}

\begin{proof}
    (1) Since $a$ is strictly positive, there is a $k\in \ZZ$ such that $a(k) > k$, so we get $a^{[-k]}(0) = a(k) - k >0$ and, since $a$ is $n$-periodic, there is also such a $k\in \ZZ_n$; thus $P_a\neq \emptyset$ and $\core(a)$ exists. Moreover, since $a$ is positive, $a^{[l]}$ is positive for each $l$, so $\core(a)(s)=\bigwedge_{l\in P_a} a^{[l]}(s)\geq s$, for all $s$, hence $\core(a)$ is positive.  Indeed, $\core(a)$ is strictly positive, since, by definition, $\core(a)(0) > 0$. 
    We set $a' = \core(a)$ and let $k\in \ZZ$. Then $a'(k) = a^{[l]}(k)$ for some $l\in P_a$, so, since $a$ is idempotent and $x\mapsto{x}^{[l]}$ is an automorphism, by Lemma~\ref{l:double-l-auto}, we get $a^{[l]}(a^{[l]}(k)) = a^{[l]}(k)$.
    But also for each $m\in P_a$, $a^{[m]}(a^{[l]}(k)) \geq  a^{[l]}(k)$, by positivity. Hence
    \[
    a'(a'(k)) = \bigwedge_{m\in P_a} a^{[m]}(a^{[l]}(k)) = a^{[l]}(k) = a'(k),
    \]
    i.e., $a'$ is idempotent. Moreover, since $a$ is strictly positive, there exists a $k\in \ZZ$ with $a(k) = k$ and $a(k+1) > k+1$, thus by Lemma~\ref{l:b-conjugation}, $a^{[-k-1]}(0) = a(k+1) -k-1 > 0$ and $a^{[-k-1]}(-1) = a(k) - k-1 = -1$. Therefore $k +1 \in Q_a$ and, by Remark~\ref{r:core}(2), we obtain $a'(-1) = -1$. So $n$-periodicity yields $a'(n-1) = n-1$ and $[0,n-1]$ is a flat interval for $a'$. 

    (2) If $a,b\in \mathrm{F}_n(\ZZ)$ are strictly positive  idempotents with $P_b = P_a$, then for all $l\in \ZZ_n$, $a(l) > l$ iff $a^{[-l]}(0) > 0$ iff $b^{[-l]}(0) > 0$ iff $b(l) > l$; in other words: $a(l) = l$ iff $b(l) = l$. Since, by Remark~\ref{r:posidem}, positive idempotents are fully determined by their fixed points we get $a = b$. 

    (3)     
    Recall that  $P_c := \{l \in \ZZ_n : c^{[l]}(0) > 0\}$ and note that  $\bigwedge_{l\in P_a} b^{[l]}$ and $\bigwedge_{l\in P_b} a^{[l]}$ are positive, since $a,b$ are positive. If both $\bigwedge_{l\in P_a} b^{[l]}$ and $\bigwedge_{l\in P_b} a^{[l]}$ are 
    strictly-positive,
    then there exist $k_1,k_2 \in \ZZ$ such that $(\bigwedge_{l\in P_a} b^{[l]})(k_1) > k_1$ and $(\bigwedge_{l\in P_b} a^{[l]})(k_2) > k_2$, i.e., $(\bigwedge_{l\in P_a} b^{[l-k_1]})(0) > 0$ and $(\bigwedge_{l\in P_b} a^{[l-k_2]})(0) > 0$, by Lemma~\ref{l:b-conjugation}(2). Thus  
    $P_b -_n k_2 \subseteq P_a$ and $P_a -_n k_1 \subseteq P_b$. 
    Since $P_a$ and $P_b$ are finite, we get $P_a = P_b -_n k_2$, 
    so by (2), $a =  b^{[k_2]}$, a contradiction.
    Hence either $\bigwedge_{l\in P_a} b^{[l]}$ or $\bigwedge_{l\in P_b} a^{[l]}$  is not  strictly positive. Moreover, by (1), we know that $\bigwedge_{l\in P_a} a^{[l]}$ and $\bigwedge_{l\in P_b} b^{[l]}$ are strictly positive,  so  $\bigwedge_{l\in P_a} a^{[l]} > \bigwedge_{l\in P_a} b^{[l]} = 1$ or $\bigwedge_{l\in P_b} b^{[l]}  > \bigwedge_{l\in P_b} a^{[l]} = 1$. 
\end{proof}

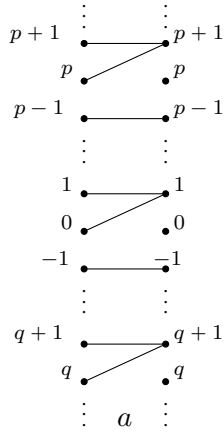
\begin{figure}[ht]\def\eq{=}
\centering
{\scriptsize
\begin{tikzpicture}
[scale=0.5]
%--------

\node[fill,draw,circle,scale=0.3,left](10) at (0,10){};
\node[right](10.1,0) at (-2.25,10.25){$p+1$};
\node[fill,draw,circle,scale=0.3,right](10r) at (2,10){};
\node[right](10.1) at (2.1,10.25){$p+1$};

\node[fill,draw,circle,scale=0.3,left](9) at (0,9){};
\node[right](9.1,0) at (-.9,9.25){$p$};
\node[fill,draw,circle,scale=0.3,right](9r) at (2,9){};
\node[right](9.1) at (2.1,9.25){$p$};

\node[fill,draw,circle,scale=0.3,left](8) at (0,8){};
\node[right](8.1,0) at (-2.15,8.25){$p-1$};
\node[fill,draw,circle,scale=0.3,right](8r) at (2,8){};
\node[right](8.1) at (2.1,8.25){$p-1$};

%\node[fill,draw,circle,scale=0.3,left](7) at (0,7){};
%\node[right](7.1,0) at (-.9,7.25){$2$};
%\node[fill,draw,circle,scale=0.3,right](7r) at (2,7){};
%\node[right](7.1) at (2.1,7.25){$2$};

\node[fill,draw,circle,scale=0.3,left](6) at (0,6){};
\node[right](6.1,0) at (-.9,6.25){$1$};
\node[fill,draw,circle,scale=0.3,right](6r) at (2,6){};
\node[right](6.1) at (2.1,6.25){$1$};

\node[fill,draw,circle,scale=0.3,left](5) at (0,5){};
\node[right](9.1,0) at (-.9,5.25){$0$};
\node[fill,draw,circle,scale=0.3,right](5r) at (2,5){};
\node[right](5.1) at (2.1,5.25){$0$};

\node[fill,draw,circle,scale=0.3,left](4) at (0,4){};
\node[right](8.1,0) at (-1.44,4.25){$-1$};
\node[fill,draw,circle,scale=0.3,right](4r) at (2,4){};
\node[right](4.1) at (1.55,4.25){$-1$};

%\node[fill,draw,circle,scale=0.3,left](3) at (0,3){};
%\node[right](3.1,0) at (-.9,3.25){$-2$};
%\node[fill,draw,circle,scale=0.3,right](3r) at (2,3){};
%\node[right](3.1) at (2.1,3.25){$-2$};

\node[fill,draw,circle,scale=0.3,left](2) at (0,2){};
\node[right](2.1,0) at (-2.15,2.25){$q+1$};
\node[fill,draw,circle,scale=0.3,right](2r) at (2,2){};
\node[right](2.1) at (2.1,2.25){$q+1$};

\node[fill,draw,circle,scale=0.3,left](1) at (0,1){};
\node[right](2.1,0) at (-.9,1.25){$q$};
\node[fill,draw,circle,scale=0.3,right](1r) at (2,1){};
\node[right](2.1) at (2.1,1.25){$q$};
%---------
\node at (10)[above=3pt]{$\vdots$};
\node at (10r)[above=3pt]{$\vdots$};
\node at (8)[below=-1pt]{$\vdots$};
\node at (8r)[below=-1pt]{$\vdots$};
\node at (4)[below=-1pt]{$\vdots$};
\node at (4r)[below=-1pt]{$\vdots$};
\node at (1)[below=-1pt]{$\vdots$};
\node at (1r)[below=-1pt]{$\vdots$};
%--------
\draw[-](1)--(2r);
\draw[-](2)--(2r);
\draw[-](4)--(4r);
\draw[-](5)--(6r); %0
\draw[-](6)--(6r);
\draw[-](8)--(8r);
\draw[-](9)--(10r);
\draw[-](10)--(10r);

%--------
\node at (1,0){\small $a$};
\end{tikzpicture}
}
\caption{The points $p$ and $q$ for $\core(a)$.}
\label{f: p-q}
\end{figure}

\begin{theorem}\label{t:minimal-idemp} 
    If $a\in \mathrm{F}_n(\ZZ)$ is a strictly positive idempotent, then $\core(a)$ is an $m$-atom, where $m = \per(a)$. 
\end{theorem}

\begin{proof}
    We may assume that $a$ has periodicity $n$, since if it has periodicity $m<n$, then it is an element of $ \mathrm{F}_m(\ZZ)\leq \mathrm{F}_n(\ZZ)$, so we can apply the  theorem to $ \mathrm{F}_m(\ZZ)$. Let $a' = \core(a)$.
    By Lemma~\ref{l:core-lemma}(1), $a'$ is idempotent and strictly positive with flat interval $[0,n-1]$.
    We will show that $a'$ is minimal among the strictly positive elements; since $a'(0)>0$, this amounts to showing that $a'$ is the identity on all elements that are not multiples of $n$. 
    Let $p \in [1,n]$ be minimal with $a'(p) > p$ and $q \in [-n, -1]$ be maximal with $a'(q) > q$; {see Figure~\ref{f: p-q}.}

    \underline{Claim 1.}   $p = -q$.
    
    If $p< -q$, then $q < p+q < 0$, and $a'(p+q) = p+q$, by the maximality of $q$, i.e., there exists an $l\in P_a$ such that $a^{[l]}(p+q) = p+q$. Since $l\in P_a$, we get $a^{[l]}(0) >0$, and since  $a'(q)>q$, Remark~\ref{r:core}(3) yields $a^{[l]}(q) > q$. So, by Lemma~\ref{l:b-conjugation}(2), we get $a^{[l-q]}(0)=a^{[l]}(0+q)-q > 0$ and $a^{[l-q]}(p) =a^{[l]}(p+q)-q= p+q-q= p$. Hence $l-q \in P_a$, yielding $a'(p) = p$,  a contradiction. 

    If $p > -q$, then $p > p + q > 0$, and $a'(p+q)=p+q$, by the minimality of $p$, i.e., there exists an $l\in P_a$ such that $a^{[l]}(p+q) = p+q$ and $a^{[l]}(p) > p$. So $a^{[l-p]}(0) > 0$ and $a^{[l-p]}(q) = q$. Hence $l-p \in P_a$, yielding $a'(q) = q$,  a contradiction.

    \underline{Claim 2.}    For all $k\in \ZZ$, $a(k) > k$ iff $a(k+p) > k+p$ or, equivalently, $a(k) = k$ iff $a(k+p) = k+p$.

    If $a(k)>k$, then $a^{[-k]}(0) = a(k) -k > 0$. Since we also have $a'(p) > p$, by Remark~\ref{r:core}(3), we get  $a(p+k) - k = a^{[-k]}(p) >p$, 
    yielding $a(p+k) > k+p$.
    Conversely, if $a(k+p) > k+p$, then $a^{[-k-p]}(0) = a(k+p) -k-p > 0$, so, by Claim 1, $a(k)- k-p = a^{[-k-p]}(-p) =a^{[-k-p]}(q) > q = -p$, yielding $a(k) > k$. Thus Claim~2 is shown.
    
    By Remark~\ref{r:posidem}, positive idempotents are uniquely determined by their fixed points, so Claim~2 yields that $a$ is $p$-periodic. Therefore, $p=n$, so $a'(k) = k$ for each $k \in [1,n-1]$ and $a'(0) = 1$, yielding that $a'$ is an $n$-atom.
\end{proof}

By Theorem~\ref{t:minimal-idemp} and  Proposition~\ref{p: n-cover generates} we obtain the following corollary, which gets us closer to the generation result of this section.

\begin{corollary}\label{c:posidem-gen}
    If $a\in \mathrm{F}_n(\ZZ)$ is a strictly positive idempotent of periodicity $k$, then $\langle a \rangle = \mathrm{F}_k(\ZZ)$.
\end{corollary}

Note that Theorem~\ref{t:minimal-idemp} does not generalize to arbitrary strictly positive flat elements. Indeed, the map $f$ in Figure~\ref{f: core} is strictly positive and flat of periodicity $6$, but $\core(f)$ has periodicity $3$. The reason for this is that the fixed points of $f$ are ``distributed'' $3$-periodically forcing $f^2 = f^{3}$ to be $3$-periodic.
In general, if $a\in \mathrm{F}_n(\ZZ)$ is a positive flat element, then, by Lemma~\ref{l:flat-Fn}, {we know that} $a^{n-1} = a^{n}$, so the map $a^{n-1}$ is idempotent 
(actually $a^{n-1}$ is  the least idempotent above $a$); we call the periodicity of $a^{n-1}$ the \emph{final periodicity} of $a$.

The next lemma shows that every element of  periodicity $n$ and final periodicity $k<n$, generates an element of strictly larger final periodicity $k'$ ($k<k'\leq n$); see Figure~\ref{f: final-per} for the idea of the proof.

\begin{figure}[ht]\def\eq{=}
\centering
{\scriptsize
\begin{tikzpicture}
[scale=0.43]
%--------
\node[fill,draw,circle,scale=0.3,left](9) at (0,9){};
\node[right](9.1,0) at (-.9,9.25){$7$};
\node[fill,draw,circle,scale=0.3,right](9r) at (2,9){};
\node[right](9.1) at (2.1,9.25){$7$};

\node[fill,draw,circle,scale=0.3,left](8) at (0,8){};
\node[right](8.1,0) at (-.9,8.25){$6$};
\node[fill,draw,circle,scale=0.3,right](8r) at (2,8){};
\node[right](8.1) at (2.1,8.25){$6$};

\node[fill,draw,circle,scale=0.3,left](7) at (0,7){};
\node[right](7.1,0) at (-.9,7.25){$5$};
\node[fill,draw,circle,scale=0.3,right](7r) at (2,7){};
\node[right](7.1) at (2.1,7.25){$5$};

\node[fill,draw,circle,scale=0.3,left](6) at (0,6){};
\node[right](6.1,0) at (-.9,6.25){$4$};
\node[fill,draw,circle,scale=0.3,right](6r) at (2,6){};
\node[right](6.1) at (2.1,6.25){$4$};

\node[fill,draw,circle,scale=0.3,left](5) at (0,5){};
\node[right](9.1,0) at (-.9,5.25){$3$};
\node[fill,draw,circle,scale=0.3,right](5r) at (2,5){};
\node[right](5.1) at (2.1,5.25){$3$};

\node[fill,draw,circle,scale=0.3,left](4) at (0,4){};
\node[right](8.1,0) at (-.9,4.25){$2$};
\node[fill,draw,circle,scale=0.3,right](4r) at (2,4){};
\node[right](4.1) at (2.1,4.25){$2$};

\node[fill,draw,circle,scale=0.3,left](3) at (0,3){};
\node[right](3.1,0) at (-.9,3.25){$1$};
\node[fill,draw,circle,scale=0.3,right](3r) at (2,3){};
\node[right](3.1) at (2.1,3.25){$1$};

\node[fill,draw,circle,scale=0.3,left](2) at (-0.06,2){};
\node[right](2.1,0) at (-.9,2.25){$0$};
\node[fill,draw,circle,scale=0.3,right](2r) at (2.06,2){};
\node[right](2.1) at (2.1,2.25){$0$};
%---------
\node at (9)[above=3pt]{$\vdots$};
\node at (9r)[above=3pt]{$\vdots$};
\node at (2)[below=-1pt]{$\vdots$};
\node at (2r)[below=-1pt]{$\vdots$};
%--------
\draw[-](2)--(4r);
\draw[-](3)--(4r);
\draw[-](4)--(4r);
\draw[-](5)--(6r);
\draw[-](6)--(7r);
\draw[-](7)--(7r);

%--------
\node at (1,-0.1){\small $a$};
\end{tikzpicture}
\quad
\begin{tikzpicture}
[scale=0.43]
%--------
\node[fill,draw,circle,scale=0.3,left](9) at (0,9){};
\node[right](9.1,0) at (-.9,9.25){$7$};
\node[fill,draw,circle,scale=0.3,right](9r) at (2,9){};
\node[right](9.1) at (2.1,9.25){$7$};

\node[fill,draw,circle,scale=0.3,left](8) at (0,8){};
\node[right](8.1,0) at (-.9,8.25){$6$};
\node[fill,draw,circle,scale=0.3,right](8r) at (2,8){};
\node[right](8.1) at (2.1,8.25){$6$};

\node[fill,draw,circle,scale=0.3,left](7) at (0,7){};
\node[right](7.1,0) at (-.9,7.25){$5$};
\node[fill,draw,circle,scale=0.3,right](7r) at (2,7){};
\node[right](7.1) at (2.1,7.25){$5$};

\node[fill,draw,circle,scale=0.3,left](6) at (0,6){};
\node[right](6.1,0) at (-.9,6.25){$4$};
\node[fill,draw,circle,scale=0.3,right](6r) at (2,6){};
\node[right](6.1) at (2.1,6.25){$4$};

\node[fill,draw,circle,scale=0.3,left](5) at (0,5){};
\node[right](9.1,0) at (-.9,5.25){$3$};
\node[fill,draw,circle,scale=0.3,right](5r) at (2,5){};
\node[right](5.1) at (2.1,5.25){$3$};

\node[fill,draw,circle,scale=0.3,left](4) at (0,4){};
\node[right](8.1,0) at (-.9,4.25){$2$};
\node[fill,draw,circle,scale=0.3,right](4r) at (2,4){};
\node[right](4.1) at (2.1,4.25){$2$};

\node[fill,draw,circle,scale=0.3,left](3) at (0,3){};
\node[right](3.1,0) at (-.9,3.25){$1$};
\node[fill,draw,circle,scale=0.3,right](3r) at (2,3){};
\node[right](3.1) at (2.1,3.25){$1$};

\node[fill,draw,circle,scale=0.3,left](2) at (-0.06,2){};
\node[right](2.1,0) at (-.9,2.25){$0$};
\node[fill,draw,circle,scale=0.3,right](2r) at (2.06,2){};
\node[right](2.1) at (2.1,2.25){$0$};
%---------
\node at (9)[above=3pt]{$\vdots$};
\node at (9r)[above=3pt]{$\vdots$};
\node at (2)[below=-1pt]{$\vdots$};
\node at (2r)[below=-1pt]{$\vdots$};
%--------
\draw[-](2)--(4r);
\draw[-](3)--(4r);
\draw[-](4)--(4r);
\draw[-](5)--(7r);
\draw[-](6)--(7r);
\draw[-](7)--(7r);

%--------
\node at (1,0.07){\small $a^2$};
\end{tikzpicture}
\quad
\begin{tikzpicture}
[scale=0.43]
%--------
\node[fill,draw,circle,scale=0.3,left](9) at (0,9){};
\node[right](9.1,0) at (-.9,9.25){$7$};
\node[fill,draw,circle,scale=0.3,right](9r) at (2,9){};
\node[right](9.1) at (2.1,9.25){$7$};

\node[fill,draw,circle,scale=0.3,left](8) at (0,8){};
\node[right](8.1,0) at (-.9,8.25){$6$};
\node[fill,draw,circle,scale=0.3,right](8r) at (2,8){};
\node[right](8.1) at (2.1,8.25){$6$};

\node[fill,draw,circle,scale=0.3,left](7) at (0,7){};
\node[right](7.1,0) at (-.9,7.25){$5$};
\node[fill,draw,circle,scale=0.3,right](7r) at (2,7){};
\node[right](7.1) at (2.1,7.25){$5$};

\node[fill,draw,circle,scale=0.3,left](6) at (0,6){};
\node[right](6.1,0) at (-.9,6.25){$4$};
\node[fill,draw,circle,scale=0.3,right](6r) at (2,6){};
\node[right](6.1) at (2.1,6.25){$4$};

\node[fill,draw,circle,scale=0.3,left](5) at (0,5){};
\node[right](9.1,0) at (-.9,5.25){$3$};
\node[fill,draw,circle,scale=0.3,right](5r) at (2,5){};
\node[right](5.1) at (2.1,5.25){$3$};

\node[fill,draw,circle,scale=0.3,left](4) at (0,4){};
\node[right](8.1,0) at (-.9,4.25){$2$};
\node[fill,draw,circle,scale=0.3,right](4r) at (2,4){};
\node[right](4.1) at (2.1,4.25){$2$};

\node[fill,draw,circle,scale=0.3,left](3) at (0,3){};
\node[right](3.1,0) at (-.9,3.25){$1$};
\node[fill,draw,circle,scale=0.3,right](3r) at (2,3){};
\node[right](3.1) at (2.1,3.25){$1$};

\node[fill,draw,circle,scale=0.3,left](2) at (-0.06,2){};
\node[right](2.1,0) at (-.9,2.25){$0$};
\node[fill,draw,circle,scale=0.3,right](2r) at (2.06,2){};
\node[right](2.1) at (2.1,2.25){$0$};
%---------
\node at (9)[above=3pt]{$\vdots$};
\node at (9r)[above=3pt]{$\vdots$};
\node at (2)[below=-1pt]{$\vdots$};
\node at (2r)[below=-1pt]{$\vdots$};
%--------
\draw[-](2)--(3r);
\draw[-](3)--(3r);
\draw[-](4)--(6r);
\draw[-](5)--(6r);
\draw[-](6)--(6r);
\draw[-](7)--(9r);
\draw[-](8)--(9r);
\draw[-](9)--(9r);
%--------
\node at (1,0.2){\small $\top :=(a^2)^{[2]}$};
\end{tikzpicture}
\quad
\begin{tikzpicture}
[scale=0.43]
%--------
\node[fill,draw,circle,scale=0.3,left](9) at (0,9){};
\node[right](9.1,0) at (-.9,9.25){$7$};
\node[fill,draw,circle,scale=0.3,right](9r) at (2,9){};
\node[right](9.1) at (2.1,9.25){$7$};

\node[fill,draw,circle,scale=0.3,left](8) at (0,8){};
\node[right](8.1,0) at (-.9,8.25){$6$};
\node[fill,draw,circle,scale=0.3,right](8r) at (2,8){};
\node[right](8.1) at (2.1,8.25){$6$};

\node[fill,draw,circle,scale=0.3,left](7) at (0,7){};
\node[right](7.1,0) at (-.9,7.25){$5$};
\node[fill,draw,circle,scale=0.3,right](7r) at (2,7){};
\node[right](7.1) at (2.1,7.25){$5$};

\node[fill,draw,circle,scale=0.3,left](6) at (0,6){};
\node[right](6.1,0) at (-.9,6.25){$4$};
\node[fill,draw,circle,scale=0.3,right](6r) at (2,6){};
\node[right](6.1) at (2.1,6.25){$4$};

\node[fill,draw,circle,scale=0.3,left](5) at (0,5){};
\node[right](9.1,0) at (-.9,5.25){$3$};
\node[fill,draw,circle,scale=0.3,right](5r) at (2,5){};
\node[right](5.1) at (2.1,5.25){$3$};

\node[fill,draw,circle,scale=0.3,left](4) at (0,4){};
\node[right](8.1,0) at (-.9,4.25){$2$};
\node[fill,draw,circle,scale=0.3,right](4r) at (2,4){};
\node[right](4.1) at (2.1,4.25){$2$};

\node[fill,draw,circle,scale=0.3,left](3) at (0,3){};
\node[right](3.1,0) at (-.9,3.25){$1$};
\node[fill,draw,circle,scale=0.3,right](3r) at (2,3){};
\node[right](3.1) at (2.1,3.25){$1$};

\node[fill,draw,circle,scale=0.3,left](2) at (-0.06,2){};
\node[right](2.1,0) at (-.9,2.25){$0$};
\node[fill,draw,circle,scale=0.3,right](2r) at (2.06,2){};
\node[right](2.1) at (2.1,2.25){$0$};
%---------
\node at (9)[above=3pt]{$\vdots$};
\node at (9r)[above=3pt]{$\vdots$};
\node at (2)[below=-1pt]{$\vdots$};
\node at (2r)[below=-1pt]{$\vdots$};
%--------
\draw[-](2)--(6r);
\draw[-](3)--(6r);
\draw[-](4)--(6r);
\draw[-](5)--(6r);
\draw[-](6)--(9r);
\draw[-](7)--(9r);
%\draw[-](8)--(9r);
%\draw[-](9)--(9r);
%--------
\node at (1,-0.01){\small $d:=\top a$};
\end{tikzpicture}
\quad
\begin{tikzpicture}
[scale=0.43]
%--------
\node[fill,draw,circle,scale=0.3,left](9) at (0,9){};
\node[right](9.1,0) at (-.9,9.25){$7$};
\node[fill,draw,circle,scale=0.3,right](9r) at (2,9){};
\node[right](9.1) at (2.1,9.25){$7$};

\node[fill,draw,circle,scale=0.3,left](8) at (0,8){};
\node[right](8.1,0) at (-.9,8.25){$6$};
\node[fill,draw,circle,scale=0.3,right](8r) at (2,8){};
\node[right](8.1) at (2.1,8.25){$6$};

\node[fill,draw,circle,scale=0.3,left](7) at (0,7){};
\node[right](7.1,0) at (-.9,7.25){$5$};
\node[fill,draw,circle,scale=0.3,right](7r) at (2,7){};
\node[right](7.1) at (2.1,7.25){$5$};

\node[fill,draw,circle,scale=0.3,left](6) at (0,6){};
\node[right](6.1,0) at (-.9,6.25){$4$};
\node[fill,draw,circle,scale=0.3,right](6r) at (2,6){};
\node[right](6.1) at (2.1,6.25){$4$};

\node[fill,draw,circle,scale=0.3,left](5) at (0,5){};
\node[right](9.1,0) at (-.9,5.25){$3$};
\node[fill,draw,circle,scale=0.3,right](5r) at (2,5){};
\node[right](5.1) at (2.1,5.25){$3$};

\node[fill,draw,circle,scale=0.3,left](4) at (0,4){};
\node[right](8.1,0) at (-.9,4.25){$2$};
\node[fill,draw,circle,scale=0.3,right](4r) at (2,4){};
\node[right](4.1) at (2.1,4.25){$2$};

\node[fill,draw,circle,scale=0.3,left](3) at (0,3){};
\node[right](3.1,0) at (-.9,3.25){$1$};
\node[fill,draw,circle,scale=0.3,right](3r) at (2,3){};
\node[right](3.1) at (2.1,3.25){$1$};

\node[fill,draw,circle,scale=0.3,left](2) at (-0.06,2){};
\node[right](2.1,0) at (-.9,2.25){$0$};
\node[fill,draw,circle,scale=0.3,right](2r) at (2.06,2){};
\node[right](2.1) at (2.1,2.25){$0$};
%---------
\node at (9)[above=3pt]{$\vdots$};
\node at (9r)[above=3pt]{$\vdots$};
\node at (2)[below=-1pt]{$\vdots$};
\node at (2r)[below=-1pt]{$\vdots$};
%--------
\draw[-](2)--(5r);
\draw[-](3)--(5r);
\draw[-](4)--(5r);
\draw[-](5)--(5r);
\draw[-](6)--(8r);
\draw[-](7)--(8r);
%\draw[-](8)--(9r);
%\draw[-](9)--(9r);
%--------
\node at (1,0.17){\small $c:= \g(d)^{-1}d$};
\end{tikzpicture}
}
\caption{The map $a\in \mathrm{F}_6(\ZZ)$ has final periodicity $3$, but we have $a(3) < a(0)+3$. Therefore we can generate the positive flat map $c$ with final periodicity $6$.}
\label{f: final-per}
\end{figure}
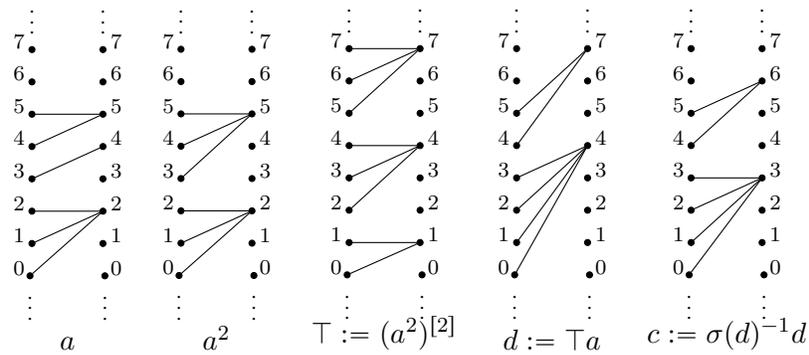

\begin{lemma}\label{l:final-period}
    Let $a\in \mathrm{F}_n(\ZZ)$ be a positive flat element of periodicity $n$ and $b \in \langle a \rangle$ a positive flat element of final periodicity $k\neq n$. Then there exists a positive flat element $c\in \langle a\rangle$ of final periodicity $k'>k$ (and $k'\leq n$). 
\end{lemma}

\begin{proof}
     By Lemma~\ref{l:k-interval},  there exists an $l\in \ZZ$ such that $a(l+k) < a(l) + k$.
     Let $\top$ be the function such that $\top(x)=a(l) + k-1$, for all $a(l)\leq x \leq a(l) + k-1$, and extended $k$-periodically on all of $\ZZ$; note that $\top$ is positive, flat  and a
     maximal idempotent of periodicity $k$, hence $\top \in \mathrm{F}_k(\ZZ)$. 
      Since $a(l+k) \leq a(l)+k -1 $, for $d := \top a$ we have $d(l) =\top(a(l)) = a(l) + k-1$ and 
      \[
      d(l) \leq d(l+k)=\top (a(l+k))\leq \top (a(l)+k -1)=a(l) + k-1,
      \]
      so $\lvert d^{-1}[a(l) + k-1]\rvert \geq k+1$. By Corollary~\ref{c:flat-g} the element $c := \g(d)^{-1}d$ is positive and flat, and $\lvert c^{-1}[\g(d)^{-1}(a(l) + k-1)]\rvert \geq k+1$; hence $c$ has final periodicity $k'>k$, since for a $k$-periodic map on $\ZZ$ the size of the inverse image of a point is bounded by $k$. Moreover, since $b$ has final periodicity $k$, by Corollary~\ref{c:posidem-gen} we get $\mathrm{F}_k(\ZZ) \subseteq \langle b^{n-1} \rangle \subseteq  \langle b\rangle$, so $\top \in \mathrm{F}_k(\ZZ) \subseteq \langle b \rangle \subseteq \langle a \rangle$, thus $c = \g(\top a)^{-1} \top a\in \langle a \rangle$.
\end{proof}

We are now ready to prove the main result of this section, which will also be used later in the proof of Theorem~\ref{t:Fn-FSI-local}.

\begin{theorem}\label{t:Fn-generation}
    If $n\neq 1$ and $a\in \mathrm{F}_n(\ZZ)$ is of periodicity $n$, then $\langle a \rangle = \m{F}_n(\ZZ)$. 
\end{theorem}

\begin{proof}
    In view of Corollary~\ref{c:flat-g}, by considering $a\g(a)^{-1}$ we may assume that $a$ is strictly positive and flat, and let $k>1$ be the final periodicity of $a$. 
    We set $c_0 :=a$ and $k_0 := k$, and we  define inductively elements $c_i \in \langle a \rangle$, for $i \in \NN$, with increasing final periodicity $k_i$, provided that $k_{i-1}<n$. If $c_i$ and $k_i$ are defined and $k_i < n$, then, by Lemma~\ref{l:final-period}, there is a positive flat $c_{i+1}\in \langle a \rangle$ with final periodicity $k_{i+1}>k_i$ (where $k_{i+1}\leq n$). 
    Since there are only finitely many divisors of $n$, this process terminates 
    and we get a (strictly) positive flat element $c\in \langle a \rangle$ of final periodicity $n$.
     Thus  $c^{n-1}$ is a strictly positive idempotent and $\per(c^{n-1})=n$, so  $ \langle c^{n-1} \rangle = \m{F}_n(\ZZ)$, by Corollary~\ref{c:posidem-gen}. Since $a\in \mathrm{F}_n(\ZZ)$, we get $\langle a \rangle \subseteq \m{F}_n(\ZZ) = \langle c^{n-1} \rangle\subseteq \langle c \rangle \subseteq \langle a \rangle$.
\end{proof}

We now show how elements of different periodicities interact to produce elements of smaller periodicities.

\begin{proposition} \label{p: maxper}
Let $a,b \in \mathrm{F}(\ZZ)$ such that $a$ is an $n$-atom and $b$ is an $m$-atom with $m \nmid n$ and $n \nmid m$. Then there exist $s, t \in \ZZ$ such that $a^{[s]} \jn b^{[t]}$ has periodicity $\lcm\{n,m\}$.  
\end{proposition}

\begin{proof}
Let $k:=\lcm\{n,m\}$. Note that by Proposition~\ref{p:n-cover} the heights of $a$ and of $b$ are either $0$ or $1$ and in each $n$-element interval there exists exactly one point where $a$ has height $1$, and likewise for $b$.
 Since 
 $a\neq 1$, 
 there exists $s$ such that $a^{[s]}(0)=1$ and the height of $a^{[s]}$ is $0$ in all of $[1,n-1]$.
 Also,  since $b\neq 1$, there exists $t$ such that $b^{[t]}(0)=1$ and the height of $b^{[t]}$ is $0$ in all of $[1,m-1]$.
 Moreover, $a^{[s]}$ has period $n$ and $b^{[t]}$ has period $m$.
 Hence, the  height of the element $c:=(a^{[s]} \jn b^{[t]})$ is equal to $1$ at precisely the multiples of $n$ and the multiples of $m$. Therefore, the only possible periods of $c$ are the multiples of $n$ and the multiples of $m$. In particular, $c$ is $k$-periodic.
 If the period is $pm$, then $-n+1=c(-n)=c^{[pm]}(-n)=c(-n-pm)+pm$, so $c(-n-pm)=-n-pm+1$.  Hence, $-n-pm$ is either divisible by $m$ or by $n$. Since $m \nmid n$, we get $n \mid pm$,  yielding $k \mid pm$, so $pm=k$. Likewise, if the period is $pn$, then we get $pn=k$. 
\end{proof}

\begin{corollary}
If $a,b \in \mathrm{F}(\ZZ)$  are elements with periodicities $n$ and $m$, respectively, 
then they generate an element of periodicity $\lcm\{n,m\}$.     
\end{corollary}

\begin{proof}
 If $m \mid n$ or $n \mid m$, then $\lcm\{n,m\}$ is $n$ or $m$, so the desired element of periodicity $\lcm\{n,m\}$ can be taken to be $a$ or $b$, respectively.
 Otherwise, $n \nmid m$ and $m \nmid n$. By Theorem~\ref{t:Fn-generation}, $a$ generates $\m F_n(\ZZ)$ and $b$ generates $\m{F}_m(\ZZ)$; let $c \in \mathrm{F}_n(\ZZ) = \langle a \rangle$ be an $n$-atom and $d \in \mathrm{F}_m(\ZZ) = \langle b \rangle$ be an $m$-atom. Then, by Proposition~\ref{p: maxper}, there exist $s, t$ such that $c^{[s]} \jn d^{[t]}$ has periodicity $\lcm\{n,m\}$. 
\end{proof}

We conclude the section by showing how from a positive idempotent $a$ we can generate a positive idempotent $\slt(a)$  that has maximal height equal to $1$. This lemma will be useful in analyzing $n$-periodic $\ell$-pregroups with a totally ordered skeleton. We define the term $\slt(x) := x^{\ell\ell\ell}x \jn 1$; note that, in the next lemma, the strictly positive idempotents produced are given uniformly by the single term $\slt$.

\begin{lemma}\label{l:idempotent-slope1}
 If $a \in \mathrm{F}(\ZZ)$ is a strictly positive idempotent, then $\slt(a)$ is a strictly positive idempotent with  maximal height $1$. If additionally $a \in \mathrm{F}_n(\ZZ)$ for some $n\geq 1$, then $\h(\slt(a)) = \cn{s}$. 
\end{lemma}

\begin{proof}
  Since $a$ is positive and idempotent, we have $a a^\ell=a$ and $a^\ell a=a^\ell$, by Remark~\ref{r: xxl}; so for all $x \in \ZZ$, we have $a^\ell(x)= a^\ell(a(x)) =\min a^{-1}[a(x)]$, hence also $a^\ell(x)\leq x \leq a(x)$.
   We will establish two auxiliary facts:
        \begin{enumerate}[label = \textup{(\roman*)}]
            \item If $a(x) \neq a^\ell(x)$, then $(a^{\ell\ell\ell}a)(x) = a^\ell(x) + 1$, so  $\slt(a)(x)\leq x+1$.
            \item If $a(x) = a^\ell(x)$, then $(a^{\ell\ell\ell}a)(x) \leq x$, so $\slt(a)(x)=x$.
            \end{enumerate}
    For (i), if $a(x) \neq a^\ell(x)$ then
    $a^\ell(x) < a(x)$, hence $a^\ell(x) \leq  a(x)-1$; 
    so $a(x) =a(a^\ell(x))\leq a(a(x)- 1) \leq a(a(x))=a(x)$, thus $a(a(x)- 1) = a(x)$.
    Therefore, 
    \[
    a^{\ell\ell\ell}(a(x)) = a^\ell(a(x)-1) + 1 = a^\ell a(a(x)-1) + 1 = a^\ell a(x) + 1= a^\ell(x) + 1.
    \]
    Finally, since $a^\ell(x)\leq x$, we get $(a^{\ell\ell\ell}a)(x) = a^\ell(x) + 1\leq x +1$.
        
    For (ii), note that if  $a(x) = a^\ell(x)$, then $a^\ell(x)\leq x \leq a(x)$ yields $a(x) = x$.
   Hence, since $a^\ell \leq 1$,
    \[
    a^{\ell\ell\ell}(a(x)) = a^{\ell\ell\ell}(x) = a^{\ell}(x-1) + 1 \leq x-1 + 1 = x.
    \]
    From (i) and (ii), it follows that $\slt(a)$ has height no bigger than $1$. Moreover, since  $a$ is strictly positive and idempotent, we have $a \neq a^\ell$, so there is an $x \in \ZZ$ with $a(x) \neq a^\ell(x)$. Since, $a^\ell(x)\leq x \leq a(x)$, for $y:=a^\ell(x)$, we have $a(y)=a(x)$ and $a^\ell(y)=a^\ell(x)$, hence $a(y) \neq a^\ell(y)$. So, by (i), we get $(a^{\ell\ell\ell}a)(y) =(a^{\ell\ell\ell}a)(x)= a^\ell(x) + 1=y+1$, hence  $\slt(a)$ has height $1$ at $y$.
    It follows that $\slt(a)$ is strictly positive and has maximal height 1, 
    So, if $a$ is $n$-periodic, then  $\h(\slt(a)) = \cn{s}$. 
    
    Now to see that $\slt(a)$ is idempotent, let $x\in \ZZ$. If $\slt(a)(x) \neq x$, then, by (ii), we have $a(x) \neq a^\ell(x)$, so by (i) we get $(a^{\ell\ell\ell}a)(x)=a^\ell(x) + 1$.
    So $x<\slt(a)(x)=(a^{\ell\ell\ell}a)(x) \jn x= (a^\ell(x) + 1) \jn x$, hence $x < a^\ell(x) + 1$. Since $a^\ell(x) \leq x$, we get $x = a^\ell(x)$.
    We will now show that $\slt(a)(a^\ell(x)+1) = a^\ell(x)+1$ and
    therefore conclude that $\slt(a)$ is idempotent.
    Let $y:=a^\ell(x)+1$ and note that $a(y)=a(x)$ and $a^\ell(y)=a^\ell(x)$. So, by (i) applied to $y$ we get  $(a^{\ell\ell\ell}a)(y)  = a^\ell(y) + 1= a^\ell(x) + 1=y$, thus $\slt(a)(y) = y$, as desired.
\end{proof}

\section{Lexicographic products}\label{s: lex}
In this section we define lexicographic products in the setting of $\ell$-pregroups and show that $\vr(\m{F}_n(\ZZ))$ is closed under lexicographic products with totally ordered abelian $\ell$-groups, a fact that will be important in showing that algebras satisfying the proposed axiomatization are in $\vr(\m{F}_n(\ZZ))$.

Let $\m{H}$ be a totally ordered $\ell$-group and $\m{L}$ an $\ell$-pregroup. The \emph{lexicographic product} of $\m{H}$ and $\m{L}$ is the algebra $\m{H}\overrightarrow{\times} \m{L} = ( H\times L, \mt, \jn,\cdot, (1,1), {}^\ell, {}^r) $, where $\cdot$, ${}^\ell$, and ${}^r$ are defined coordinatewise and the lattice-order is defined by
\[
(h_1,a_1) \leq (h_2,a_2) \iff h_1 < h_2 \text{ or } (h_1=h_2 \text{ and } a_1 \leq a_2).
\]
If $\m L$ is an $\ell$-group, then we obtain the usual lexicographic product in the setting of $\ell$-groups.

\begin{lemma}
    If $\m H$ is a totally ordered $\ell$-group and $\m{L}$ is an ($n$-periodic) $\ell$-pregroup, then the lexicographic product $\m{H}\overrightarrow{\times} \m{L}$ is an ($n$-periodic) $\ell$-pregroup. 
\end{lemma}

\begin{proof}
Since $\m H$ is totally ordered, the following formulas show that the join and the meet always exist, so  $\m{H}\overrightarrow{\times} \m{L}$ is a lattice. 
\[
(a,b) \jn (c,d) = \left\{ \begin{array}{ll}
   (a, b \jn d) & \textrm{if } a=c \\
   (a , b ) & \textrm{if } a > c \\
   (c , d ) & \textrm{if } a < c \\
   \end{array} \right. \; \; (a,b) \mt (c,d) = \left\{ \begin{array}{ll}
   (a, b \mt d) & \textrm{if } a=c \\
   (c , d) & \textrm{if } a > c \\
   (a , b) & \textrm{if } a < c \\
   \end{array} \right.
\]
To show that multiplication is order-preserving, let $(h_1,a_1), (h_2,a_2), (h,a)$ be elements of $\m{H}\overrightarrow{\times} \m{L}$ with
$(h_1,a_1) \leq (h_2,a_2)$, i.e., $h_1 < h_2$ or $(h_1=h_2 \text{ and } a_1 \leq a_2)$; we want to show that $(h_1h,a_1a) \leq (h_2h,a_2a)$ and $(hh_1,aa_1) \leq (hh_2,aa_2)$.
If $h_1 < h_2$, then $h_1h < h_2h$ and $hh_1 < hh_2$, so both inequalities hold. If $(h_1=h_2 \text{ and } a_1 \leq a_2)$, then  $h_1h = h_2h$ and $hh_1 = hh_2$; also, $a_1a \leq a_2a$ and $aa_1 \leq aa_2$, so the inequalities hold in this case, as well.

Finally, for $(h,a)$  in $\m{H}\overrightarrow{\times} \m{L}$, we have 
\[
(h,a)^\ell(h,a)=(h^{-1},a^\ell)(h,a)=(1,a^\ell a) \leq (1, 1)\leq (1,a a^\ell)
=(h,a)(h,a)^\ell
\]
and
\[
(h,a)(h,a)^r=(h,a)(h^{-1},a^r)=(1,a a^r)\leq (1,1) \leq (1,a^r a) 
=(h,a)^r(h,a).
\]
So,  $\m{H}\overrightarrow{\times} \m{L}$ is an $\ell$-pregroup. Moreover, since $(h,a)^{\ell\ell}=((h^{-1})^{-1},a^{\ell\ell})=(h,a^{\ell\ell})$, if $\m{L}$ is an $n$-periodic then $(h,a)^{[n]}=(h,a^{[n]})=(h,a)$, for all $(h,a)$, hence $\m{H}\overrightarrow{\times} \m{L}$ is $n$-periodic.
\end{proof}

The next lemma gives a characterization of inner lexicographic products.

\begin{lemma}\label{l:inner-lex}
    Let $\m{L}$ be an $\ell$-pregroup and $\m{H},\m{M} \leq \m{L}$ such that $\m{H}$ is a totally ordered $\ell$-group. Then $\f \colon \m{H} \overrightarrow{\times} \m{M} \to \m{L}$, where $\f(h,m) = hm$, is an isomorphism if and only if the following hold:
    \begin{enumerate}[label = \textup{(\roman*)}]
        \item For all $h\in H$, $m\in M$, $hm = mh$.
        \item For every $a \in L$, there exist unique $h\in H$, $m\in M$ such that $a = hm$.
        \item For all $h_1,h_2\in H$ and $m_1,m_2\in M$, we have $h_1m_1 < h_2m_2$ iff $h_1 < h_2$ or ($h_1 = h_2$ and $m_1 < m_2$).
    \end{enumerate}
\end{lemma}

\begin{proof}
    For the non-trivial (right-to-left) direction note that  $\f(1,1) = 1$ and for all $h_1,h_2 \in H$ and $m_1,m_2 \in M$, by (i) we have
    \[
    \f(h_1,m_1)\f(h_2,m_2) = h_1m_1h_2m_2 = h_1h_2m_1m_2 = \f(h_1h_2,m_1m_2).
    \]
    Moreover, by (ii), $\f$ is bijective, and, by (iii), it preserves the  lattice-order. Thus $\f$ is an isomorphism of lattice-ordered monoids. To show that $\f$ preserves the inverses, note that for $h\in H$, $m\in M$,  by (i) we get
    \[
    \f((h,m)^{\ell}) = \f(h^{-1},m^{\ell}) = h^{-1}m^{\ell} = m^\ell h^{-1} = (hm)^\ell = \f(h,m)^\ell,
    \]
    and similarly for ${}^r$.
\end{proof}
Note that if $\m{M}$ is also an $\ell$-group in Lemma~\ref{l:inner-lex}, then part (ii) is equivalent to demanding that $H\cap M = \{1\}$.

Recall that a (positive) element $g$ in an $\ell$-pregroup $\m L$ is called a \emph{strong order unit} if  for each $a\in L$ there exists $n \geq 1$ such that $a \leq g^n$. In the next two lemmas we generalize results of \cite{Han} about lexicographic products in varieties of $\ell$-groups to varieties of $\ell$-pregroups.

\begin{lemma}\label{l:closure-lexZ}
    If an $\ell$-pregroup $\m{L}$ contains a central invertible element $g$ that is a strong order unit, then $\vr(\m{L})$ is closed under lexicographic products with $\ZZ$.
\end{lemma}

\begin{proof}
    Note that  for any non-principal ultrafilter $U$ over $\NN$, the subalgebra of the ultrapower $\m{L}^\NN/U$ generated by all the constant sequences $\overline{a}:=[(a, a, a,\dots)]$, for each $a\in L$, together with the element $\hat{g}:=[(g,g^2,g^3,\dots)]$ is isomorphic to $\ZZ \overrightarrow{\times}\m{L}$. 
    Indeed, the constants generate a subalgebra isomorphic to $\m L$, while the element  $\hat{g}$ is greater than all constants and generates a subalgebra isomorphic to $\ZZ$. Moreover, since $g$ is central, $\hat{g}$ is also central, so the set of all products of the form $\overline{a}\hat{g}^n$, where $a \in L$ and $n \in \ZZ$, is a submonoid of $\m{L}^\NN/U$ and is closed under $^\ell$ and $^r$. Also, for two products, using Lemma~\ref{l: InRL}, we have $\overline{a}\hat{g}^n \leq \overline{b}\hat{g}^m$ iff  $\overline{b^\ell a} \leq \hat{g}^{m-n}$ iff $m-n>0$ or ($m-n=0$ and $b^\ell a \leq 1$) iff $\hat{g}^{m-n}>1$ or ($\hat{g}^{m-n}=1$ and $\overline{b^\ell a} \leq 1$) iff $\hat{g}^{n}<\hat{g}^{m}$ or ($\hat{g}^{n}=\hat{g}^{m}$ and $\overline{a} \leq \overline{b}$). As a result, the set of all such products is a subalgebra and its order is the lexicographic order. Thus, by Lemma~\ref{l:inner-lex}, it follows that this subalgebra is isomorphic to the lexicographic product $\ZZ \overrightarrow{\times}\m{L}$.
    Therefore, $\ZZ \overrightarrow{\times}\m{L} \in \vr(\m{L})$.
    
    Now, for $\m{M}\in \vr(\m{L})$, there exists a set $I$, an algebra $\m{A} \leq \prod_{i\in I} \m{L}$, and a surjective homomorphism $\f\colon \m{A} \to \m{M}$.
    Since $\ZZ \overrightarrow{\times}\m{L} \in \vr(\m{L})$, also $\prod_{i\in I} (\ZZ \overrightarrow{\times}\m{L}) \in \vr(\m{L})$. It is straightforward to check that the subalgebra  of $\prod_{i\in I} (\ZZ \overrightarrow{\times}\m{L})$  with universe   
     $\{(z,a_i)_{i\in I} : (a_i)_{i\in I} \in A \text{ and }z \in \ZZ \}$  is isomorphic to $\ZZ \overrightarrow{\times}\m{A}$, so $\ZZ \overrightarrow{\times}\m{A} \in \vr(\m{L})$. It is also easy to verify that the map $\psi\colon \ZZ \overrightarrow{\times}\m{A} \to \ZZ \overrightarrow{\times}\m{M}$, where $\psi(z,a) = (z,\f(a))$, is a surjective homomorphism, hence  $\ZZ \overrightarrow{\times}\m{M} \in \vr(\m{L})$.
\end{proof}

\begin{lemma}\label{l:closure-lex}
     If an $\ell$-pregroup  $\m{L}$ contains a central invertible element $g$ that is a strong order unit, then $\vr(\m{L})$ is closed under lexicographic products with totally ordered abelian $\ell$-groups.
\end{lemma}

\begin{proof}
    Let $\m{G}$ be a totally ordered abelian $\ell$-group and let $\m{M} \in \vr(\m{L})$. Then, by Lemma~\ref{l:closure-lexZ}, $\ZZ \overrightarrow{\times} \m{M} \in \vr(\m{L})$. Now since $\ZZ$ generates the variety of abelian $\ell$-groups as a quasivariety, there exists an ultrapower $\ZZ^I/U$ and an embedding $f\colon \m{G} \to \ZZ^I/U$. We consider the ultrapower $(\ZZ \overrightarrow{\times} \m{M})^I/U$ and we define the map $\f \colon \m{G}\overrightarrow{\times} \m{M} \to (\ZZ \overrightarrow{\times} \m{M})^I/U$,  by 
    $\f(g,m) = [(n_i,m)_{i\in I}]$, where $f(g)=[(n_i)_{i\in I}]$. It is straightforward to check
    that this is a well-defined map and an embedding, hence $\m{G}\overrightarrow{\times} \m{M} \in \vr(\m{L})$. 
\end{proof}

\begin{proposition}\label{p:Fn-closed-lexprod}
    For each $n>1$ the variety $\vr(\m{F}_n(\ZZ))$ is closed under lexicographic products with totally ordered abelian $\ell$-groups.
\end{proposition}

\begin{proof}
    By Lemma~\ref{l:grp-hg}, $\cn{s}$ is a strong order unit of $\m{F}_n(\ZZ)$, hence $\cn{s}^n$ is one, as well. Moreover, by Lemma~\ref{l:b-conjugation}(3),  $\cn{s}^n$ is central, so it satisfies the assumptions of Lemma~\ref{l:closure-lex}.
\end{proof}

The next proposition will be important in characterizing the finitely generated finitely subdirectly irreducibles of $\vr(\m{F}_n(\ZZ))$.
\begin{proposition}\label{p:Zconv-lex}
     Let $\m{G}$ be a finitely generated totally ordered abelian $\ell$-group with convex subgroup $\ZZ$. Then $\m{G} \cong \m{H} \overrightarrow{\times}\ZZ$ for some finitely generated totally ordered abelian $\ell$-group $\m{H}$. Actually, $\m{G} \cong (\m{G}/\ZZ) \overrightarrow{\times} \ZZ$. 
 \end{proposition}

\begin{proof}
    Since $\ZZ$ is a subgroup of $\m{G}$, we will write $0$ for the neutral element of $\m{G}$, but we will still write the operation of $\m{G}$ multiplicatively. 
    As $\m{G}$ is finitely generated, $\m{G}/\ZZ$ is finitely generated, as well. Since an $\ell$-group cannot contain an element of finite non-zero order, by the fundamental theorem of finitely generated abelian groups, we get that $\m{G}/\ZZ$ is isomorphic, as a group, to a direct power of $\ZZ$. As a result,  $\m{G}/\ZZ$ is  freely generated. Let $[f_1],\dots,[f_n] \in \m{G}/\ZZ$ be free generators of $\m{G}/\ZZ$, where $f_1,\dots, f_n \in G$ and let $\m{H} = \langle f_1,\dots,f_n \rangle$; we claim that $\m{G} \cong \m{H}\overrightarrow{\times}\ZZ$. First note that $H \cdot \ZZ = G$, using a standard group-theory argument on cosets. 
    Moreover, if $g \in H \cap \ZZ$, then $f_1^{k_1}\cdots f_n^{k_n} = g \in \ZZ$ for some $k_1,\dots, k_n \in \ZZ$, i.e., $[f_1]^{k_1} \cdots [f_n]^{k_n} = [g] = [0]$. Since $[f_1],\dots,[f_n]$ freely generate $\m{G}/\ZZ$ we get $k_1 = \dots =k_n = 0$, i.e., $g = 0$. Thus $H \cap \ZZ = \{0\}$. So it follows that the group reduct of $\m{G}$ is isomorphic to the group reduct of $\m{H} \times \ZZ$. 
    It remains to check part (iii) of Lemma~\ref{l:inner-lex}.
    For $h_1m_1, h_2m_2 \in G$ with $m_1,m_2 \in \ZZ$ and $h_1,h_2 \in H$, if $h_1m_1 < h_2m_2$, then $[h_1] = [h_1m_1] \leq [h_2m_2] = [h_2]$, so $h_2 \not< h_1$, i.e., $h_1 \leq h_2$, and if additionally $h_1 = h_2$, then $m_1 < m_2$. Conversely if $h_1 < h_2$, then $[h_1] \leq [h_2]$ and $[h_1] \neq [h_2]$, for otherwise $h_1h_2^{-1} \in \ZZ$, yielding $h_1h_2^{-1} = 0$, i.e., $h_1 = h_2$, a contradiction. Hence $[h_1] < [h_2]$, i.e., $h_1m_1 < h_2m_2$. Finally if $h_1 = h_2$ and $m_1 < m_2$, then $h_1m_1 < h_2m_2$. Summarizing we get for all $h_1,h_2 \in H$, $m_1,m_2 \in \ZZ$,
    \[
    h_1m_1 < h_2m_2 \iff h_1 < h_2 \text{ or } (h_1 = h_2 \text{ and } m_1 < m_2).
    \]
    Hence, by Lemma~\ref{l:inner-lex}, $\m{G} \cong \m{H} \overrightarrow{\times} \ZZ$.
    \end{proof}

\section{The structure of \texorpdfstring{$n$}{n}-periodic \texorpdfstring{$\ell$}{ℓ}-pregroups}\label{s: n-per l-pr}

In this section we first recall the wreath product representation of $n$-periodic $\ell$-pregroups and we provide internal descriptions (inside the representation) of various attributes of elements. Moreover, we define the local subalgebra of an $n$-periodic $\ell$-pregroup, relative to a given representation, and showcase its interaction with the group skeleton.

\subsection{Representation and consequences}

We will make use of the two representation theorems for  $n$-periodic $\ell$-pregroups involving wreath products, established in \cite{GG2}, in order to describe idempotent and flat elements in these algebras, as well as the periodicities of the elements. 
We will use this representation mainly when the periodic $\ell$-pregroup is not itself an $\ell$-group (we will call such $\ell$-pregroups proper).

\begin{proposition}[\cite{GG2}]\label{p: representation1}
    For every $n$-periodic $\ell$-pregroup $\m{L}$ there exists a chain $\m{J}$, such that $\m{L}$ embeds into $\m{F}_n(\m{J} \overrightarrow{\times} \ZZ)$.
\end{proposition}

In many cases to prove properties of $\m L$ we will rely on the isomorphism between $\m L$ and its image under this embedding, hence we will often  only consider the case where $\m{L} \leq \m{F}_n(\m J \overrightarrow{\times} \ZZ)$.

The second representation theorem uses the extension of the wreath product construction to the $\ell$-pregroup setting, given in \cite{GG2}; here we present the definition in the special  case where the $\ell$-group is an automorphism $\ell$-group.

For a chain $\m{J}$ and an $\ell$-pregroup $\m{L}$ we define the \emph{wreath product} $\m{Aut}(\m{J}) \wr \m{L} = ( \mathrm{Aut}(\m{J}) \times {L}^J, \mt, \jn,  \cdot, (id, \overline{1}), {}^{\ell}, {}^r )$, where we write $\overline{1}$ for the sequence $(1)_{i\in J}$, and for $f = (\widetilde{f}, \overline{f}), g = (\widetilde{g},\overline{g}) \in \mathrm{Aut}(\m{J}) \times {L}^J$ we let 
\begin{align*}
    (\widetilde{f},\overline{f}) \cdot (\widetilde{g},\overline{g}) &= (\widetilde{f}\circ \widetilde{g}, (\overline{f}\otimes \widetilde{g}) \cdot \overline{g}) \\
    (\widetilde{f},\overline{f})^\ell &= (\widetilde{f}^{-1}, \overline{f}^\ell \otimes \widetilde{f}^{-1}) \\
    (\widetilde{f},\overline{f})^r &= (\widetilde{f}^{-1}, \overline{f}^r \otimes \widetilde{f}^{-1}),
\end{align*}
where for $h\in \mathrm{Aut}(\m{J})$ and $a\in {L}^{J}$ we define the element $a \otimes h \in {L}^{J}$ by $(a \otimes h)_i = a_{h(i)}$; the order of $\m{Aut}(\m{J}) \wr \m{L}$ is defined by 
\[
(\widetilde{f},\overline{f}) \leq (\widetilde{g},\overline{g}) \quad \text{iff} \quad \widetilde{f} \leq \widetilde{g} \quad \text{and} \quad \widetilde{f}(i) = \widetilde{g}(i) \implies \overline{f}_i \leq \overline{g}_i, \text{ for all } i\in J.
\]
In particular, we have $\widetilde{f \mt g} = \widetilde{f} \mt \widetilde{g}$, $\widetilde{f \jn g} = \widetilde{f} \jn \widetilde{g}$, and for each $i\in J$
\[
\overline{(f \mt g)}_i = 
\begin{cases}
    \overline{f}_i &\text{if } \widetilde{f}(i) \leq \widetilde{g}(i) \\
    \overline{g}_i &\text{if } \widetilde{f}(i) \geq \widetilde{g}(i) \\
    \overline{f}_i \mt \overline{g}_i &\text{if } \widetilde{f}(i) = \widetilde{g}(i).
\end{cases} 
\quad
\overline{(f\jn g)}_i =
\begin{cases}
    \overline{f}_i &\text{if } \widetilde{f}(i) \geq \widetilde{g}(i) \\
    \overline{g}_i &\text{if } \widetilde{f}(i) \leq \widetilde{g}(i) \\
    \overline{f}_i \jn \overline{g}_i &\text{if } \widetilde{f}(i) = \widetilde{g}(i).
\end{cases} 
\]
For $f =(\widetilde{f},\overline{f}) \in \mathrm{Aut}(\m{J}) \times L^J$ we call $\widetilde{f}$ the \emph{global component} and $\overline{f}$ the \emph{local component} of $f$.

\begin{lemma}[\cite{GG2}]
   If $\m{J}$ is a chain and $\m{L}$ is an $\ell$-pregroup, then $\m{Aut}(\m{J})\wr \m{L}$ is an $\ell$-pregroup. If in addition, $\m{L}$ is $n$-periodic, then so is $\m{Aut}(\m{J})\wr \m{L}$.
\end{lemma}
\begin{proposition}[\cite{GG2}]
    If $\m{J}$ is a chain and $n\geq 1$, then $\m{F}_n(\m{J}\overrightarrow{\times} \ZZ)$ is isomorphic to the wreath product $\m{Aut}(\m{J}) \wr \m{F}_n(\ZZ)$. In particular, for $f\in \mathrm{F}_n(\m{J}\overrightarrow{\times} \ZZ)$ and corresponding  $(\widetilde{f},\overline{f}) \in \mathrm{Aut}(\m{J}) \times \mathrm{F}_n(\ZZ)^J$, we have that for all $(i,m) \in J\times \ZZ$: 
    \[
    f(i,m) = (\widetilde{f}(i), \overline{f}_i(m)).
    \]
\end{proposition}
For $f\in \mathrm{F}_n(\m{J}\overrightarrow{\times} \ZZ)$ and   $(\widetilde{f},\overline{f}) \in \mathrm{Aut}(\m{J}) \times \mathrm{F}_n(\ZZ)^J$ we will be writing $f \equiv (\widetilde{f},\overline{f})$ to indicate that the two elements correspond; see Figure~\ref{f: global-local} for an example.

\begin{figure}[ht]\def\eq{=}
\centering
{\scriptsize
\begin{tikzpicture}
[scale=0.4]
%--------
\node[fill,draw,circle,scale=0.3,left](9) at (0,9){};
\node[right](9.1,0) at (-.9,9.25){$3$};
\node[fill,draw,circle,scale=0.3,right](9r) at (2,9){};
\node[right](9.1) at (2.1,9.25){$3$};

\node[fill,draw,circle,scale=0.3,left](8) at (0,8){};
\node[right](8.1,0) at (-.9,8.25){$2$};
\node[fill,draw,circle,scale=0.3,right](8r) at (2,8){};
\node[right](8.1) at (2.1,8.25){$2$};

\node[fill,draw,circle,scale=0.3,left](7) at (0,7){};
\node[right](7.1,0) at (-.9,7.25){$1$};
\node[fill,draw,circle,scale=0.3,right](7r) at (2,7){};
\node[right](7.1) at (2.1,7.25){$1$};

\node[fill,draw,circle,scale=0.3,left](6) at (0,6){};
\node[right](6.1,0) at (-.9,6.25){$0$};
\node[fill,draw,circle,scale=0.3,right](6r) at (2,6){};
\node[right](6.1) at (2.1,6.25){$0$};

%---------
\node[fill,draw,circle,scale=0.3,left](4) at (0,4){};
\node[right](9.1,0) at (-.9,4.25){$7$};
\node[fill,draw,circle,scale=0.3,right](4r) at (2,4){};
\node[right](4.1) at (2.1,4.25){$7$};

\node[fill,draw,circle,scale=0.3,left](3) at (0,3){};
\node[right](3.1,0) at (-.9,3.25){$6$};
\node[fill,draw,circle,scale=0.3,right](3r) at (2,3){};
\node[right](3.1) at (2.1,3.25){$6$};

\node[fill,draw,circle,scale=0.3,left](2) at (0,2){};
\node[right](2.1,0) at (-.9,2.25){$5$};
\node[fill,draw,circle,scale=0.3,right](2r) at (2,2){};
\node[right](2.1) at (2.1,2.25){$5$};

\node[fill,draw,circle,scale=0.3,left](1) at (0,1){};
\node[right](1.1,0) at (-.9,1.25){$4$};
\node[fill,draw,circle,scale=0.3,right](1r) at (2,1){};
\node[right](1.1) at (2.1,1.25){$4$};

%--------

\node[fill,draw,circle,scale=0.3,left](-1) at (0,-1){};
\node[right](-1.1,0) at (-.9,-1.25){$4$};
\node[fill,draw,circle,scale=0.3,right](-1r) at (2,-1){};
\node[right](9.1) at (2.1,-1.25){$4$};

\node[fill,draw,circle,scale=0.3,left](-1) at (0,-2){};
\node[right](-1.1,0) at (-.9,-2.25){$3$};
\node[fill,draw,circle,scale=0.3,right](-1r) at (2,-2){};
\node[right](9.1) at (2.1,-2.25){$3$};

\node[fill,draw,circle,scale=0.3,left](-2) at (0,-3){};
\node[right](-2.1,0) at (-.9,-3.25){$2$};
\node[fill,draw,circle,scale=0.3,right](-2r) at (2,-3){};
\node[right](-2.1) at (2.1,-3.25){$2$};

\node[fill,draw,circle,scale=0.3,left](-3) at (0,-4){};
\node[right](-3.1,0) at (-.9,-4.25){$1$};
\node[fill,draw,circle,scale=0.3,right](-3r) at (2,-4){};
\node[right](-3.1) at (2.1,-4.25){$1$};

\node[fill,draw,circle,scale=0.3,left](-4) at (0,-5){};
\node[right](-4.1,0) at (-.9,-5.25){$0$};
\node[fill,draw,circle,scale=0.3,right](-4r) at (2,-5){};
\node[right](-4.1) at (2.1,-5.25){$0$};
%--------
\node at (1)[below= 1pt]{$\vdots$};
\node at (1r)[below= 1pt]{$\vdots$};
\node at (6)[below=1pt]{$\vdots$};
\node at (6r)[below=1pt]{$\vdots$};
%\node at (1)[above=3pt]{$\vdots$};
%\node at (1r)[above=3pt]{$\vdots$};
\node at (9)[above=3pt]{$\vdots$};
\node at (9r)[above=3pt]{$\vdots$};
\node at (-4)[below=-1pt]{$\vdots$};
\node at (-4r)[below=-1pt]{$\vdots$};
%--------

\draw[-](1)--(7r);
\draw[-](2)--(7r);
\draw[-](3)--(9r);
\draw[-](4)--(9r);

\draw[-](-1)--(3r);
\draw[-](-2)--(3r);
\draw[-](-3)--(1r);
\draw[-](-4)--(1r);

\node at (1,-6){\normalsize$f$};
\end{tikzpicture}
\qquad \qquad \qquad
\begin{tikzpicture}
[scale=0.4]
%--------
%\node[fill,draw,circle,scale=0.3,left](9) at (0,9){};
%\node[right](9.1,0) at (-.9,9.25){$7$};
%\node[fill,draw,circle,scale=0.3,right](9r) at (2,9){};
%\node[right](9.1) at (2.1,9.25){$7$};

%\node[fill,draw,circle,scale=0.3,left](8) at (0,8){};
%\node[right](8.1,0) at (-.9,8.25){$6$};
%\node[fill,draw,circle,scale=0.3,right](8r) at (2,8){};
%\node[right](8.1) at (2.1,8.25){$6$};

%\node[fill,draw,circle,scale=0.3,left](7) at (0,7){};
%\node[right](7.1,0) at (-.9,7.25){$5$};
%\node[fill,draw,circle,scale=0.3,right](7r) at (2,7){};
%\node[right](7.1) at (2.1,7.25){$5$};

\node[fill,draw,circle,scale=0.3,left](6) at (0,7){};
%\node[right](6.1,0) at (-.9,6.25){$2$};
\node[fill,draw,circle,scale=0.3,right](6r) at (2,7){};
%\node[right](6.1) at (2.1,6.25){$2$};

%\node[fill,draw,circle,scale=0.3,left](5) at (0,5){};
%\node[right](9.1,0) at (-.9,5.25){$3$};
%\node[fill,draw,circle,scale=0.3,right](5r) at (2,5){};
%\node[right](5.1) at (2.1,5.25){$3$};

\node[fill,draw,circle,scale=0.3,left](4) at (0,1.5){};
%\node[right](8.1,0) at (-.9,4.25){$1$};
\node[fill,draw,circle,scale=0.3,right](4r) at (2,1.5){};
%\node[right](4.1) at (2.1,4.25){$1$};

%\node[fill,draw,circle,scale=0.3,left](3) at (0,3){};
%\node[right](3.1,0) at (-.9,3.25){$1$};
%\node[fill,draw,circle,scale=0.3,right](3r) at (2,3){};
%\node[right](3.1) at (2.1,3.25){$1$};

\node[fill,draw,circle,scale=0.3,left](2) at (0,-4){};
%\node[right](2.1,0) at (-.9,2.25){$0$};
\node[fill,draw,circle,scale=0.3,right](2r) at (2,-4){};
%\node[right](2.1) at (2.1,2.25){$0$};
%---------
\node at (6)[above=3pt]{$\vdots$};
\node at (6r)[above=3pt]{$\vdots$};
\node at (2)[below=-1pt]{$\vdots$};
\node at (2r)[below=-1pt]{$\vdots$};
%--------
\draw[-](2)--(4r);
\draw[-](4)--(6r);
%\draw[-](4)--(8r);
%\draw[-](5)--(8r);

%--------
\node at (1,-5){\normalsize $\tilde{f}$};
\end{tikzpicture}
\qquad \qquad \qquad
\begin{tikzpicture}
[scale=0.4]
%--------
\node[fill,draw,circle,scale=0.3,left](9) at (0,9){};
\node[right](9.1,0) at (-.9,9.25){$7$};
\node[fill,draw,circle,scale=0.3,right](9r) at (2,9){};
\node[right](9.1) at (2.1,9.25){$7$};

\node[fill,draw,circle,scale=0.3,left](8) at (0,8){};
\node[right](8.1,0) at (-.9,8.25){$6$};
\node[fill,draw,circle,scale=0.3,right](8r) at (2,8){};
\node[right](8.1) at (2.1,8.25){$6$};

\node[fill,draw,circle,scale=0.3,left](7) at (0,7){};
\node[right](7.1,0) at (-.9,7.25){$5$};
\node[fill,draw,circle,scale=0.3,right](7r) at (2,7){};
\node[right](7.1) at (2.1,7.25){$5$};

\node[fill,draw,circle,scale=0.3,left](6) at (0,6){};
\node[right](6.1,0) at (-.9,6.25){$4$};
\node[fill,draw,circle,scale=0.3,right](6r) at (2,6){};
\node[right](6.1) at (2.1,6.25){$4$};

\node[fill,draw,circle,scale=0.3,left](5) at (0,5){};
\node[right](9.1,0) at (-.9,5.25){$3$};
\node[fill,draw,circle,scale=0.3,right](5r) at (2,5){};
\node[right](5.1) at (2.1,5.25){$3$};

\node[fill,draw,circle,scale=0.3,left](4) at (0,4){};
\node[right](8.1,0) at (-.9,4.25){$2$};
\node[fill,draw,circle,scale=0.3,right](4r) at (2,4){};
\node[right](4.1) at (2.1,4.25){$2$};

\node[fill,draw,circle,scale=0.3,left](3) at (0,3){};
\node[right](3.1,0) at (-.9,3.25){$1$};
\node[fill,draw,circle,scale=0.3,right](3r) at (2,3){};
\node[right](3.1) at (2.1,3.25){$1$};

\node[fill,draw,circle,scale=0.3,left](2) at (0,2){};
\node[right](2.1,0) at (-.9,2.25){$0$};
\node[fill,draw,circle,scale=0.3,right](2r) at (2,2){};
\node[right](2.1) at (2.1,2.25){$0$};
%---------
\node at (9)[above=3pt]{$\vdots$};
\node at (9r)[above=3pt]{$\vdots$};
\node at (2)[below=-1pt]{$\vdots$};
\node at (2r)[below=-1pt]{$\vdots$};
%--------
\draw[-](2)--(6r);
\draw[-](3)--(6r);
\draw[-](4)--(8r);
\draw[-](5)--(8r);

%--------
\node at (1,1.25){\normalsize $\overline{f}_0$};
\end{tikzpicture}
\qquad \qquad \qquad
\begin{tikzpicture}
[scale=0.4]
%--------
\node[fill,draw,circle,scale=0.3,left](9) at (0,9){};
\node[right](9.1,0) at (-.9,9.25){$7$};
\node[fill,draw,circle,scale=0.3,right](9r) at (2,9){};
\node[right](9.1) at (2.1,9.25){$7$};

\node[fill,draw,circle,scale=0.3,left](8) at (0,8){};
\node[right](8.1,0) at (-.9,8.25){$6$};
\node[fill,draw,circle,scale=0.3,right](8r) at (2,8){};
\node[right](8.1) at (2.1,8.25){$6$};

\node[fill,draw,circle,scale=0.3,left](7) at (0,7){};
\node[right](7.1,0) at (-.9,7.25){$5$};
\node[fill,draw,circle,scale=0.3,right](7r) at (2,7){};
\node[right](7.1) at (2.1,7.25){$5$};

\node[fill,draw,circle,scale=0.3,left](6) at (0,6){};
\node[right](6.1,0) at (-.9,6.25){$4$};
\node[fill,draw,circle,scale=0.3,right](6r) at (2,6){};
\node[right](6.1) at (2.1,6.25){$4$};

\node[fill,draw,circle,scale=0.3,left](5) at (0,5){};
\node[right](9.1,0) at (-.9,5.25){$3$};
\node[fill,draw,circle,scale=0.3,right](5r) at (2,5){};
\node[right](5.1) at (2.1,5.25){$3$};

\node[fill,draw,circle,scale=0.3,left](4) at (0,4){};
\node[right](8.1,0) at (-.9,4.25){$2$};
\node[fill,draw,circle,scale=0.3,right](4r) at (2,4){};
\node[right](4.1) at (2.1,4.25){$2$};

\node[fill,draw,circle,scale=0.3,left](3) at (0,3){};
\node[right](3.1,0) at (-.9,3.25){$1$};
\node[fill,draw,circle,scale=0.3,right](3r) at (2,3){};
\node[right](3.1) at (2.1,3.25){$1$};

\node[fill,draw,circle,scale=0.3,left](2) at (0,2){};
\node[right](2.1,0) at (-.9,2.25){$0$};
\node[fill,draw,circle,scale=0.3,right](2r) at (2,2){};
\node[right](2.1) at (2.1,2.25){$0$};
%---------
\node at (9)[above=3pt]{$\vdots$};
\node at (9r)[above=3pt]{$\vdots$};
\node at (2)[below=-1pt]{$\vdots$};
\node at (2r)[below=-1pt]{$\vdots$};
%--------
\draw[-](6)--(3r);
\draw[-](7)--(3r);
\draw[-](8)--(5r);
\draw[-](9)--(5r);

%--------
\node at (1,1.25){\normalsize $\overline{f}_1$};
\end{tikzpicture}
}
\caption{An element $f$ of $\m F_2(\mathbb{Q} \overrightarrow{\times} \mathbb{Z})$, its global component $\widetilde{f} \in \m F_1(\mathbb{Z})$, and two of its local components  $\overline{f}_0, \overline{f}_1 \in \m F_2(\mathbb{Z})$.
}
\label{f: global-local}
\end{figure}
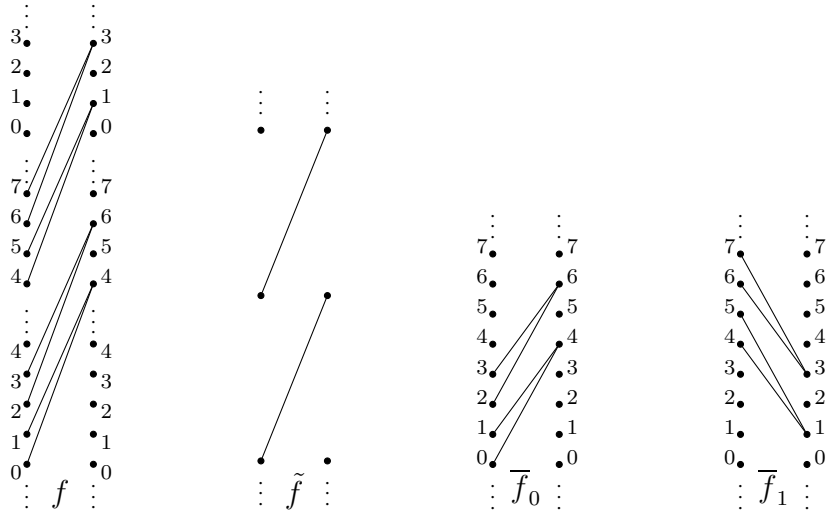

\begin{lemma}
  If $(\widetilde{f},\overline{f}) \equiv f\in \mathrm{F}_n(\m{J}\overrightarrow{\times} \ZZ)$, where $\m{J}$ is a chain, then we have $\g(f) \equiv (\widetilde{f},\g(\overline{f}))$ and $\h(f) \equiv (\widetilde{f},\h(\overline{f}))$, where $\g(\overline{f})$ and $\h(\overline{f})$ are computed coordinate-wise. 
\end{lemma}
\begin{proof}
    Let $(\widetilde{f},\overline{f}) \equiv f \in \mathrm{F}_n(\m{J}\overrightarrow{\times} \ZZ)$. First note that 
       \begin{align*}
        f^{\ell\ell} \equiv (\widetilde{f}^{-1},\overline{f}^{\ell} \otimes \widetilde{f}^{-1})^\ell &= (\widetilde{f},(\overline{f}^{\ell} \otimes \widetilde{f}^{-1})^\ell \otimes \widetilde{f}) \\
                      &= (\widetilde{f},(\overline{f}^{\ell\ell} \otimes \widetilde{f}^{-1}) \otimes \widetilde{f}) = (\widetilde{f},\overline{f}^{\ell\ell}),
    \end{align*}
    where the third equality holds because $^{\ell}$ is computed coordinate-wise. Thus we get 
    \begin{align*}
        \g(f) &= f \mt f^{[1]} \mt \dots \mt f^{[n-1]} \\
              &\equiv (\widetilde{f},\overline{f}) \mt (\widetilde{f},\overline{f}^{[1]}) \mt \dots \mt (\widetilde{f},\overline{f}^{[n-1]}) \\
              &= (\widetilde{f},\overline{f} \mt \overline{f}^{[1]} \mt \dots \mt \overline{f}^{[n-1]}) \\
              &= (\widetilde{f},\g(\overline{f})),
    \end{align*}
    where the third equality holds, since all the meetands have the same global component. Similarly it follows that $\h(f) \equiv (\widetilde{f},\h(\overline{f}))$.
\end{proof}

The next lemma shows how the periodicity of an element is calculated in a representation. 

\begin{lemma} \label{l: per}
  If $\m L \leq  \m{F}_n(\m{J}\overrightarrow{\times} \ZZ)$, for some chain $\m J$, then $\per(a)=\per(\overline{a})=\lcm\{\per(\overline{a}_i) : i \in J\}$. 
 \end{lemma}

 \begin{proof}
By the definition of ${}^\ell$ on wreath products we get $a^{\ell\ell}=(\widetilde{a}, \overline{a})^{\ell\ell}=(\widetilde{a}, \overline{a}^{\ell\ell})$, so  $a^{[k]}=a$ iff $\overline{a}^{[k]}=\overline{a}$; hence $\per(a)=\per(\overline{a})$. 
Furthermore, $\overline{a}^{[k]}=((\overline{a}_i)_{i \in J})^{[k]}=((\overline{a}_i^{[k]})_{i \in J})$, so  $\overline{a}^{[k]}= \overline{a}$ iff $(\overline{a}_i^{[k]})_{i \in J}=(\overline{a}_i)_{i \in J}$ iff $\overline{a}_i^{[k]}=\overline{a}_i$, for all $i \in J$. Thus, $\per(\overline{a})=\lcm\{\per(\overline{a}_i) : i \in J\}$.
 \end{proof}

The following lemma characterizes the idempotent and the flat elements in $n$-periodic $\ell$-pregroups in terms of their global and local coordinates.

\begin{lemma}\label{l: flat}
If  $\m{L}\leq \m{F}_n(\m{J}\overrightarrow{\times} \ZZ)$ for some chain $\m{J}$, then an element $a\in L$ is idempotent (flat) iff \ $\widetilde{a} = id$ and $\overline{a}_i$ is idempotent (flat) in $\m{F}_n(\ZZ)$, for each $i\in J$. In particular, the only invertible flat element in an $n$-periodic $\ell$-pregroup is $1$.
\end{lemma}

\begin{proof}
 First note that an element $b$ of $L$ is idempotent iff $\widetilde{b} = id$ and  $\overline{b}$ is idempotent. Also, $\overline{b}$ is idempotent iff $\overline{b}_i$ is idempotent for all $i \in J$.

So, an element $a$ of $L$ is flat iff there exists idempotents $b, c \in L$ such that $b \leq a\leq c$ iff 
there exist idempotents $b, c \in L$ such that $\widetilde{b}=\widetilde{a}=\widetilde{c}=id$
and  $\overline{b} \leq \overline{a}\leq \overline{c}$ 
iff $\widetilde{a}=id$ and
there exist  idempotents $\overline{b}, \overline{c} \in F_n(\ZZ)^J$ such that $\overline{b} \leq \overline{a}\leq \overline{c}$ 
iff $\widetilde{a}=id$ and $\overline{a}$ is flat. Moreover, $\overline{a}$ is flat iff $\overline{a}_i$ is flat, for all $i \in J$.

In particular, it follows that $1$ is the only invertible flat element. 
\end{proof}

The following corollary characterizes positive flat elements in the $n$-periodic case. 

\begin{corollary}
    If $\m{L}$ is an $n$-periodic $\ell$-pregroup and $a\in L$ with $1 \leq a$, then $a$ is flat if and only if $\g(a) = 1$.
\end{corollary}

\begin{proof}
    We may consider $\m{L}$ as a subalgebra of $\m{F}_n(\m{J}\overrightarrow{\times}\ZZ)$ for some chain $\m{J}$.
    If $1\leq a$ is flat, then there exists an idempotent $c \in L$ such that $1 \leq a\leq c$. By Lemma~\ref{l: flat},
    we have $\widetilde{c} = id$ and  $\overline{c}_i$ is a positive idempotent, for each $i\in J$. By Lemma~\ref{l:flat-Fn}, $\g(\overline{c}_i) = 1$, for each $i\in J$, hence $1\leq \g(a) \leq \g(c) = 1$. 
    
    Conversely, if $\g(a) = 1$, then $\widetilde{a} = id$, and for each $i\in J$, $\g(\overline{a}_i) = 1$. Since $1\leq a$, by Lemma~\ref{l:flat-Fn}, {$b:=\overline{a}^{n-1}$ is idempotent and we have $1 \leq \overline{a} \leq b$,} so $(id, b) = a^{n-1}$ is idempotent and $1 \leq a \leq (id, b)$. Thus $a$ is flat.
\end{proof}

\subsection{Local subalgebra}

For $\m{L} \leq \m{F}_n(\m{J}\overrightarrow{\times} \ZZ)$, where $\m J$ is a chain, we define its \emph{local subalgebra} by $\loc{\m{L}} = \{f \in L : \widetilde{f} = id \}$ and we call the elements of $\loc{\m{L}}$ \emph{local}.
Since $\m{F}_n(\m{J}\overrightarrow{\times} \ZZ) \cong \m{Aut}(\m J) \wr \m{F}_n(\ZZ)=\m{Aut}(\m J) \rtimes \m{F}_n(\ZZ)^J$, we get that $\loc{\m{L}}$ is a subalgebra of $\{id\} \rtimes \m{F}_n(\ZZ)^J \cong \m{F}_n(\ZZ)^J$; here $\m G \rtimes \m L$ denotes the semidirect product of an $\ell$-group $\m G$ and an $\ell$-pregroup $\m L$, which is defined in \cite{GG2} and generalizes the usual semidirect product of monoids. Therefore, $\loc{\m{L}}$ is isomorphic to a subalgebra of $\m{F}_n(\ZZ)^J$. Note that, by Lemma~\ref{l: flat}, every flat element is local with respect to every representation. Also, for a given representation we have $\grp{(\loc{L})} = \loc{(\grp{L})} = \grp{L} \cap \loc{L}$.

\begin{example}  It turns out that in general the local subalgebra (and the notion of local element) depends on the representation, as we illustrate by an example.
   Consider the $\ell$-group $\m{L} =  \ZZ\times \ZZ$ and its two generators $(1,0)$ and $(0,1)$. Let $\m{J} = \m{2}\overrightarrow{\times} \ZZ$.  Note that the map $\f \colon L \to \mathrm{F}_n(\m{J}\overrightarrow{\times} \ZZ)$, induced by $\f(1,0) = f = (id_J,\overline{f})$, $\f(0,1) = g = (id_J,\overline{g})$ with 
    \[
        \overline{f}_{(i,n)} = 
        \begin{cases}
        \cn{s} &\text{if } i = 0 \\
        1 &\text{if } i = 1
        \end{cases} \quad 
        \overline{g}_{(i,n)} = 
        \begin{cases}
        1 &\text{if } i = 0\\
        \cn{s} &\text{if } i = 1
        \end{cases}
    \]
    is an embedding and with respect to $\f$ we have $\loc{\m{L}} = \m{L}$. On the other hand the map $\ps \colon L \to \mathrm{F}_n(\m{J}\overrightarrow{\times} \ZZ)$ induced by $\ps(1,0) = f$, $\ps(0,1) = h$ with  
    \[
        \widetilde{h}(i,n) = 
        \begin{cases}
            (i,n)   &\text{if } i=0 \\
            (i,n+1) &\text{if } i=1
        \end{cases}
        \quad
        \overline{h}_{(i,n)} = 
        \begin{cases}
        1 &\text{if } i=0 \\
        \cn{s} &\text{if } i=1
        \end{cases}
   \]
   is another embedding but with respect to $\ps$ we have $\loc{\m{L}} = \langle f \rangle$.
We conjecture that every $n$-periodic $\ell$-pregroup has a maximal local subalgebra. In any case, as we will see, the results we will establish either do not depend on the particular representation of the $n$-periodic $\ell$-pregroup or, when they mention local elements, they hold for all possible representations (and the associated local subalgebras).
\end{example}

We say that an $\ell$-pregroup is \emph{proper} if it is not an $\ell$-group (equivalently, it is not equal to its group skeleton).

\begin{lemma}\label{l: local-idempotent}
If $\m L$ is a proper periodic $\ell$-pregroup, then it has a strictly positive idempotent element that is local with respect to any given representation of $\m L$. 
\end{lemma}

\begin{proof}
    Since $\m{L}$ is not an $\ell$-group, there exists a non-invertible element $a$. So, $b:=a a^\ell$ is a 
    non-identity element that is idempotent, by Remark~\ref{r: xxl}, hence $b$ is not invertible. Moreover, by Lemma~\ref{l: flat}, $b$ is local.
    Finally, $b=a a^\ell \geq 1$,  hence, since $b\neq 1$, $b$ is strictly positive.
\end{proof}

Like the preceding lemma, the next lemma also applies to periodic proper $\ell$-pregroups, but it replaces the idempotency of the produced element with the preservation of the periodicity.

\begin{lemma}\label{l: spfpl}
    If $\m{L}$ is an $n$-periodic $\ell$-pregroup and $a \in  L$ is a non-invertible element, then the element $\g(a)^{-1}a$ is strictly positive, flat, of the same  periodicity as $a$, and it  can be taken to be local with respect to any given representation of $\m L$.
\end{lemma}

\begin{proof} 
Since $\m L$  periodic, it has a representation $\m{L}\leq \m{F}_n(\m{J} \overrightarrow{\times} \ZZ)$, where $\m J$ is a chain. Then $\widetilde{\g(a)}=\widetilde{a}$, so $\widetilde{\g(a)^{-1}a}=id$ and $c:=\g(a)^{-1}a$ is local. For each $i \in J$, we have $\overline{c}_i=\overline{\g(a)^{-1}a}_i=\g(\overline{a}_i)^{-1}\overline{a}_i$, so, by Corollary~\ref{c:flat-g}, the element $\overline{c}_i$ is positive, flat and of the same  periodicity as $\overline{a}_i$. Therefore, $c$ is also positive and since $a$ is not invertible we get that $c$ is strictly positive. Since for each $i \in J$, $\overline{c}_i$ is flat, we have that $\overline{c}_i^n$ is idempotent. Thus $c^n$ is idempotent and, $1 \leq c \leq c^n$, so $c$ is flat.
By Lemma~\ref{l: per}, $\per(a)=\lcm\{\per(\overline{a}_i) : i \in J\}=\lcm\{\per(\overline{c}_i) : i \in J\}=\per(c)$.
\end{proof}

\begin{lemma}\label{l:local-subuniverse}
    Let $\m{L} \leq \m{F}_n(\m{J}\overrightarrow{\times} \ZZ)$.
    \begin{enumerate}[label = \textup{(\arabic*)}]
        \item $\loc{\m{L}}$ is a convex normal subalgebra of $\m{L}$.
        \item $\loc{(\Grp{\m{L}})}$ is a convex normal $\ell$-subgroup of $\Grp{\m{L}}$.
        \item $\m{L}/\loc{\m L} \cong \Grp{\m L}/\loc{(\Grp{\m L})}$.
    \end{enumerate}
\end{lemma}

\begin{proof}
    (1) Clearly $\loc{\m{L}}$ is a subalgebra of $\m{L}$. 
    To see that it is convex let $a,b \in \loc{L}$ and $c\in L$ with $a\leq c \leq b$. Then, since $\widetilde{a} = id$ and $\widetilde{b} = id$, we also get that $\widetilde{c} = id$, since for each $(j,k) \in J\times \ZZ$, $(j,\overline{a}_j(k)) = a(j,k) \leq c(j,k) = (\widetilde{c}(j),\overline{c}_j(k)) \leq b(j,k) = (j,\overline{b}_j(k))$. Hence $c$ is also local, i.e., $c\in \loc{L}$. To see that it is normal let $a\in \loc{L}$ and $b\in L$. Then for each $(j,k) \in J\times \ZZ$, 
    \begin{align*}
    bab^\ell(j,k) &= ba(\widetilde{b}^{-1}(j), \overline{b}^{\ell}_{\widetilde{b}^{-1}(j)}(k))\\
    &= b(\widetilde{b}^{-1}(j),\overline{a}_{\widetilde{b}^{-1}(j)}\overline{b}^{\ell}_{\widetilde{b}^{-1}(j)}(k)) \\
    &= (j, \overline{b}_{\widetilde{b}^{-1}(j)} \overline{a}_{\widetilde{b}^{-1}(j)}\overline{b}^{\ell}_{\widetilde{b}^{-1}(j)}(k)).
    \end{align*}
    Thus $\widetilde{bab^\ell}=id$ and $bab^\ell \in \loc{L}$. In particular, $bab^\ell \mt 1 \in \loc{L}$.
    Thus $\loc{\m L}$ is normal.

    (2) Note that $\loc{(\Grp{L})} = \loc{L}\cap \Grp{L}$, so, by Lemma~\ref{l:groupNC-pregroupNC}, $\loc{(\Grp{\m{L}})}$ is a convex normal subgroup of $\grp{L}$.

    (3) Let $\pi \colon \m{L} \to \m{L}/\loc{\m L}$ be the canonical epimorphism and $\f\colon \Grp{\m L} \to \m{L}/\loc{\m L}$ its restriction to the subalgebra $ \Grp{\m L}$.  
    We have $\pi(a) = 1$ iff $a \in \loc{L}$, for all $a \in L$. So, since, by Lemma~\ref{l: spfpl}, $\g(a)^{-1}a$ is local for all $a \in L$, we get $\pi(a) = \pi(\g(a)\g(a)^{-1}a) =\pi(\g(a))$; thus $\pi[L]=\pi[\grp{L}]$ and
    $\f$ is an onto homomorphism.
   Moreover, for all $f \in \grp{L}$ we have  $f\in \ker(\f)$ iff $f\in \loc{L}$ iff $f\in \loc{L} \cap \grp{L}=\loc{(\grp{L})}$, so
    $\ker(\f) = \loc{(\grp{L})}$. The result follows by the first isomorphism theorem.
\end{proof}

\section{Axiomatization of \texorpdfstring{$\vr(\m F_n(\ZZ))$}{V(Fₙ(ℤ))}}\label{s: axiomatization}

In this section we establish the main result of the paper: $\vr(\m F_n(\ZZ))$ is axiomatized relative to $\va{LP}_n$ by the equation $x\g(y)^n \approx \g(y)^n x$. 

 In Corollary~\ref{c:fsi-commutative} we proved that among $n$-periodic $\ell$-pregroups that satisfy the equation $x\g(y)^n \approx \g(y)^n x$,  the FSIs  are exactly the ones whose group skeleton is totally ordered. 
 Also, in Proposition~\ref{p:Fn-closed-lexprod} we showed that
 $\vr(\m{F}_n(\ZZ))$ is closed under lexicographic products by totally ordered abelian $\ell$-groups.
  Here we show that  if a finitely generated algebra in $\va{LP}_n$ has a totally ordered abelian  group skeleton, then it is a  lexicographic product of a  totally ordered abelian $\ell$-group with $\m F_k(\ZZ)$ for some $k \mid n$. 
 We conclude that every  finitely generated FSI  $n$-periodic $\ell$-pregroup that satisfies the equation $x\g(y)^n \approx \g(y)^n x$ is in  $\vr(\m F_n(\ZZ))$.

 We use these results to further axiomatize all joins of varieties of the form $\vr(\m F_n(\ZZ))$ and fully describe the lattice they form.

\subsection{Algebras with totally ordered group skeleton}

We will study $n$-periodic $\ell$-pregroups with a totally ordered group skeleton and  prove that this skeleton is discretely ordered (Remark~\ref{r: discreteskeleton}) and  its local subalgebra is isomorphic to $\ZZ$ (Lemma~\ref{l:support-local}), when the $\ell$-pregroup is not an $\ell$-group; in particular, if the skeleton is further abelian, then it decomposes as a lexicographic product where the inner factor is $\ZZ$ (Lemma~\ref{l:loc-invariant-grp}). 
{This allows us to prove that  if a proper $n$-periodic $\ell$-pregroup has a totally ordered abelian group skeleton, then its local subalgebra is isomorphic to $\m{F}_k(\ZZ)$, for some $k \mid n$ (Theorem~\ref{t:Fn-FSI-local}).

In order to prove the axiomatization result, our main goal is to show that every FSI $n$-periodic $\ell$-pregroup satisfying $x\g(y)^n \approx \g(y)^n x$ is in $\vr(\m{F}_n(\ZZ))$; we also note that it is enough to focus only on finitely generated algebras. Since, by Proposition~\ref{p:Fn-closed-lexprod}, $\vr(\m{F}_n(\ZZ))$ is closed under lexicographic products by totally ordered abelian $\ell$-groups, in view of Corollary~\ref{c:fsi-commutative} it suffices to show that finitely generated $n$-periodic $\ell$-pregroups with a totally ordered abelian skeleton are lexicographic products of a totally ordered abelian $\ell$-group and an algebra in $\vr(\m{F}_n(\ZZ))$. In Theorem~\ref{t:lexprod-fsi-Fn} we will establish exactly that, by showing that this algebra can be taken to be $\m{F}_k(\ZZ)$ for some  $k \mid n$.

For a local function $f \in \m{F}_n(\m J \overrightarrow{\times} \ZZ)$ we define
its \emph{support}, $\supp(f)$, to be the support $\supp(\overline{f})= \{j \in J : \overline{f}_j \neq 1 \}$ of its local component $\overline{f}$. 

A key part of the next lemma is that if an $n$-periodic $\ell$-pregroup has a totally ordered skeleton, then all local elements have the same support.

\begin{lemma}\label{l:support-local} 
   Let $\m{L} \leq \m{F}_n(\m J \overrightarrow{\times} \ZZ)$  be an $n$-periodic $\ell$-pregroup. 
    \begin{enumerate}[label = \textup{(\arabic*)}]
        \item For all $a\in \loc{L}$, $\supp(a) = \supp(a^\ell) = \supp(a^r)$. 
        \item If $\Grp{\m L}$ is totally ordered, then for all $a,b \in \loc{L}\setminus \{1\}$, $\supp(a) = \supp(b)$.
        \item  If $\Grp{\m L}$ is totally ordered and $\Grp{\m L} \neq \m{L}$, then $\h(\overline{a}_i) = \h(\overline{a}_j)$ and $\g(\overline{a}_i) = \g(\overline{a}_j)$, for all $a \in \loc{L}$ and $i,j \in \supp(a)$. In particular, $\overline{h}_i=\overline{h}_j$, for all $h \in \grp{(\loc{L})}$ and $i,j \in \supp(h)$.
        \item If $\Grp{\m L}$ is totally ordered 
        and $\Grp{\m L} \neq \m{L}$, then $\loc{(\Grp{\m{L}})} \cong \ZZ$. 
        \item If $\Grp{\m L}$ is totally ordered, then for each $a\in \loc{L}$ with $a\notin \grp{L}$ and for all $i\in \supp(a)$, we have $\g(\overline{a}_i) \neq \overline{a}_i$.
    \end{enumerate}
\end{lemma}

\begin{proof}
    (1) For each $i\in I$, we have $\overline{a}_i =1$ iff $\overline{a^\ell}_i =1$.

    (2)
   Suppose that $\Grp{\m L}$ is totally ordered.
    If $a,b \in \loc{L}\setminus \{1\}$ are such that that there exists $i \in \supp(a)\setminus \supp(b)$, then $\overline{a}_i \neq 1$ and $\overline{b}_i = 1$, hence also $\overline{\h(b)}_i = 1$. 

    If $\overline{a_i}$ is not invertible, then, since $\overline{a}_i \neq 1$, $\overline{(aa^\ell)}_i > 1$ and $i\in \supp(aa^\ell)$, yielding $\overline{\h(aa^\ell)}_i>1$. Otherwise, $\overline{a}_i$ is invertible, and, since $\overline{a}_i \neq 1$, either $\overline{a}_i < 1$ or $\overline{a}_i > 1$, yielding $\overline{\h(a)}_i > 1$ or  $\overline{\h(a^\ell)}_i > 1$. Thus we may assume, by possibly replacing $a$ with $aa^\ell$ or $a^\ell$, that  $\overline{\h(a)}_i > 1$ and $\overline{\h(b)}_i = 1$.
    Since $b \neq 1$ there exists  $j\in \supp(b)$, so arguing as above, without loss of generality $\overline{\h(b)}_j > 1$. Note that there exists a $k\in \NN$ such that $\overline{\h(b)^k}_j > \overline{\h(a)}_j$  and also $\overline{\h(b)^k}_i = 1 < \overline{\h(a)}_i$ (using that $\grp{(\m{F}_n(\ZZ)^J)}$ is isomorphic to $\ZZ^J$). Hence, $\h(a)$ and $\h(b)$ are incomparable which contradicts the assumption that $\Grp{\m L}$ is totally ordered.

    (3) 
    Since $\Grp{\m L} \neq \m{L}$, by Lemma~\ref{l: local-idempotent}$,\m L$ has a strictly positive, local, idempotent element $d$. 
    Now, by Lemma~\ref{l:idempotent-slope1}, $e = \slt(d)$  is also idempotent and strictly positive, and for each $i\in \supp(e)$, $\h(\overline{e}_i) = \cn{s}$; also, if for some $i,j\in \supp(a)$, $\g(\overline{a}_i)  \neq \g(\overline{a}_j)$, then, without loss of generality, $\g(\overline{a}_i) = \cn{s}^k$ with $k\neq 0$. Therefore, since, by (2), $\supp(e) = \supp(a)$, we have $i \notin \supp(\h(e)^{-k}\g(a))$ and $j\in \supp(\h(e)^{-k}\g(a))$, contradicting (2). Similarly we get $\h(\overline{a}_i) = \h(\overline{a}_j)$, for all $i,j \in \supp(a)$.

     The statement about $h\in \grp{(\loc{L})}$ follows from the fact that the image of $\h$ consists precisely of the invertible elements.

    (4) For $h \in \Grp{(\loc{L})}$, we have that $\overline{h} \in  \Grp{\m{F}_n(\ZZ)^J}=\ZZ^J$. Since $\overline{h}_i=\overline{h}_j$, for all  $i,j \in \supp(h)$, and using (2), we get that 
    \[
    \Grp{(\loc{L})} \subseteq \{(x_i)_{i \in S}\in \ZZ^{S} : \exists k \in \ZZ, \forall i\in S,\,  x_i=k\} \times \{id\}^{J-S}\cong \ZZ,
    \]
    where $S=  \supp(h)$. Moreover, since $\Grp{\m L} \neq \m{L}$, we have that $\m L$ is non-trivial, hence there is a strictly positive idempotent element $a \in \loc{L}$; thus $\h(a)$ is a strictly positive invertible element, so $\loc{(\Grp{\m{L}})}$ is not trivial either. Therefore, $\loc{(\Grp{\m{L}})} \cong \ZZ$.
    
    (5) Note that, by (2),  $\supp(a\g(a)^{-1}) = \supp(a)$ or $\supp(a\g(a)^{-1}) = \emptyset$.
\end{proof}

\begin{remark}\label{r: discreteskeleton}
It follows from Lemma~\ref{l:support-local}(4) that the totally ordered group skeleton of any  proper periodic $\ell$-pregroup is discretely ordered. So if a totally ordered $\ell$-group is densely ordered, it cannot be the group skeleton of a proper periodic $\ell$-pregroup.
\end{remark}

For a proper $n$-periodic $\ell$-pregroup $\m{L} \leq \m{F}_n(\m J \overrightarrow{\times} \ZZ)$ such that $\Grp{\m L}$ is totally ordered we define $\loc{\supp}(\m{L}): = \supp(a) \subseteq J$, where $a \in \loc{L}\setminus \{1\}$. Note that, by Lemma~\ref{l:support-local}(2), $\loc{\supp}(\m{L})$ does not depend on the choice of $a$.

\begin{lemma}\label{l:fin-gen}
       If $\m{L}$ is a finitely generated proper periodic $\ell$-pregroup such that $\Grp{\m L}$ is a totally ordered abelian $\ell$-group, then $\Grp{\m{L}}$ is finitely generated. 
\end{lemma}

\begin{proof}
   Let $F \subseteq L$ be a finite generating set for $\m{L}$ and $\m{L} \leq \m{F}_n(\m J \overrightarrow{\times} \ZZ)$ a representation of $\m L$. Then, by Lemma~\ref{l:local-subuniverse}(3), there is a surjective homomorphism $\pi \colon \m{L} \to \Grp{\m L}/\loc{(\Grp{\m L})}$, and $\pi[F]$ is a finite generating set for $\Grp{\m L}/\loc{(\Grp{\m L})}$. Moreover, by Lemma~\ref{l:support-local}(4), $\loc{(\Grp{\m L})}\cong \ZZ$, so $\loc{(\Grp{\m L})}$ is one-generated by an element $a$. Let $F' \subseteq \grp{L}$ consist of a representative of each element in $\pi[F]$; in particular $F'$ is finite. Then every element of $\grp{L}$ is a product of an element in $\loc{(\grp{L})}$ and an element in $\langle F' \rangle$. So $F'\cup\{a\}$ is a finite generating set of $\grp{L}$.
\end{proof}

\begin{lemma}\label{l:loc-invariant-grp}
    If $\m{L}$ is a finitely generated proper periodic $\ell$-pregroup such that $\Grp{\m L}$ is a totally ordered abelian $\ell$-group, then we have $\Grp{\m{L}} \cong (\Grp{\m{L}}/\loc{(\Grp{\m{L}})}) \overrightarrow{\times} \ZZ$, for any representation of $\m L$.
\end{lemma}

\begin{proof}
    Let $\m{L} \leq \m{F}_n(\m J \overrightarrow{\times} \ZZ)$ be a representation of $\m L$. By Lemma~\ref{l:support-local}(4) we have $\loc{(\Grp{\m{L}})} \cong \ZZ$ and, by Lemma~\ref{l:fin-gen}, $\loc{(\Grp{\m L})}$ is finitely generated. Moreover, $\loc{(\Grp{\m{L}})}$ is a convex normal subalgebra of $\Grp{\m{L}}$, by Lemma~\ref{l:local-subuniverse}(2). Thus, by Proposition~\ref{p:Zconv-lex}, $\grp{L} \cong (\Grp{\m{L}}/\loc{(\Grp{\m{L}})}) \overrightarrow{\times} \ZZ$
\end{proof}

The following theorem shows that if a proper $n$-periodic $\ell$-pregroup has  totally ordered group skeleton, then all local elements follow a single common pattern in their local coordinates. This forces the local subalgebra to be isomorphic to $\m{F}_k(\ZZ)$, for some $k \mid n$. 

We say that an $\ell$-pregroup is of \emph{periodicity $n$} if it is $n$-periodic and it contains an element of periodicity $n$. 

\begin{remark}\label{r:per-of-alg}
    Note that, by Lemma~\ref{l: spfpl}, a periodic $\ell$-pregroup $\m{L}$ is of periodicity $n$ if and only if $\loc{\m{L}}$ is of periodicity $n$, where $\loc{\m{L}}$ is computed with respect to any representation of $\m L$.
\end{remark}

\begin{theorem}\label{t:Fn-FSI-local}
    Let $\m{L} \leq \m{F}_n(\m J \overrightarrow{\times} \ZZ)$, where $\m J$ is a chain, be a proper  $\ell$-pregroup of periodicity $n$ such that $\Grp{\m L}$ is a totally ordered  $\ell$-group and let $i\in S:= \loc{\supp}(\m{L})$. Then there exists a sequence $(k_j)_{j\in S} \subseteq \ZZ_n$ such that 
    for each $c\in \loc{L}$ and $j\in S$, $\overline{c}_j = \overline{c}_i^{[k_j]}$, and $k_i=0$. In particular, $\loc{\m{L}} \cong \m{F}_n(\ZZ)$.  
\end{theorem}

\begin{proof}
    We let $i \in S$ and we first prove some claims.
    
     \underline{Claim 1.}    If  $b\in \loc{L}$ is a strictly positive idempotent, then there exists a sequence $(k_j)_{j\in S} \subseteq \ZZ_n$ such that for all $j \in S$, $\overline{b}_j = \overline{b}_i^{[k_j]}$.
    
    By way of contradiction, we assume that there exists a $j\in S = \supp(b)$ such that for all $k\in \ZZ_n$, $\overline{b}_i \neq \overline{b}_j^{[k]}$. By Lemma~\ref{l:core-lemma}(3), without loss of generality, there exist $l_1,\dots,l_m \in \ZZ_n$ such that $\overline{b}_i^{[l_1]} \mt \dots \mt \overline{b}_i^{[l_m]} > \overline{b}_j^{[l_1]} \mt \dots \mt \overline{b}_j^{[l_m]} =1$. 
    Since $b$ is local,  $c := b^{[l_1]} \mt \dots \mt b^{[l_m]}$ is local, as well; also, by the above, we have that $\overline{c}_i > 1$ and $\overline{c}_j = 0$, hence $j \notin \supp(c)$. On the other hand,
    by Lemma~\ref{l:support-local}(2) we get $\supp(b) = \supp(c)$, hence $j \in \supp(c)$, a contradiction.  
    
     \underline{Claim 2.}      There exists  $b\in \loc{L}$ such that for each $j\in S$, $\overline{b}_j$ is an $n$-atom.
    
    Let $k = \max\{\per(\overline{a}_i) : a \in L,i\in J\}= \max\{\per(\overline{a}) : a \in L\}$. 
    Since $\m{L}$ is proper, it contains a non-invertible element $a$, so $\per(a) \neq 1$. By Lemma~\ref{l: per}, $\lcm \{\per(\overline{a}) : a \in L\}\neq 1$, so $\max\{\per(\overline{a}) : a \in L\} \neq 1$, hence $k>1$. We want to show that $k=n$. So let $a\in L$ and $i\in J$ such that $\overline{a}_i$ has periodicity $k>1$, hence $a$ is non-invertible, and let $b:= \g(a)^{-1}a$.  By Corollary~\ref{c:flat-g}, $   \overline{b}_i:= \overline{(\g(a)^{-1}a)}_i$ is positive,  and flat of periodicity $k$,
    thus, by Theorem~\ref{t:Fn-generation}, we have $\langle \overline{b}_i \rangle =\m{F}_k(\ZZ)$.
    Let $c_i \in \mathrm{F}_k(\ZZ)$ be a $k$-atom, 
    and let $t(x)$ denote the unary term witnessing that $c_i \in\langle \overline{b}_i \rangle $, i.e., $t(\overline{b}_i)=c_i$. Note that for $c:=t(b)$ we have $\overline{c}_i=\overline{t(b)}_i=t(\overline{b}_i)=c_i$. 
    So there is a $c \in L$ with $\overline{c}_i$ a 
    $k$-atom. 
     Also, by Lemma~\ref{l: spfpl}, $b:= \g(a)^{-1}a$ is local,  so $c=t(b)$ is also local.
    Now $c' := cc^\ell$ is a positive idempotent, by Remark~\ref{r: xxl}, and $\overline{c}'_i =\overline{cc^\ell}_i =\overline{c}_i\overline{c}_i^\ell = \overline{c}_i$ (since $\overline{c}_i$ is positive idempotent). Also, since $c$ is local, so is $c'$, so Claim~1 yields that $\overline{c}'_j$ is a 
    $k$-atom for each $j\in \supp(c')$. If $k<n$, then, by the maximality of $k$ and, since $\m{L}$ is of periodicity $n$, there exists a $d \in L$ of periodicity $n$, so by Lemma~\ref{l: per} there exist $l \in \NN$ and  $i\in J$ such that $\per(\overline{d}_i)=l$, where $l\nmid k$ and $k\nmid l$. Arguing as above we obtain a $d' \in L$ such that $\overline{d'}_i$ is an $l$-atom for each $i\in J$. But then, by Proposition~\ref{p: maxper}, there exist $s,t\in \ZZ$ such that $\overline{(c'^{[s]}\jn d'^{[t]})}_i$ has periodicity 
    $\lcm(k,l) > k$ for each $i\in J$, contradicting the maximality of $k$. Hence $k=n$ and the claim is proved.

    \underline{Claim 3.}      There exists a sequence $(k_j)_{j\in S}$ such that 
        for each $c\in \loc{L}$ and $j\in S$, $\overline{c}_j = \overline{c}_i^{[k_j]}$.
    
    By Claim~1 together with Claim~2, there exists a $b\in \loc{L}$ and a sequence $(k_j)_{j\in S}$ such that $\overline{b}_i$ is an $n$-atom and for each $j \in S$, $\overline{b}_j = \overline{b}_i^{[k_j]}$. In particular, for each $\overline{a}\in \mathrm{F}(\ZZ)^J$ with $\overline{a}_j = \overline{a}_i^{[k_j]}$ we have $\overline{a} \in \langle b \rangle$, by Theorem~\ref{t:Fn-generation}.

    Suppose for a contradiction that there exists a $c \in \loc{L}$ and $j\in S$ such that $\overline{c}_j \neq \overline{c}_i^{[k_j]}$; by Lemma~\ref{l:support-local}(3), $c$ is not invertible. Thus, by Lemma~\ref{l: spfpl} and considering $\g(c)^{-1}c$,  we may assume that $c$ is strictly positive and flat, while remaining local. Note that $c$ is local iff $c^{[k]}$ is, so in view of Remark~\ref{r: flat via automorohism} and the fact that $c$ is flat, possibly by replacing $c$ by $c^{[k]}$ for a suitable $k \in \ZZ$  
    we may assume that $\overline{c}_j[\ZZ_n] \subseteq \ZZ_n$ and that there is a  $q\in \ZZ_n$ such that $\overline{c}_i^{[k_j]}(q) \neq \overline{c}_j(q)$. We set $p:= \overline{c}_j(q) \geq q$ and distinguish two cases.

    If $\overline{c}_i^{[k_j]}(q) < p$, then we let  $\overline{d}\in \mathrm{F}(\ZZ)^J$ be such that $\overline{d}_j(x) = p$ for $x \in [q,p]$ and $\overline{d}_j(x) = x$ for $x \in [0,q-1]\cup [p+1,n-1]$ extended $n$-periodically {with $\overline{d}_i = \overline{d}_j^{[-k_j]}$} and $\overline{d}_l = \overline{d}_i^{[k_l]}$ for each $l\in \ZZ$. 
    So, in particular, 
    $\overline{d} \in \langle \overline{b} \rangle$, 
    as shown above.
    Therefore, by setting $d:=(id, \overline{d})$, we get $d \in \langle b \rangle \subseteq \loc{L}$. Moreover, we have $(\overline{d}_j \mt \overline{c}_j)(q) = p$ and all other heights of $\overline{d}_j \mt \overline{c}_j$ are at most $p-q$, thus $\overline{\h(d \mt c)}_j(x) = x + p-q$, for all $x \in \ZZ$. On the other hand we have $(\overline{d}_j \mt \overline{c}_i^{[k_j]})(q) < p$ and all other heights of $\overline{d}_j \mt \overline{c}_i^{[k_j]}$ are strictly smaller than $p-q$, thus  for all $x\in \ZZ$,
    \[
    \overline{\h(d \mt c)}_i(x) = \h(\overline{d}_j^{[-k_j]}\mt \overline{c}_i)(x) =  \h(\overline{d}_j \mt \overline{c}_i^{[k_j]})(x)  < x + p-q.
    \]
    Hence $\overline{\h(d \mt c)}_i(x) \neq \overline{\h(d \mt c)}_j(x)$, contradicting Lemma~\ref{l:support-local}(3).

    Similarly we obtain a contradiction if $\overline{c}_i^{[k_j]}(q) > p$, so this concludes the proof of Claim~3, thus establishing the first part of the lemma.
    
    For the second part note that the map $\pi\colon \loc{L} \to \mathrm{F}_n(\ZZ)$, $\pi(c) = \overline{c}_i$ is an isomorphism, by Claim~3.
\end{proof}

The following lemma links conjugation by a group element to the automorphism $a \mapsto a^{[m]}$, for $m \in \ZZ$, in the setting of certain $n$-periodic $\ell$-pregroups, exactly as it happens in $\m F_n(\ZZ)$; see Lemma~\ref{l:b-conjugation}(1).

\begin{lemma}\label{l:conj-grp-elements}
    Let $\m{L} \leq \m{F}_n(\m J \overrightarrow{\times} \ZZ)$, where $\m J$ is a chain, be a proper $n$-periodic $\ell$-pregroup such that $\Grp{\m L}$ is a totally ordered  $\ell$-group. Then for each $f\in \grp{L}$, there is an $m_f\in \ZZ$ such that for each $a\in \loc{L}$, $faf^{-1} = a^{[m_f]}$. 
\end{lemma}

\begin{proof}
     Let  $S= \loc{\supp}(\m{L})$, $i\in S$, and $(k_j)_{j\in S}$ the sequence associated with $i$ as in Theorem~\ref{t:Fn-FSI-local}; we define 
     $m_f :=  \overline{f}_{\widetilde{f}^{-1}(i)}(0) +k_{\widetilde{f}^{-1}(i)}$.
     Using the definition of multiplication in a wreath product we have
     \begin{align*}
    faf^{-1} &=(\widetilde{f},\overline{f})(id,\overline{a})(\widetilde{f},\overline{f})^{-1}    \\
                            &=(\widetilde{f} \circ id,(\overline{f} \otimes id )\cdot \overline{a})(\widetilde{f}^{-1},\overline{f}^{-1} \otimes\widetilde{f}^{-1})  \\
                            &=(\widetilde{f},\overline{f}\cdot \overline{a})(\widetilde{f}^{-1},\overline{f}^{-1} \otimes\widetilde{f}^{-1})  \\
                            &=(\widetilde{f}\circ \widetilde{f}^{-1},((\overline{f}\cdot \overline{a}) \otimes \widetilde{f}^{-1})\cdot (\overline{f}^{-1} \otimes\widetilde{f}^{-1}) ) \\
                            &=(id,((\overline{f}\cdot \overline{a}) \otimes \widetilde{f}^{-1})\cdot (\overline{f}^{-1} \otimes\widetilde{f}^{-1}) ).
     \end{align*}                           
     So
     \begin{align*}
    \overline{(faf^{-1})}_i &=   ((\overline{f}\cdot \overline{a}) \otimes \widetilde{f}^{-1})_i\cdot (\overline{f}^{-1} \otimes\widetilde{f}^{-1})_i                    \\
                            &= (\overline{f}\cdot \overline{a})_{ \widetilde{f}^{-1}(i)}\cdot \overline{f}^{-1}_{ \widetilde{f}^{-1}(i)}\\
                            &= \overline{f}_{\widetilde{f}^{-1}(i)} \overline{a}_{\widetilde{f}^{-1}(i)} \overline{f}^{-1}_{\widetilde{f}^{-1}(i)},
    \end{align*}
  Recall that by Lemma~\ref{l:b-conjugation}, for every $c \in F_n(\ZZ)$ and $k \in \ZZ$, we have $\cn{s}^k c \cn{s}^{-k}=c^{[k]}$. Since each invertible element $g$ of $\m F_n(\ZZ)$ is of the form $g=\cn{s}^k$ for some $k$, and since $g(0)=\cn{s}^k(0)=k$, we get that $g c g^{-1}=c^{[g(0)]}$, for every $c \in F_n(\ZZ)$.  Using this fact
        for the second equality and Theorem~\ref{t:Fn-FSI-local} for the third equality we obtain
    \begin{align*}
    \overline{(faf^{-1})}_i &= \overline{f}_{\widetilde{f}^{-1}(i)} \overline{a}_{\widetilde{f}^{-1}(i)} \overline{f}^{-1}_{\widetilde{f}^{-1}(i)} \\
                            &= \overline{a}_{\widetilde{f}^{-1}(i)}^{[\overline{f}_{\widetilde{f}^{-1}(i)}(0)]} \\
                            &= \overline{a}_i^{[\overline{f}_{\widetilde{f}^{-1}(i)}(0)+k_{\widetilde{f}^{-1}(i)}]} \\
                            &= \overline{a}_i^{[m_f]},
    \end{align*}
    Thus, since $faf^{-1}$ is local, two applications of Theorem~\ref{t:Fn-FSI-local} yield that, for all $j \in S$,
    \[
    \overline{(faf^{-1})}_j = \overline{(faf^{-1})}_i^{[k_j]} = \overline{a}_i^{[m_f+ k_j]} = \overline{a}_j^{[m_f]} \qedhere
    \]
\end{proof}

Now we are able to prove the main result of this subsection.

\begin{theorem}\label{t:lexprod-fsi-Fn}
    If $\m{L}$  is a finitely generated proper $n$-periodic $\ell$-pregroup where $\Grp{\m L}$ is a  totally ordered abelian $\ell$-group, 
    then $\m{L} \cong \m{H} \overrightarrow{\times}\m{F}_k(\ZZ)$ for some totally ordered abelian $\ell$-group $\m{H}$ and some $k \mid n$.
\end{theorem}

\begin{proof}
    Since $\m{L}$ is a proper $n$-periodic $\ell$-pregroup, by Proposition~\ref{p: representation1}, it has a representation $\m{L} \leq \m{F}_n(J \overrightarrow{\times} \ZZ)$, where $\m J$ is a chain, and  there is a $k>1$ with $k\mid n$ such that $\m{L}$ is of periodicity $k$; hence we  have $\m{L} \leq \m{F}_k(J \overrightarrow{\times} \ZZ)$ and we want to show that  $\m{L} \cong \m{H} \overrightarrow{\times}\m{F}_k(\ZZ)$. So, without loss of generality we may assume that $k=n$.
    By Lemma~\ref{l:support-local}(4), we have $\loc{(\Grp{\m{L}})} \cong \ZZ$ and, by Lemma~\ref{l:loc-invariant-grp}, we have
    $\Grp{\m{L}} \cong (\Grp{\m{L}}/\loc{(\Grp{\m{L}})}) \overrightarrow{\times} \ZZ$. So, by  Lemma~\ref{l:inner-lex}
    there exists $\m H \leq \Grp{\m{L}}$ such that $\m{H} \cong \Grp{\m{L}}/\loc{(\Grp{\m{L}})}$ and   
    $\m{H} \overrightarrow{\times}  \loc{(\Grp{\m{L}})} \cong \Grp{\m{L}}$, via the map $(h,n) \mapsto hn$.
    By Lemma~\ref{l:fin-gen}, $\Grp{\m{L}}$ is finitely generated, so $\m H \cong \Grp{\m{L}}/\loc{(\Grp{\m{L}})} $ is finitely generated, as well; let $F$ be a set of independent generators of $\m{H}$. By Lemma~\ref{l:conj-grp-elements}, for each $f \in F$ there exists an $m_f \in \ZZ$ such that for each $a\in \loc{L}$, $faf^{-1} = a^{[m_f]}$. 
    Moreover, by Theorem~\ref{t:Fn-FSI-local}, we have $\loc{\m{L}} \cong \m{F}_n(\ZZ)$; abusing notation, we denote by
    $\cn{s}$, as well, the element of $\loc{(\grp{L})}$ corresponding to the positive generator $\cn{s}$ of $\ZZ$, hence $\loc{(\grp{L})}=\langle \cn{s}\rangle$, and define $F' = \{f\cn{s}^{n-m_f} : f\in F\}$.
    By Lemma~\ref{l:b-conjugation}(1), we have for every $a \in \m{F}_n(\ZZ)$,  $\cn{s}^m a \cn{s}^{-m} = a^{[m]}$.  Since $\loc{\m{L}} \cong \m{F}_n(\ZZ)$, for every $a\in \loc{L}$, we have $\cn{s}^m a \cn{s}^{-m} = a^{[m]}$, as well, hence we also have 
    \begin{align*}
    f\cn{s}^{n-m_f} a (f\cn{s}^{n-m_f})^{-1} &= f\cn{s}^{n-m_f} a \cn{s}^{-n+m_f}f^{-1} =f a^{[n-m_f]}f^{-1} \\
    &= (a^{[n-m_f]})^{[m_f]}=a^{[n]} = a.
    \end{align*}
    So every element of $F'$ commutes with every element of $\loc{\m{L}}$. Let $\m{H}'$ be the $\ell$-group generated by $F'$; we will show that  $\m{H}'\overrightarrow{\times} \loc{\m L} \cong \m{L}$ via $(h,a) \mapsto ha$. For this we check that (i)-(iii) of Lemma~\ref{l:inner-lex} are satisfied; we have already showed that every generator of $H'$ commutes with every element of $\loc{\m{L}}$, so every element of $H'$ commutes with every element of $\loc{\m{L}}$; i.e., (i) holds. 
   First we show
    \begin{equation}\tag{$\ast$}\label{eq:triv-intersection}
         H' \cap \grp{(\loc{L})} = \{1\}.
    \end{equation}
    If $a \in H' \cap \grp{(\loc{L})}$, then there exist $f_1,\dots, f_k,g_1,\dots, g_l \in F$ such that 
    \begin{align*}  
    a &= \cn{s}^{n-m_{f_1}}f_1\cdots \cn{s}^{n-m_{f_k}}f_k \cn{s}^{-n+m_{g_1}} g_1^{-1} \cdots \cn{s}^{-n+m_{g_l}}g_l^{-1} \\
      &= \cn{s}^{(k-l)n-m_{f_1}-\dots - m_{f_k} + m_{g_1} + \dots + m_{g_l}}f_1\cdots f_k g_1^{-1} \cdots g_l^{-1},
    \end{align*}
    where we used commutativity.
    So, since $a \in \grp{(\loc{L})}$, $f_1\cdots f_k g_1^{-1} \cdots g_l^{-1} =1$. Moreover, since $F$ is an independent generating set, we get that $\{g_1, \dots, g_l \}= \{f_1, \dots, f_k\}$ as multisets, hence $k=l$ and
    $(k-l)n-m_{f_1}-\dots - m_{f_k} + m_{g_1} + \dots + m_{g_l}=0+0=0$; thus $a =1$.

    If $a\in L$, then $\g(a) \in \grp{L}$  and, by Lemma~\ref{l: spfpl},  $\g(a)^{-1}a \in \loc{L}$. We have already mentioned that $\m{H} \overrightarrow{\times}  \loc{(\Grp{\m{L}})} \cong \Grp{\m{L}}$ via the map $(h,m) \mapsto hm$, so $\g(a)=hm$, for some $h \in H$ and $m \in \loc{(\grp{L})}$. Since $F$ is a generating set for $H$,  there are $f_1, \ldots, f_k \in F$ such that $h=f_1 \cdots f_k$. Also, since $m \in\loc{(\grp{L})}=\langle \cn{s}\rangle$, there exists $l \in \ZZ$ such that $m=\cn{s}^l$; hence 
    \[
    m=\cn{s}^l=\cn{s}^{n-m_{f_1}}\cdots \cn{s}^{n-m_{f_k}}\cdot \cn{s}^{l-kn + m_{f_1}+ \ldots + m_{f_k}}.
    \]
    Since  $\m{H} \overrightarrow{\times}  \loc{(\Grp{\m{L}})} \cong \Grp{\m{L}}$, every element of $H$ (hence also of $F$) commutes with every element of $\loc{(\grp{L})}$; thus 
    \begin{align*}
        \g(a)=hm &=f_1 \cdots f_k\cdot \cn{s}^{n-m_{f_1}}\cdots \cn{s}^{n-m_{f_k}}\cdot \cn{s}^{l-kn + m_{f_1}+ \ldots + m_{f_k}} \\
        &= f_1  \cn{s}^{n-m_{f_1}}\cdots f_k \cn{s}^{n-m_{f_k}}\cdot \cn{s}^{l-kn + m_{f_1} + \ldots + m_{f_k}} \in H' \cdot \loc{(\grp{L})}.
    \end{align*}
    Therefore, 
    $a = \g(a)(\g(a)^{-1}a) \in  H' \cdot \loc{(\grp{L})}\cdot \loc{L}\subseteq H' \cdot \loc{L}$. 
     Now suppose that $c_1,c_2 \in \loc{L}$ and $g_1,g_2 \in H'$ are such that $c_1g_1 = c_2g_2$; so, $c_1g_1g_2^{-1} =c_2\leq c_2$  and by residuation $ g_1g_2^{-1} \leq c_1^{r}c_2$. Since $\m H'$ is totally ordered, without loss of generality we may assume that $g_2 \leq g_1$; so $1 \leq g_1g_2^{-1}$.   So  $1 \leq g_1g_2^{-1} \leq c_1^{r}c_2$ and $1,  c_1^{r}c_2 \in \loc{L}$; moreover by Lemma~\ref{l:local-subuniverse}(1), $\loc{L}$ is convex, hence $g_1g_2^{-1} \in \loc{L}$. Since we also have $g_1g_2^{-1} \in H'$, \eqref{eq:triv-intersection} yields $g_1= g_2$, hence $c_1 = c_2$, as well. Thus (ii) holds.

    Finally for (iii)  suppose that $g_1,g_2 \in H'$, $a_1,a_2 \in \loc{L}$ and  $g_1a_1 < g_2a_2$. As we argued above, the elements of $F$ commute with the elements of $\langle \cn{s} \rangle$; so every element of $H'$, which is a product of elements of $F' \subseteq F \cdot \langle \cn{s} \rangle$, is also a product of elements of $F$ (hence an element of $H$) times a product of elements of $\langle \cn{s} \rangle$ (hence itself an element of $\langle \cn{s} \rangle$), thus $H' \subseteq H \cdot \langle \cn{s} \rangle$. Therefore,  $g_i = m_i h_i$ with $h_i \in H$, $m_i \in \langle \cn{s} \rangle=  \grp{(\loc{L})}$ for $i=1,2$, so $h_1m_1a_1 < h_2m_2a_2$. If $h_1 = h_2$, then $h_1h_2^{-1} =1$, so $g_1g_2^{-1}=m_1h_1h_2^{-1}m_2^{-1}= m_1m_2^{-1}\in H\cap \langle \cn{s} \rangle$, hence $g_1 = g_2$, by \eqref{eq:triv-intersection}; thus also $m_1 = m_2$, hence $a_1 < a_2$. 
    If $h_1 <h_2$, then $g_1 = h_1m_1 < h_2m_2 = g_2$, since $\Grp{\m{L}} = \m{H}\overrightarrow{\times} \langle \cn{s} \rangle$. Moreover, if $h_2 < h_1$, then, by Remark~\ref{r:grp-conucleus} and since $\Grp{\m{L}} = \m{H}\overrightarrow{\times} \langle \cn{s} \rangle$,
    \[
    h_2m_2 a_2 \leq \h(h_2m_2a_2) = h_2m_2\h(a_2) < h_1m_1\g(a_1) = \g(h_1m_1a_1) \leq h_1m_1a_1,
    \]
    contradicting the fact that $h_1m_1a_1 < h_2m_2a_2$.

    Conversely, if $g_1 < g_2$, then either $h_1 < h_2$, or ($h_1 = h_2$ and $m_1 < m_2$), since $\Grp{\m{L}} = \m{H}\overrightarrow{\times} \langle \cn{s} \rangle$. As seen above the second case cannot happen, so $h_1 < h_2$. But, then, by Remark~\ref{r:grp-conucleus} and since $\Grp{\m{L}} = \m{H}\overrightarrow{\times} \langle \cn{s} \rangle$, 
    \[
    h_1m_1 a_1 \leq \h(h_1m_1a_1) = h_1m_1\h(a_1) < h_2m_2\g(a_2) = \g(h_2m_2a_2) \leq h_2m_2a_2,
    \]
    yielding $h_1m_1a_1 <h_2m_2a_2$. Finally, if $g_1 = g_2$, and $a_1 < a_2$, then clearly $g_1a_1 < g_2a_2$.
\end{proof}

\subsection{Axiomatization and subvarieties}

We now have all the ingredients to prove the axiomatization result for all varieties of the form $\vr(\m{F}_n(\ZZ))$, as well as for all possible joins of such varieties. We also prove that these joins (other than the top variety) form an ideal in the subvariety lattice and characterize fully its lattice structure.

\begin{proposition}\label{p: main_FW}
Every $n$-periodic $\ell$-pregroup with a totally ordered abelian group skeleton is in $\vr(\m{F}_n(\ZZ))$.
\end{proposition}

\begin{proof}
  If $\m L$ is $n$-periodic $\ell$-pregroup with a totally ordered abelian group skeleton, then each finitely generated subalgebra $\m K$ of $\m L$  is also an $n$-periodic $\ell$-pregroup with a totally ordered abelian group skeleton. By Theorem~\ref{t:lexprod-fsi-Fn}, $\m{K} \cong \m{H} \overrightarrow{\times}\m{F}_k(\ZZ)$ for some totally ordered abelian $\ell$-group $\m{H}$ and some $k \mid n$. By Proposition~\ref{p:Fn-closed-lexprod}, $\vr(\m{F}_n(\ZZ))$  is closed under lexicographic products with totally ordered abelian $\ell$-groups and, since $\m{F}_k(\ZZ) \in \vr(\m{F}_n(\ZZ))$, we get $\m{K} \cong \m{H} \overrightarrow{\times}\m{F}_k(\ZZ) \in \vr(\m{F}_n(\ZZ))$. Since $\m L$  
 embeds into an ultrapower of its finitely generated subalgebras, we get that $\m L \in \vr(\m{F}_n(\ZZ))$.    
\end{proof}

\begin{theorem}\label{t: axiomatization}
    For each $n\in \ZZ^+$, the variety $\vr(\m{F}_n(\ZZ))$ is axiomatized relative to $\va{LP}_n$ by the equation $x\g(y)^n \approx \g(y)^n x$.
\end{theorem}

\begin{proof}
   By Lemma~\ref{l:b-conjugation}(3) the element $\cn{s}^n$ is central in $\m{F}_n(\ZZ)$, hence so is every power of it. Since $\Grp{\m{F}_n(\ZZ)}=\langle \cn{s} \rangle$, for all  $g \in \grp{\mathrm{F}_n(\ZZ)}$ there exists $k \in \ZZ$ such that $g=\cn{s}^k$, so  $xg^n=x(\cn{s}^n)^k=(\cn{s}^n)^kx=g^nx$. So, the equation $x\g(y)^n \approx \g(y)^n x$ holds in $\m{F}_n(\ZZ)$, hence also in  $\vr(\m{F}_n(\ZZ))$.

   For the converse direction, it is enough to show that every  finitely subdirectly irreducible $\m L\in \va{LP}_n$ that satisfies the equation $x\g(y)^n \approx \g(y)^n x$ is in $\vr(\m{F}_n(\ZZ))$. By Corollary~\ref{c:fsi-commutative} for such an $\m L$ we have that $\Grp{\m L}$ is totally ordered and abelian. 
   Hence, by Proposition~\ref{p: main_FW}, $\m L \in \vr(\m{F}_n(\ZZ))$.
\end{proof}

\begin{remark} \label{r: half}
We note that the equation $x\g(y)^n \approx \g(y)^n x$ is equivalent to $x\g(y)^n \leq \g(y)^n x$; indeed, from the latter we get $\g(y)^n x \leq x\g(y)^n$, by multiplying both sides by $\g(y)^n$ on the left and by $\g(y)^{-n}$ on the right.  
Also, the $n$-periodicity equation $x  \approx x^{\ell^{2n}}$ is equivalent to  $x \leq x^{\ell^{2n}}$; indeed,  by substituting $x^\ell$ for $x$ in the latter, we get $x^\ell \leq x^{\ell^{2n+1}}$, and by applying the order-reversing map $^r$ we get $x^{\ell^{2n}} \leq x$. Therefore,  for each $n\in \ZZ^+$,  $\vr(\m{F}_n(\ZZ))$ is axiomatized relative to $\ell$-pregroups by the equations $x\g(y)^n \leq \g(y)^n x$ and $x \leq x^{\ell^{2n}}$, or equivalently $1 \leq \g(y)^n x \g(y)^{-n} x^\ell$ and $1 \leq x^{\ell^{2n}}x^\ell$, by residuation. 
\end{remark}

\begin{corollary}\label{c: FSI VFnZ}
    The FSIs  of $\vr(\m{F}_n(\ZZ))$ are exactly the non-trivial $n$-periodic $\ell$-pregroups with a totally ordered abelian group skeleton. In particular, each such FSI that is not an $\ell$-group is subdirectly irreducible. 
    So the FSIs of $\vr(\m{F}_n(\ZZ))$ are exactly the  SIs of $\vr(\m{F}_n(\ZZ))$ and the non-trivial totally ordered abelian $\ell$-groups.
\end{corollary}

\begin{proof}
   By Theorem~\ref{t: axiomatization}, $\vr(\m{F}_n(\ZZ))$ is axiomatized relative to $\va{LP}_n$ by $x\g(y)^n \approx \g(y)^n x$, so by Corollary~\ref{c:fsi-commutative} its finitely subdirectly irreducibles are exactly the non-trivial members of $\vr(\m{F}_n(\ZZ))$ with a totally ordered abelian group skeleton, which, by Proposition~\ref{p: main_FW}, are  exactly the non-trivial  $n$-periodic $\ell$-pregroups with a totally ordered abelian group skeleton.
   This shows the first part.

For the second part, note that  if $\m L$ is a finitely subdirectly irreducible member of $\vr(\m{F}_n(\ZZ))$ and it is not an $\ell$-group, then it is a proper $n$-periodic $\ell$-pregroup, so $\loc{(\Grp{\m{L}})} \cong \ZZ$ with respect to any given representation, by Lemma~\ref{l:support-local}. Moreover,  $\loc{(\Grp{\m{L}})}$ is a convex normal subalgebra of $\Grp{\m{L}}$, by Lemma~\ref{l:local-subuniverse}. 
  As 
  the lattice of convex normal subalgerbas of $\Grp{\m{L}}$ is a chain,  $\loc{(\Grp{\m{L}})}$ is the smallest non-trivial convex normal subalgebra of $\Grp{\m{L}}$; so $\Grp{\m L}$ is subdirectly irreducible. 
  By Proposition~\ref{p:isoCN},
   $\m L$ is subdirectly irreducible, as well.
\end{proof}

\begin{corollary}\label{cor:char-of-fin-SI}
    The proper finitely generated SIs in $\vr(\m{F}_n(\ZZ))$ are exactly the algebras isomorphic to an algebra of the form $\m{H}\overrightarrow{\times} \m{F}_k(\ZZ)$, where $\m{H}$ is a finitely generated totally ordered abelian $\ell$-group and $k\mid n$. Also, these are precisely the  proper finitely generated FSIs in $\vr(\m{F}_n(\ZZ))$. 
\end{corollary}

\begin{proof}
Immediate, by Corollary~\ref{c: FSI VFnZ} together with Theorem~\ref{t:lexprod-fsi-Fn}.
\end{proof}

\begin{corollary}\label{c: subFnZ}
   For all $k,n\in \ZZ^+$, $\m{F}_k(\ZZ) \in \vr(\m{F}_n(\ZZ))$ if and only if $k\mid n$.
\end{corollary}

\begin{proof}
As noted above  $\m{F}_k(\ZZ)$ is a simple algebra, for all $k \in \ZZ^+$, hence also subdirectly irreducible; also it is finitely generated by Proposition~\ref{p: n-cover generates} 
and it is proper for $k>1$ (as there are non-invertible $k$-periodic functions on $\ZZ)$. So, by Corollary~\ref{cor:char-of-fin-SI},
$\m{F}_k(\ZZ) \in \vr(\m{F}_n(\ZZ))$ iff $k \mid n$ (a fact that also holds for $k=1$). 
\end{proof}

\begin{corollary}\label{c: downsetsubvarieties}
    Any  subvariety of \,$\vr(\m{F}_n(\ZZ))$ is of the form 
    \[
    \vr(\m{F}_{k_1}(\ZZ),\dots, \m{F}_{k_m}(\ZZ))
    \]
    for some $k_1 \mid n,\dots, k_m \mid n$ and $m \in \NN$. 
\end{corollary}

\begin{proof}
By Corollary~\ref{c: FSI VFnZ} and Corollary~\ref{cor:char-of-fin-SI}, each finitely generated FSI member of $\vr(\m{F}_n(\ZZ))$ is either a non-trivial totally ordered abelian $\ell$-group
or it is a non-trivial algebra of the form  $\m{H}\overrightarrow{\times} \m{F}_k(\ZZ)$, where $\m{H}$ is a finitely generated totally ordered abelian $\ell$-group and $k\mid n$. 
Now, if $\m A$ is a non-trivial totally ordered abelian $\ell$-group, then it contains a copy of $\ZZ \cong \m{F}_1(\ZZ)$, which generates the variety of abelian $\ell$-groups, so $\vr(\m A)=\vr(\m{F}_1(\ZZ))$. Also, for all $k$, $\m{F}_k(\ZZ) \in \mathbb{S}(\m{H}\overrightarrow{\times} \m{F}_k(\ZZ))$ and $\m{H}\overrightarrow{\times} \m{F}_k(\ZZ) \in \vr(\m{F}_k(\ZZ))$, by Proposition~\ref{p:Fn-closed-lexprod}, thus $\vr(\m{H}\overrightarrow{\times}\m{F}_k(\ZZ))=\vr(\m{F}_k(\ZZ))$.
Hence,  the variety generated by any finitely generated FSI member of $\vr(\m{F}_n(\ZZ))$ is of the form $\vr(\m{F}_k(\ZZ))$ for $k\mid n$.
Therefore, if $\mathcal{V}$ is a subvariety of $\vr(\m{F}_n(\ZZ))$ and $\mathcal{K}$ is the class of its finitely generated FSIs, then $\mathcal{V}=\vr(\mathcal{K})=\bigvee \{\vr(\m L) : \m L \in \mathcal{K}\}=\bigvee \{\vr(\m{F}_k(\ZZ)) : k \in S\}$, for some finite set $S$ of divisors of $n$.
\end{proof}

\begin{remark}
We note that the variety generated by any infinite subset of the set $\{\m{F}_n(\ZZ) : n \in \ZZ^+\}$ is the whole variety $\mathsf{DLP}$, as is implicit in \cite{GG2}. Indeed, Theorem 3.8 of \cite{GG2} states that if an equation $\varepsilon$ fails in $\mathsf{DLP}$ then it fails in $\vr(\m{F}_n(\ZZ))$, where $n:=2^{\lvert\varepsilon\rvert} \lvert\varepsilon\rvert^4$ and $\lvert\varepsilon\rvert$ is the number of symbols in $\varepsilon$; the argument is based on the proof of Theorem 3.6 of \cite{GG2}, from which it follows that actually $\varepsilon$ fails in $\vr(\m{F}_m)$ for all $m \geq n$. As a result, given an infinite subset $K$ of $\{\m{F}_n(\ZZ) : n \in \ZZ^+\}$, if an equation fails in $\mathsf{DLP}$, then (actually if and only if) it fails in some  algebra (actually in infinitely-many algebras)  in $K$.
\end{remark}

The results we have proved allow us to also describe all of the varieties generated by finite subsets of $\{\m{F}_n(\ZZ) : n \in \ZZ^+\}$, as we see below.

We denote by $\m{D} = ( \ZZ^+, \mid )$ the lattice of positive integers ordered by divisibility   and by $\mathcal{O}_\mathrm{fin}(\m{D})$ the lattice of finite downsets of $\m{D}$. In the following for $S\subseteq \ZZ^+$, we will denote by ${\downarrow} S$ the downset $\{k\in \ZZ^+ : k\mid n \text{ for some  } n\in S \}$ of $S$ in $\mathcal{O}_\mathrm{fin}(\m{D})$. 

\begin{theorem}\label{t: ideal}
The subvarieties generated by finite subsets of the set $\{\m{F}_n(\ZZ) : n \in \ZZ^+\}$ form an ideal of the subvariety lattice of $\mathsf{DLP}$. This ideal is isomorphic to  $\mathcal{O}_{\mathrm{fin}}(\m{D})$ via the map $S \mapsto \vr(\{\m{F}_{k}(\ZZ) : k \in S\})$. 
\end{theorem}

\begin{proof}
For every finite subset $S$ of $\ZZ^+$, we set $V(S):=\vr(\{\m{F}_{k}(\ZZ) : k \in S\})$ and $\Lambda:=\{V(S) : S \subseteq \ZZ^+, S$ is finite$\}$; we will show that $S \mapsto V(S)$ from $\mathcal{O}_{\mathrm{fin}}(\m{D})$ to  $\Lambda$ is a lattice isomorphism (from a lattice to a poset). 

Note that for every finite subset $S$ of $\ZZ^+$,  $V(S)=\vr(\{\m{F}_{k}(\ZZ) : k \in S\})=\bigvee \{\vr(\m{F}_{k}(\ZZ)) : k \in S\}$. By a corollary to Jónsson's Lemma the (finitely) subdirectly irreducibles of this join are the union of the (finitely) subdirectly irreducibles of each $\vr(\m{F}_{k}(\ZZ))$, where $k \in S$, i.e., $V(S)_{SI}=\bigcup\{ V(\{k\})_{SI} : k \in S\}$.
By combining the above equation 
with Corollary~\ref{c: subFnZ}, we have that for finite  $S,T\subseteq \ZZ^+$:
\begin{align*}
    V(S) \subseteq V(T) &\iff \vr(\{\m{F}_{k}(\ZZ) : k \in S\}) \subseteq \vr(\{\m{F}_{n}(\ZZ) : n \in T\}) \\
    & \iff \text{for all $k \in S$, $\m{F}_k(\ZZ) \in  \vr(\{\m{F}_{n}(\ZZ) : n \in T\})$}  \\
    &\iff \text{for all $k \in S$ there exists $n \in T$ with $\m{F}_k(\ZZ) \in  \vr(\m{F}_{n}(\ZZ))$} \\
    &\iff \text{for all $k \in S$ there exists $n \in T$ with $k \mid n$} \\
    &\iff S \subseteq {\downarrow} T.
\end{align*}
In particular, $V(S) = V(T)$ iff ${\downarrow} S = {\downarrow} T$; hence $V(S) = V({\downarrow} S)$.

 By the above argument we get that for $S,T \in \mathcal{O}_{\mathrm{fin}}(\m{D})$, $V(S) \subseteq V(T)$ iff $S \subseteq {\downarrow} T$ iff $S \subseteq T$, and also that for all  finite subsets $S$ of $\ZZ^+$, $V(S)=V({\downarrow} S)$. Thus the map $S \mapsto V(S)$ from $\mathcal{O}_{\mathrm{fin}}(\m{D})$ to  $\Lambda$  preserves and reflects order, and is onto; as a result it is an isomorphism of posets that is also a lattice isomorphism, since $\mathcal{O}_{\mathrm{fin}}(\m{D})$ is a lattice.
Furthermore, 
$V(S) \jn V(T)=V(S \cup T)\in \Lambda$.
Therefore, $\Lambda$ is closed under joins; that it is downward closed follows by Corollary~\ref{c: downsetsubvarieties},
so $\Lambda$ is an ideal of the subvariety lattice of $\mathsf{DLP}$.
\end{proof}

\begin{remark}
We note that in the proof of Theorem~\ref{t: ideal} we established that for every finite subset $S$ of $\ZZ^+$,  $\vr(\{\m{F}_{k}(\ZZ) : k \in S\})=\vr(\{\m{F}_{k}(\ZZ) : k \in {\downarrow} S\})= \vr(\{\m{F}_{k}(\ZZ) : k \in \max {\downarrow} S\})$. 
\end{remark}

For $n\geq 1$, let $D_n$ be the set of divisors of $n$ and let $\m{D}_n = ( D_n, \mid )$ be the associated lattice. 
We denote by $\mathcal{O}(\m{D}_n)$ the downset lattice of $\m{D}_n$. 
\begin{corollary}
    For each $n \geq 1$, the subvariety lattice of $\vr(\m{F}_n(\ZZ))$ is isomorphic to  $\mathcal{O}(\m{D}_n)$ via the map  $S \mapsto \vr(\{\m{F}_{k}(\ZZ)) : k \in S\})$, for $S \in \mathcal{O}(\m D_n)$.    
\end{corollary}

\begin{corollary}\label{c: properly n per}
For each $n$, the subvarieties of $\mathsf{DLP}$ that are properly $n$-periodic (not periodic for a smaller period) comprise the interval 
$[\bigvee_{q\in P} \vr(\m{F}_{q}(\ZZ)), \mathsf{LP}_n]$, where $P = \{p_1^{k_1},\dots, p_{l}^{k_l} \}$ is the set of all maximal prime powers that divide $n$, i.e., $p_1,\dots,p_l$ are distinct primes such that $n = p_1^{k_1}\cdots p_l^{k_l}$;
if $n$ is a prime power this interval is $[\vr(\m{F}_n(\ZZ)), \mathsf{LP}_n]$. 
\end{corollary}

\begin{proof}

Note that a variety of the form  $\vr(\m{F}_k(\ZZ))$ is $n$-periodic iff $k \mid n$; and a variety of the form  $V(S):=\vr(\{\m{F}_{k}(\ZZ) : k \in S\})$, where  $S$ is a finite subset of $\ZZ^+$, is $n$-periodic iff $S \subseteq {\downarrow} n$, where the latter is computed in the divisibility lattice. So, asking that $V(S)$ fails to be $m$-periodic for a given $m \mid n$, $m\neq n$, amounts to asking that $S  \nsubseteq {\downarrow} m$.

We first show that that for $S\subseteq {\downarrow}n$, $V(S)$ is properly $n$-periodic if and only if $p_1^{k_1},\dots,p_l^{k_l} \in {\downarrow}S$.

For the right-to-left direction note that if $m\mid n$ with $m\neq n$, then there exists $i \in \{1,\dots,l\}$ such that $p_i^{k_k}\nmid m$, i.e., $S\nsubseteq {\downarrow}m$. Conversely assume that there is an $i\in \{1,\dots,l\}$ such that $p_i^{k_i}\notin {\downarrow}S$. Then, since $S\subseteq {\downarrow}n$, for $m = n/p_i$ we obtain $S\subseteq  {\downarrow}m$, i.e., $V(S)$ is not properly $n$-periodic.

Now let $\mathcal{V}$ be a properly $n$-periodic variety. Then $\mathcal{V}$ contains an $\ell$-pregroup $\m{L}$ of periodicity $n$; let $\m{L} \leq \m{F}_n(\m J \overrightarrow{\times} \ZZ)$ be a representation of $\m L$, where $\m J$ is a chain. By Remark~\ref{r:per-of-alg}, also $\loc{\m{L}}$  
is of periodicity $n$; hence, $\vr(\loc{\m{L}})$ is properly $n$-periodic. Also $\loc{\m{L}} \leq \m{F}_n(\ZZ)^J \in \vr(\m{F}_n(\ZZ))$,  
so by Corollary~\ref{c: downsetsubvarieties}, $\vr(\loc{\m{L}}) = V(S)$ for some finite $S\subseteq \ZZ^+$ with $S\subseteq {\downarrow}n$. Thus, by the above argument, $p_1^{k_1},\dots, p_{l}^{k_l} \in {\downarrow}S$, i.e., $\bigvee_{q\in P} \vr(\m{F}_q(\ZZ)) = V(P) \subseteq V(S) \subseteq \mathcal{V}\subseteq \mathsf{LP}_n$.
\end{proof}

\begin{remark}
For $n=1$, we can see that Corollary~\ref{c: properly n per} specializes to the interval $[\vr(\m{F}_1(\ZZ)), \mathsf{LP}_1]=[\vr(\ZZ), \mathsf{LG}]$, which together with the trivial variety constitutes the subvariety lattice of $\ell$-groups. 
\end{remark}

We now describe how to obtain finite axiomatizations for all of the varieties of the form  $\vr(\{\m{F}_{k}(\ZZ) : k \in S\})$, where $S$ is a finite subset of $\ZZ^+$, i.e., of all of the varieties mentioned in  Theorem~\ref{t: ideal}.

\begin{theorem}\label{t: expandedAxiom}
If  $\mathcal{V}_1, \ldots, \mathcal{V}_m$ are  subvarieties of $\mathsf{DLP}$ axiomatized by the equations $1 \leq  t_1 , \ldots,  1 \leq  t_m$, respectively, then their join is axiomatized by  the equation $1\leq  x_1t_1x_1^\ell \jn \cdots \jn x_mt_mx_m^\ell$, where $x_1, \ldots, x_m$ are variables not occurring in $t_1, \ldots, t_m$. 
\end{theorem}

\begin{proof}
By Corollary~4.3 of \cite{Ga04}, the join of the varieties is axiomatized by the set of all equations of the form  $1 \approx \rho_1(t_1 \mt 1) \jn \cdots \jn \rho_m(t_m \mt 1)$, where the $\rho_i$'s range over all possible iterated conjugates. By Remark~\ref{r: conjLP}, each iterated conjugate reduces to a single right conjugate, so the join of the varieties is axiomatized by the single equation $1 \approx [x_1(t_1 \mt 1)x_1^\ell \mt 1] \jn \cdots \jn [x_m(t_m \mt 1)x_m^\ell \mt 1]$, where $x_1, \ldots, x_m$ are distinct variables not occurring in $t_1, \ldots, t_m$. By Lemma~\ref{l: InRL}, the equation is equivalent to  $1\approx [x_1t_1x_1^\ell \mt x_1x_1^\ell \mt 1] \jn \cdots \jn [x_mt_mx_m^\ell \mt x_mx_m^\ell \mt 1]$. Since $1 \leq xx^\ell$ in $\ell$-pregroups, the equation can be further written in the form
$1 \approx  [x_1t_1x_1^\ell \mt 1] \jn \cdots \jn [x_mt_mx_m^\ell  \mt 1]$. In distributive $\ell$-pregroups this simplifies to
$1\leq  x_1t_1x_1^\ell \jn \cdots \jn x_mt_mx_m^\ell$.
\end{proof}

\begin{remark} 
Following \cite{Ga04}, we note that the demand that a finitely axiomatized variety of $\ell$-pregroups is axiomatized by a single equation of the form $1 \leq  r$ is not a real restriction and it always holds. Indeed, first note that every equation $s=t$ is equivalent to the conjunction of $s \leq t$ and $t \leq s$, hence also to the conjunction of $1 \leq s^\ell t$ and $1 \leq t^\ell s$, by residuation. Therefore, every finite equational axiomatization is equivalent to a conjunction of inequations of the form $1 \leq r_1, \ldots, 1 \leq r_m$. This conjunction is further equivalent to the inequality $1 \leq r_1 \mt\cdots \mt r_m$. 

In the setting of distributive $\ell$-pregroups we can further prove that  the term $r$ can be taken to be a join of $\{1, \cdot, {}^\ell, {}^r\}$-terms. Indeed, since multiplication distributes over joins and meets by Lemma~\ref{l: InRL}, $r$ is equivalent to a lattice term over $\{1, \cdot, {}^\ell, {}^r\}$-terms. By lattice distributivity, the lattice term can be written as a meet of joins, so $1 \leq r$ is equivalent to a conjunction of equations of the form $1 \leq r_1, \ldots, 1 \leq r_m$, where each $r_i$ is a join of $\{1, \cdot, {}^\ell, {}^r\}$-terms; we can further assume that the variable sets in the terms $r_i$ are disjoint. Then the conjunction is equivalent to the inequality $1 \leq r_1 r_2 \cdots  r_m$ (where the converse direction is obtained by setting all variables equal to $1$ except the ones in $r_i$, for one $i$ at a time). By distributivity of multiplication over join, we get an inequality of the form $1 \leq r'$, where $r'$ is a join of $\{1, \cdot, {}^\ell, {}^r\}$-terms.
\end{remark}

\begin{corollary} 
 For every  finite subset $S$ of $\ZZ^+$, the variety $\vr(\{\m{F}_{k}(\ZZ) : k \in S\})$ is axiomatized by
 the single $4\lvert S\rvert$-variable equation
 \[
1 \leq \bigvee_{k \in S} w_k[ z_k^{\ell^{2k}}z_k^\ell  \mt  \g_k(y_k)^k x_k \g_k(y_k)^{-k}x_k^\ell]w_k^\ell.
\]
\end{corollary}

\begin{proof}
By Theorem~\ref{t: axiomatization} and Remark~\ref{r: half}, for each $k \in S$, the variety $\vr(\m{F}_{k}(\ZZ)$ is axiomatized by the equations $1 \leq z_k^{\ell^{2k}}z_k^\ell$ and $1 \leq\g_k(y_k)^k x_k \g_k(y_k)^{-k}x_k^\ell$.

The variety  $\vr(\{\m{F}_{k}(\ZZ) : k \in S\})$ is equal to the join of the varieties $\vr(\m{F}_{k}(\ZZ)$, where $k \in S$; by Theorem~\ref{t: expandedAxiom} it is axiomatized by the   $4\lvert S\rvert$-variable equation $1 \leq \bigvee_{k \in S} w_k[ z_k^{\ell^{2k}}z_k^\ell  \mt  \g_k(y_k)^k x_k \g_k(y_k)^{-k}x_k^\ell]w_k^\ell$.
\end{proof}

\section*{Acknowledgments}
This project has received funding from the European Union’s Horizon 2020 research and innovation programme under the Marie Skłodowska-Curie grant agreement No 101007627. The second author was supported by the Swiss National Science Foundation (SNSF), grant no. 200021\textunderscore215157.


\begin{thebibliography}{00}


\bibitem{AF}
 M. Anderson, T. Feil. Lattice-ordered groups. An introduction. Reidel Texts in the Mathematical Sciences. D. Reidel Publishing Co., Dordrecht, 1988.

\bibitem{Ba}
M. Barr. On subgroups of the Lambek pregroup.
Theory Appl. Categ. 12 (2004), No. 8, 262–269.


\bibitem{Bo} M. Botur, J. Kühr, L. Liu, and C. Tsinakis, The Conrad program: from $\ell$-groups to algebras of logic, Journal of Algebra 450 (2016), 173--203.


\bibitem{Bu}
W. Buszkowski.
Pregroups: models and grammars.  Relational methods in computer science, 35–49,
Lecture Notes in Comput. Sci., 2561, Springer, Berlin, 2002.

\bibitem{CD}
J. Czelakowski and W. Dziobiak. Congruence distributive quasivarieties whose finitely subdirectly irreducible members form a universal class. Algebra Univers. 27 (1990) 128–149.

\bibitem{Da}
M. Darnel. Theory of lattice-ordered groups. Monographs and Textbooks in Pure and Applied Mathematics, 187. Marcel Dekker, Inc., New York, 1995.

\bibitem{Da2}
M. R. Darnel, W. C. Holland, and H. Pajoohesh. Generalized commutativity of lattice-ordered groups. {II}. Algebra Universalis 75 (2016), no. 1, 51--59.

\bibitem{GH}
N. Galatos and R. Hor{\v{c}}{\'\i}k. Cayley's and Holland's theorems for idempotent semirings and their applications to residuated lattices. Semigroup Forum 87 (2013), no. 3, 569–589. 

\bibitem{Ga04} N. Galatos, Equational bases for joins of residuated-lattice varieties. Studia Logica 76(2) (2004), 227-240. 

\bibitem{GJ}  N. Galatos, P. Jipsen, Periodic lattice-ordered pregroups are distributive. Algebra Universalis 68 (2012), no. 1-2, 145–150.

%\bibitem{GJframes} N. Galatos and P. Jipsen. Residuated frames with applications to decidability. Transactions of the AMS 365(3) (2013), 1219--1249

%\bibitem{GJKP}
%N. Galatos, P. Jipsen, M. Kinyon, and A. P{\v{r}}enosil. Lattice-ordered pregroups are semi-distributive. Algebra Universalis 82 (2021), no. 1, Paper No. 16, 6 pp.

\bibitem{GJKO}
N. Galatos, P. Jipsen, T. Kowalski, and H. Ono. Residuated lattices: an algebraic glimpse at substructural logics. Studies in Logic and the Foundations of Mathematics, 151. Elsevier B. V., Amsterdam, 2007. xxii+509 pp.

\bibitem{GG1}
N. Galatos and I. Gallardo. Distributive $\ell$-pregroups: generation and decidability. Journal of Algebra 648 (2024), 9--35. 

\bibitem{GG2}
N. Galatos and I. Gallardo. Generation and decidability for periodic $\ell$-pregroups. Journal of Algebra 669 (2025), 301--340.

 %\bibitem{GHo}
 %A. M. W. Glass and W. C. Holland (edt). Lattice-ordered groups. Advances and techniques.  Mathematics and its Applications, 48. Kluwer Academic Publishers Group, Dordrecht, 1989. 

\bibitem{Han}
M.~E. Hansen, The lex property of varieties of lattice ordered groups, Algebra Universalis 28 (1991), no.~4, 535--548.

\bibitem{Ho-em}%embedding
W. C. Holland. The lattice-ordered groups of automorphisms of an ordered
set. Michigan Mathematical Journal, 10 (4), 399–408, 1963.

\bibitem{KM}
V. M. Kopytov, N. Ya. Medvedev. The theory of lattice-ordered groups. Mathematics and its Applications, 307. Kluwer Academic Publishers Group, Dordrecht, 1994.

\bibitem{La}
J. Lambek.
Pregroup grammars and Chomsky's earliest examples. 
J. Log. Lang. Inf. 17 (2008), no. 2, 141–160.

\bibitem{Lo} K. Lorenz, {Ü}ber {S}trukturverbände von {V}erbandsgruppen, Acta Math. Acad. Sci. Hungar. 13 (1962), 55--67.

\end{thebibliography}
\end{document}